\documentclass[11pt,a4paper]{article}

\usepackage[left=2.3cm, top=2.3cm,bottom=2.3cm,right=2.3cm]{geometry}

\usepackage{mathtools,amssymb,amsthm,mathrsfs,calc,graphicx,stmaryrd,xcolor,dsfont,tikz,pgfplots,bbm,float,array,epsfig,hyperref,color}
\usepackage[numbers,square]{natbib}
\usepackage[british]{babel}
\usepackage{amsfonts}              
\usepackage{subcaption}

\numberwithin{equation}{section}
\numberwithin{figure}{section}

\newtheorem {theorem}{Theorem}[section]
\newtheorem {proposition}[theorem]{Proposition}
\newtheorem {lemma}[theorem]{Lemma}
\newtheorem {corollary}[theorem]{Corollary}

{\theoremstyle{definition}
\newtheorem{definition}[theorem]{Definition}
\newtheorem*{convention*}{Convention}

\newtheorem {example}[theorem]{Example}
\newtheorem {remark}[theorem]{Remark}
}

{
\theoremstyle{remark}
}

\newcommand{\eqdistr}{\stackrel{d}{=}}
\newcommand{\Vol}{\operatorname{Vol}}

\newcommand{\Cov}{\operatorname{Cov}}
\newcommand{\const}{\operatorname{const}}

\renewcommand{\Re}{\operatorname{Re}}  
\renewcommand{\Im}{\operatorname{Im}}  

\def\ba{\begin{array}}
\def\ea{\end{array}}
\def\bea{\begin{eqnarray} \label}
\def\eea{\end{eqnarray}}
\def\be{\begin{equation} \label}
\def\ee{\end{equation}}
\def\bit{\begin{itemize}}
\def\eit{\end{itemize}}
\def\ben{\begin{enumerate}}
\def\een{\end{enumerate}}

\def\lan{\langle}
\def\ran{\rangle}

\def\BB{\mathbb{B}}
\def\CC{\mathbb{C}}
\def\E{\mathbb{E}}

\def\N{\mathbb{N}}
\def\P{\mathbb{P}}

\def\R{\mathbb{R}}
\def\RRd1{\mathbb{R}^{d+1}}

\def\bS{\mathbb{S}}

\def\Var{\mathbf{V}\textup{ar}\,}
\def\Cov{\mathbf{C}\textup{ov}}

\def\cF{\mathcal{F}}

\newcommand{\eee}{{\rm e}}

\newcommand{\ind}{\mathbbm{1}}
\newcommand{\eps}{\varepsilon}
\newcommand{\pos}{\mathop{\mathrm{pos}}\nolimits}

\newcommand{\aff}{\mathop{\mathrm{aff}}\nolimits}
\newcommand{\lin}{\mathop{\mathrm{lin}}\nolimits}
\newcommand{\conv}{\mathop{\mathrm{conv}}\nolimits}

\newcommand{\dd}{{\rm d}}
\newcommand{\bsl}{\backslash}
\DeclareMathOperator{\relint}{relint}
\DeclareMathOperator{\relbd}{relbd}

\def\bg{\mathbf{g}}
\newcommand{\sgn}{\mathop{\mathrm{sgn}}\nolimits}

\makeatletter
\@namedef{subjclassname@2020}{\textup{2020} Mathematics Subject Classification}
\makeatother

\begin{document}

\title{\bfseries Angles of orthocentric simplices}

\author{Zakhar Kabluchko and Philipp Schange}

\date{}

\maketitle

\begin{abstract}
A \( d \)-dimensional simplex in Euclidean space is called orthocentric if all of its altitudes intersect at a single point, referred to as the orthocenter. We explicitly compute the internal and external angles at all faces of an orthocentric simplex. To this end, we introduce a parametric family of polyhedral cones, called orthocentric cones, and derive formulas for their angles and, more generally, for their conic intrinsic volumes.

We characterize the tangent and normal cones of orthocentric simplices in terms of orthocentric cones with explicit parameters. Depending on whether the orthocenter lies inside the simplex, on its boundary, or outside, the simplex is classified as acute, rectangular, or obtuse, respectively. The solid angle formulas differ in these three cases.

As a probabilistic application of the angle formulas, we explicitly compute the expected number of \( k \)-dimensional faces and the expected volume of the random polytope \([g_1/\tau_1, \ldots, g_n/\tau_n]\), where \( g_1, \ldots, g_n \) are independent standard Gaussian vectors in \( \mathbb{R}^d \), and \( \tau_1, \ldots, \tau_n > 0 \) are constants.

\noindent
\bigskip
\\
{\bf Keywords}. Orthocentric simplex, orthocentric cone, solid angle, conic intrinsic volume, regular simplex, spherical convexity, tangent cone, normal cone, Gaussian polytope, analytic continuation, standard normal distribution function, error function. \\
{\bf MSC 2020}. Primary: 60D05, 52A22; Secondary: 52A55, 52B11, 52A39, 52B05, 62H10.  
\end{abstract}

\tableofcontents

\section{Introduction and main results}
\subsection{Orthocentric simplices}
The altitudes of any triangle in the plane intersect at a single point, known as the \emph{orthocenter} of the triangle. In contrast, this property does not hold for tetrahedra in three-dimensional space~\cite{havlicek_weiss}. A $d$-dimensional simplex (with $d \geq 2$) is called \emph{orthocentric} if all of its altitudes intersect in a single point. While many classical results in elementary triangle geometry—such as the Euler line theorem and the Feuerbach nine-point circle theorem—admit natural generalizations to orthocentric simplices~\cite{mehmke,egervary_on_ortho_simpl,egervary_feuerbach}, there are also numerous additional reasons to regard orthocentric simplices as the ``true'' $d$-dimensional analogues of triangles~\cite{hajja_martini_orthocentric_true_generalization}. For a comprehensive overview of the properties of orthocentric simplices, along with references to the relevant literature, we refer the reader to the work of~\citet{Edmonds2005}.

Depending on the position of the orthocenter, orthocentric simplices fall into three subclasses:
\begin{itemize}
    \item \emph{Acute simplices} are those for which the orthocenter lies in the relative interior of the simplex. Every $d$-dimensional acute orthocentric simplex is isometric to
    \begin{equation}\label{eq:canonical_acute_orthocentric_simpl_intro}
    \left[\frac{e_0}{\tau_0},\frac{e_1}{\tau_1},\ldots, \frac{e_d}{\tau_d}\right],
    \quad \text{with orthocenter at } \frac{\tau_0 e_0 + \ldots + \tau_d e_d}{\tau_0^2 + \ldots + \tau_d^2},
    \end{equation}
    for suitable parameters $\tau_0,\tau_1,\ldots,\tau_d > 0$. Here, $e_0,\ldots,e_d$ denote the standard orthonormal basis of $\R^{d+1}$, and $[\ldots]$ denotes the convex hull.

    \item \emph{Obtuse simplices} are those in which the orthocenter lies outside the simplex. Every obtuse orthocentric simplex is isometric to
    \begin{equation}\label{eq:canonical_obtuse_orthocentric_simpl_intro}
    \left[\frac{\tau_0 e_0 + \ldots + \tau_d e_d}{\tau_0^2 + \ldots + \tau_d^2},\frac{e_1}{\tau_1}, \ldots, \frac{e_d}{\tau_d}\right],
    \quad \text{with orthocenter at } \frac{e_0}{\tau_0},
    \end{equation}
    for suitable parameters $\tau_0,\tau_1,\ldots,\tau_d > 0$.

    \item \emph{Rectangular simplices} are those for which the orthocenter lies on the relative boundary of the simplex—then the orthocenter coincides with one of the vertices. Every rectangular orthocentric simplex is isometric to
    \begin{equation}\label{eq:canonical_rectangular_orthocentric_simpl_intro}
    \left[0,\frac{e_1}{\tau_1},\ldots, \frac{e_d}{\tau_d}\right],
    \quad \text{with orthocenter at } 0,
    \end{equation}
    for suitable parameters $\tau_1,\ldots,\tau_d > 0$. These simplices can be interpreted as limiting cases of both acute and obtuse simplices~\eqref{eq:canonical_acute_orthocentric_simpl_intro} and~\eqref{eq:canonical_obtuse_orthocentric_simpl_intro}  as $\tau_0 \to \infty$.
\end{itemize}

A more detailed discussion of these and related facts will be given in Section~\ref{sec:orthocentric_simpl}.
Orthocentric simplices that are not rectangular are referred to as \emph{oblique}.

\begin{remark}\label{rem:egervary_orthocentric_collection}
The following observation, due to~\citet{egervary_feuerbach}, sheds light on the structural similarity between~\eqref{eq:canonical_acute_orthocentric_simpl_intro} and~\eqref{eq:canonical_obtuse_orthocentric_simpl_intro}. Let $[v_0, v_1, \ldots, v_d]$ be an oblique orthocentric simplex with orthocenter $w$. Then, any point in the set $\{v_0, \ldots, v_d, w\}$ is the orthocenter of the orthocentric simplex formed by the remaining $d+1$ points.
\end{remark}

\subsection{Polyhedral cones and their angles}
The aim of the present paper is to compute the internal and external angles of orthocentric simplices, along with several closely related quantities. Let us first recall some facts on polyhedral cones and their solid angles.
A set \( C \subseteq \R^d \) is called a \emph{polyhedral cone} if it is the solution set of a finite system of linear homogeneous inequalities; that is, there exist vectors \( w_1, \ldots, w_m \in \R^d \) such that
\[
C = \{ x \in \R^d : \langle w_1, x \rangle \leq 0, \ldots, \langle w_m, x \rangle \leq 0 \}.
\]
An equivalent definition is that \( C \subseteq \R^d \) is a polyhedral cone if it can be represented as the \emph{positive hull} of finitely many vectors \( v_1, \ldots, v_n \in \R^d \); that is,
\[
C = \pos(v_1, \ldots, v_n) = \left\{ \lambda_1 v_1 + \cdots + \lambda_n v_n : \lambda_1\geq 0, \ldots, \lambda_n \geq 0 \right\}.
\]
Assuming that a polyhedral cone \( C \subseteq \R^d \) has nonempty interior, its \emph{solid angle} is defined by
\[
\alpha(C) = \frac{\lambda_d(C \cap \BB^d)}{\lambda_d(\BB^d)},
\]
where \( \BB^d = \{ x \in \R^d : \|x\| \leq 1 \} \) is the closed unit ball in \( \R^d \), and \( \lambda_d \) denotes the \( d \)-dimensional Lebesgue measure. Note that $\alpha(C) = 1$ if and only if $C$ is a linear subspace;  it can be shown that otherwise we have  $0 < \alpha(C) \leq 1/2$.  For example, the solid angle of the nonnegative orthant $[0,\infty)^d$ is $1/2^d$.

A \emph{polyhedron} is an intersection of finitely many half-spaces. Given a polyhedron \( P \) and one of its faces \( F \), the \emph{tangent cone} \( T(F, P) \) is defined as the set of all directions pointing from any point \( x \) in the relative interior of \( F \) into \( P \). The \emph{normal cone} \( N(F, P) \) is defined as the polar cone of \( T(F, P) \); see Section~\ref{sec:facts_convex_geometry} for precise definitions. The solid angle of \( T(F, P) \) (respectively, \( N(F, P) \)) is called the \emph{internal} (respectively, \emph{external}) angle of \( P \) at \( F \) and is denoted by \( \beta(F, P) := \alpha(T(F, P)) \) (respectively, \( \gamma(F, P) := \alpha(N(F, P)) \)).

It is generally known that computing solid angles of polyhedral cones is a difficult task, even in the special case when $C= \pos (v_1,\ldots, v_d)$ is generated by $d$ linearly independent vectors $v_1,\ldots, v_d \in \R^d$.
In dimension $d=3$, there exist explicit formulas for $\alpha(C)$ going back to Euler~\cite{euler_mensura} and Lagrange~\cite{lagrange_1799solutions}; see, e.g., \cite{eriksson_solid_angles}. For general dimension $d$, there is an expansion of $\alpha(C)$ into an infinite hypergeometric series in the $\binom d2$ variables $\langle v_i/\|v_i\|, v_j/\|v_j\|\rangle$, $1\leq i < j \leq d$, due to~\citet{aomoto}. \citet{ribando} rediscovered Aomoto's series and gave a simplified proof.  A discussion of these results can be found in~\cite[Section~3]{beck_robins_sam} and~\cite{SatoSphericalOrthoc}; a numerical method to compute solid angles based on Aomoto's series was developed recently by  \citet{fitisone_zhou_solid_angle}.

The problem of computing the solid angle of a polyhedral cone $C\subseteq \R^d$ is equivalent to the problem of computing the spherical volume of the spherical polytope $C \cap \bS^{d-1}$, where $\bS^{d-1}$ is the unit sphere in $\R^d$. The latter  problem, together with its analogue in the hyperbolic space, received much attention starting with the work of Lobachevsky, Bolyai and Schl\"afli~\cite{schlafli_multiple}; see e.g.\ the review papers of~\citet{milnor_hyperbolic_volumes,milnor_schlafli}, \citet{vinberg_volumes_polyhedra_russ_math_surv}, \citet{abrosimov_mednykh_volumes_polytopes_spaces_const_curv,abrosimov_mednykh} and the books by~\citet[Chapter~7]{alekseevski_vinberg_solodovnikov_book} and~\citet{boehm_hertel_book}. Another closely related problem is the computation of the so-called \emph{orthant probabilities}, that is probabilities of the form $\P[\eta_1 \geq 0,\ldots, \eta_d\geq 0]$, where $(\eta_1,\ldots, \eta_d)$ is a multivariate Gaussian random vector; see the books of \citet[Chapter~5]{kotz_balakrishnan_johnson_book_cont_multivar_vol_1} and~\citet{genz_bretz_book_multivar_normal_t} for overviews of existing results.

\subsection{Orthocentric cones}
Despite considerable interest in the computation of solid angles—arising from areas as diverse as differential geometry and statistics—the number of explicit formulas for the solid angles of polyhedral cones valid in arbitrary dimensions remains very limited. In the definition below, we introduce a class of polyhedral cones that arise naturally as tangent and normal cones of orthocentric simplices and for which explicit angle computations can be carried out.
\begin{definition}[Orthocentric cones]\label{def:ortho_cone_intro}
Let $d\in \N$,  $\lambda_0,\lambda_1,\ldots,\lambda_d \in  \R\backslash\{0\}$ and $\eps_1, \ldots, \eps_d \in \{\pm 1\}$. An \emph{orthocentric cone} is defined by
$$
C_d(\lambda_0; \lambda_1,\ldots,\lambda_d; \eps_1, \ldots, \eps_d) := \pos (\eps_1 v_1,\ldots,\eps_d v_d),
$$
where $v_1,\ldots,v_d$ are linearly independent vectors in $\R^N$ (with $N\geq d$) satisfying
\begin{equation}\label{eq:v_i_scalar_prod_intro}
\langle v_i, v_j\rangle
=
\frac{1}{\lambda_0} + \frac{\delta_{ij}}{\lambda_i}
=
\begin{cases}
\frac 1 {\lambda_0}, &\text{ if } i\neq j,\\
\frac 1 {\lambda_i} + \frac 1 {\lambda_0}, & \text{ if } i=j,
\end{cases}
\qquad
i,j\in \{1,\ldots,d\}.
\end{equation}
\end{definition}
It will be shown in Lemma~\ref{lemma:condition_for_ex_of_orthocentric_cone} below  that linearly independent vectors $v_1,\ldots, v_d$ satisfying~\eqref{eq:v_i_scalar_prod_intro} exist if and only if one of the following cases occurs:
\begin{enumerate}
\item[(A)] All numbers $\lambda_0, \lambda_1, \ldots, \lambda_d$ are positive.
\item[(B)] One of $\lambda_0, \lambda_1, \ldots,\lambda_d$ is negative, all others are positive, and $\lambda_0 + \lambda_1 + \ldots + \lambda_d < 0$.
\end{enumerate}

We shall show that the class of orthocentric cones is closed under taking polar cones (Theorem~\ref{theo:orthocentric_cone_dual}). By definition, it is closed under taking faces.  Furthermore, the normal and tangent cones at faces of any dimension of an orthocentric cone are again orthocentric after factoring out the lineality space (Theorem~\ref{theo:orthocentr_cones_tangent_normal}). Normal and tangent cones of orthocentric simplices are also orthocentric cones (Theorem~\ref{theo:cones_of_orthoc_simpl_nondegenerate_case}). This last fact  justifies the terminology ``orthocentric'' for these cones; another motivation is that the intersection of an orthocentric cone with the unit sphere forms a spherical simplex whose altitudes meet at a single point (although we shall not make use of this property here).

To express the solid angle of $C_d(\lambda_0; \lambda_1, \ldots, \lambda_d; \varepsilon_1, \ldots, \varepsilon_d)$, we introduce a function $\bg_d$, defined as an orthant probability for a Gaussian random vector whose covariance matrix is the Gram matrix given in~\eqref{eq:v_i_scalar_prod_intro}.

\begin{definition}[The function $\bg_d$]\label{def:g_d_function_intro}
Let $d\in \N$ and let $(\eta_1, \ldots, \eta_d)$ be a multivariate Gaussian random vector with mean $0$ and covariance matrix
\begin{equation}\label{eq:bg_d_cov_matrix_prod_intro}
\E [\eta_i \eta_j]
=
\frac{1}{\lambda_0} + \frac{\delta_{ij}}{\lambda_i}
=
\begin{cases}
\frac 1 {\lambda_0}, &\text{ if } i\neq j,\\
\frac 1 {\lambda_i} + \frac 1 {\lambda_0}, & \text{ if } i=j,
\end{cases}
\qquad
i,j\in \{1,\ldots,d\},
\end{equation}
where $\lambda_0, \lambda_1,\ldots, \lambda_d \in \R\backslash\{0\}$ satisfy (A) or (B). Then, for $\eps_1,\ldots, \eps_d\in \{\pm 1\}$ we define
$$
\bg_d(\lambda_0; \lambda_1, \ldots, \lambda_d;\eps_1, \ldots, \eps_d)
:=
\P\left[ \eps_1\eta_1\leq 0, \ldots, \eps_d\eta_d\leq 0 \right].
$$
By convention, for $d=0$ we put   $\bg_0(\lambda_0) = 1$ for all $\lambda_0 \in \R$.
\end{definition}

\begin{theorem}
For all $\lambda_0,\ldots, \lambda_d\in \R\backslash\{0\}$ that satisfy (A) or (B) and all $\eps_1,\ldots, \eps_d\in \{\pm 1\}$, the  solid angle of the cone $C:= C_d(\lambda_0; \lambda_1,\ldots,\lambda_d; \eps_1, \ldots, \eps_d)$ is given by
$$
\alpha(C)
=
\bg_d(-\lambda_0-\lambda_1-\ldots -\lambda_d; \lambda_1,\ldots, \lambda_d; \eps_1 \sgn(\lambda_1), \ldots, \eps_d \sgn(\lambda_d)).
$$
\end{theorem}
The following theorem—one of our main results—provides an explicit formula for the function \( \bg_d \), and consequently, for the angles of orthocentric cones. To state the result, we require the analytic continuation of the cumulative distribution function of the standard normal distribution, which we denote by
\begin{equation}\label{eq:compl_std_normal_distr}
	\Phi: \CC \rightarrow \CC, \quad \Phi(z) = \frac{1}{2}+\frac{1}{\sqrt{2 \pi}} \int_{0}^{z} \eee^{-x^2/2} \dd x.
\end{equation}
The integral is taken along any contour connecting $0$ to $z$. Let $i= \sqrt{-1}$.  For $x\in \R$, it holds that
$$
\Phi(i x) = \frac{1}{2}+\frac{i}{\sqrt{2 \pi}} \int_0^x \eee^{t^2/2} \dd t.
$$
By expanding \( \eee^{-x^2/2} \) in~\eqref{eq:compl_std_normal_distr} into its Taylor series, we obtain
\begin{equation}\label{eq:taylor_series_Phi}
\Phi(z) = \frac{1}{2}+\frac{1}{\sqrt{2\pi}} \sum_{n=0}^{\infty} \frac{(-1)^n}{(2n+1)2^n n!}z^{2n+1}, \qquad z\in \CC.
\end{equation}
\begin{theorem}[Explicit formula for the function $\bg_d$]\label{theo:explicit_g_d_formula_intro}
~~
\begin{itemize}
\item[(i)]
Let $d \geq 0$ and $\lambda_1, \ldots, \lambda_d > 0$, and suppose that either $\lambda_0 > 0$ or $\lambda_0 < -(\lambda_1 + \ldots + \lambda_d)$. Then, for all $\eps_1, \ldots, \eps_d \in \{\pm 1\}$, we have
\begin{align}\label{eq:g_d_formula_intro_positive}
\bg_d(\lambda_0; \lambda_1, \ldots, \lambda_d; \eps_1, \ldots, \eps_d)
&=
\frac{1}{\sqrt{2 \pi}} \int_{-\infty}^{\infty} \prod_{j=1}^d \Phi \left(\eps_j \sqrt{\frac{\lambda_j}{\lambda_0}}x \right) \eee^{-x^2/2} \dd x.
\end{align}
For $\lambda_0<0$, we use the convention $\sqrt{\lambda_0} = i \sqrt{-\lambda_0}$.
\item[(ii)]
Let $d \in \N$ and $k \in \{1, \ldots, d\}$. Suppose that $\lambda_0, \ldots, \lambda_d$  satisfy $\lambda_i > 0$ for all $i \in \{0, \ldots, d\}\backslash\{k\}$, and $\lambda_k < 0$, with the additional condition that $\lambda_0 + \ldots + \lambda_d < 0$. Then, for all $\eps_1, \ldots, \eps_d \in \{\pm 1\}$, we have
\begin{multline*}
\bg_d(\lambda_0; \lambda_1,\ldots, \lambda_d;\eps_1, \ldots, \eps_d)
=
\sqrt{\frac{\lambda_0}{2\pi}} \Biggl(
\int_0^{\infty} \eee^{-\frac{\lambda_0 y^2}{2}} \prod_{\substack{j=1\\j\neq k}}^d \Phi \left(\eps_k \eps_j \sqrt{\lambda_j} y \right) \dd y
\\
+2 \int_0^{\infty} \eee^{\frac{\lambda_0 y^2}{2}} \Phi\left(-\sqrt{-\lambda_k} y \right) \Im \Bigg( \prod_{\substack{j=1\\j\neq k}}^d \Phi \left(i \eps_k \eps_j \sqrt{\lambda_j} y \Bigg) \right) \dd y
\Biggr).
\end{multline*}
\end{itemize}
\end{theorem}
The proof of Theorem~\ref{theo:explicit_g_d_formula_intro} will be presented in Section~\ref{sec:explicit_formula_g_d}; see Theorem~\ref{theo:formula_for_bg_d(lambda_i,eps_i)} for Part~(i) and Theorem~\ref{theo:formula_for_bg_d(lambda_i,eps_i)_negative} for Part~(ii).

\subsection{Angles of orthocentric simplices}
We are now ready to state the results concerning the angles of orthocentric simplices. Recall that $\beta(F, S)$ and $\gamma(F, S)$ denote the internal and external angles of a simplex $S$ at a face $F$.

\begin{theorem}[Angles of acute orthocentric simplices]\label{theo:orthocentr_simpl_relint_tangent_normal_angles_intro}
Let $d\geq 2$ and fix some $\tau_0,\ldots,\tau_d>0$. Consider the simplex $S:= [e_0/\tau_0,\ldots, e_d/\tau_d] \subseteq \R^{d+1}$ and let $k\in \{0,\dots, d\}$. The internal and external angles of $S$ at its $k$-dimensional face $F:= [e_0/\tau_0,\ldots,e_k/\tau_k]$ are given by
\begin{align*}
\beta(F,S)
&=
\bg_{d-k}(-\tau_0^2-\ldots-\tau_d^2;\tau_{k+1}^2,\dots,\tau_d^2; 1, \ldots, 1),
\\
\gamma(F,S)
&=
\bg_{d-k}(\tau_0^2+\ldots+\tau_k^2;\tau_{k+1}^2,\dots,\tau_d^2; 1,\ldots ,1).
\end{align*}
Here, for $\ell\geq 0$, $\lambda_1, \ldots, \lambda_\ell>0$ and $\lambda_0\in \R$ satisfying either $\lambda_0>0$ or $\lambda_0<-\lambda_1-\ldots -\lambda_\ell$,
\begin{align}\label{eq:funct_g_formula_all_positive_intro}
\bg_\ell(\lambda_0; \lambda_1, \ldots, \lambda_\ell; 1, \ldots, 1)
&=
\frac{1}{\sqrt{2 \pi}} \int_{-\infty}^{\infty} \prod_{j=1}^\ell \Phi \left(\sqrt{\frac{\lambda_j}{\lambda_0}}x \right) \eee^{-x^2/2} \dd x.
\end{align}
\end{theorem}

If $\tau_0 = \ldots = \tau_d > 0$, the simplex $S$ becomes regular, and Theorem~\ref{theo:orthocentr_simpl_relint_tangent_normal_angles_intro} recovers known formulas for the internal and external angles of regular simplices; see~\cite[Section~4]{rogers}, \cite[Lemma~4]{vershik_sporyshev_asymptotic_faces_random_polyhedra1992}, \cite[Proposition~1.2]{kabluchko_zaporozhets_absorption}, \cite[pp.~408--409]{henk_richter_gebbert_polytopes_review}, \cite{kabluchko_zaporozhets_expected_volumes_18}. Closely related results are stated in terms of orthant probabilities in~\cite{steck_orthant_equicorrelated} and~\cite{steck_owen_equicorr_multivar_normal}. Analogous formulas for the internal and external angles of regular crosspolytopes can be found in~\cite[Theorem~2.1]{betke_henk}, \cite[Lemma~3.1]{henk_cifre_notes_roots_steiner}, \cite[Lemma~2.2]{henk_cifre_intrinsic_successive_radii}, \cite{boeroeczky_henk}, and~\cite{kabluchko_seidel_convex_cones_spanned_reg_poly}.

Motivated by the Youden problem concerning the position of the empirical mean of an i.i.d.\ Gaussian sample among the order statistics, \citet[Section~2.B]{kuchelmeister_phd} (see also~\cite{kuchelmeister_probability_linear_separability_intrinsic}) recently computed certain orthant probabilities that are closely related to~\eqref{eq:g_d_formula_intro_positive} and  the angles appearing in Theorem~\ref{theo:orthocentr_simpl_relint_tangent_normal_angles_intro}.
We shall comment on this connection in Section~\ref{subsec:preparation_orthant_probab_for_orthocentric_simpl}.

We now turn to obtuse simplices. These give rise to two distinct types of angles, for neither of which, to the best of our knowledge, related results are available in the existing literature.

\begin{theorem}[Angles of obtuse orthocentric simplices]\label{theo:cones_angles_orthoc_simplices_outside_intro}
For $d\geq 2$, $\tau_0, \tau_1, \ldots, \tau_d >0$, consider the simplex
$$
S = \left[\frac{\tau_0 e_0 + \ldots + \tau_d e_d}{\tau_0^2 + \ldots + \tau_d^2},\frac{e_1}{\tau_1}, \ldots, \frac{e_d}{\tau_d}\right].
$$
Let $k \in \{0,\ldots, d-1\}$.
\begin{itemize}
\item[(i)]
The  angles of $S$ at its $k$-dimensional face $F=[\frac{\tau_0 e_0 + \ldots + \tau_d e_d}{\tau_0^2 + \ldots + \tau_d^2}, \frac{e_1}{\tau_1}, \ldots, \frac{e_k}{\tau_k}]$ are given by
\begin{align*}
\beta(F,S)
&=
\bg_{d-k}\left(\tau_0^2; \tau_{k+1}^2, \ldots, \tau_d^2; 1, \ldots, 1 \right),
\\
\gamma(F,S)
			&=
			\bg_{d-k}\left( -\tau_0^2 -\tau_{k+1}^2-\ldots -\tau_d^2; \tau_{k+1}^2, \ldots, \tau_d^2; 1, \ldots, 1 \right).
		\end{align*}
where $\bg_\ell$ is given by~\eqref{eq:funct_g_formula_all_positive_intro}.
\item[(ii)]
The angles of $S$ at its $k$-dimensional face $F=[\frac{e_1}{\tau_1}, \ldots, \frac{e_{k+1}}{\tau_{k+1}}]$ are given by
\begin{align*}
\beta(F,S)
&=
\bg_{d-k}\left( \tau_0^2; -\tau_0^2 -\ldots - \tau_d^2, \tau_{k+2}^2, \ldots, \tau_d^2; -1,1, \ldots, 1 \right),
\\
\gamma(F,S)
&=
\bg_{d-k}\left( \tau_1^2 +\ldots +\tau_{k+1}^2; -\tau_0^2 -\ldots - \tau_d^2, \tau_{k+2}^2, \ldots, \tau_d^2; 1, \ldots, 1 \right).
\end{align*}
Here,  for $\ell\in \N$,  $\lambda_0>0$, $\lambda_2>0,\ldots, \lambda_\ell>0$,  $\lambda_1<-(\lambda_0+\lambda_2+ \ldots + \lambda_\ell)$ and $\eps_1\in \{\pm 1\}$, the function $\bg_\ell$ is given by
\begin{multline*}
\bg_\ell(\lambda_0; \lambda_1,\ldots, \lambda_\ell;\eps_1, 1, \ldots,1)
=
\sqrt{\frac{\lambda_0}{2\pi}} \Biggl(
\int_0^{\infty} \eee^{-\frac{\lambda_0 y^2}{2}} \prod_{\substack{j=2}}^\ell \Phi \left(\eps_1 \sqrt{\lambda_j} y \right) \dd y
\\
+2 \int_0^{\infty} \eee^{\frac{\lambda_0 y^2}{2}} \Phi\left(-\sqrt{-\lambda_1} y \right) \Im \Bigg( \prod_{\substack{j=2}}^\ell \Phi \left(  i \eps_1  \sqrt{\lambda_j} y \Bigg) \right) \dd y
\Biggr).
\end{multline*}
\end{itemize}
\end{theorem}

\begin{theorem}[Angles of rectangular simplices]\label{theo:angles_rectangular_simpl_intro}
For $d\geq 2$ and $\tau_1, \ldots, \tau_d>0$ consider the simplex $S=[0, e_1/\tau_1, \ldots, e_d/\tau_d]$. Let also $k\in \{0, \ldots, d-1\}$.
\begin{enumerate}
\item[(i)] The  angles of $S$ at its $k$-dimensional face $F=[0, e_1/\tau_1, \ldots, e_k/\tau_k]$ are given by
$
\beta(F,S) = 1/2^{d-k} = \gamma(F,S).
$
\item[(ii)] The angles of $S$ at its $k$-dimensional face $F=[e_1/\tau_1, \ldots, e_{k+1}/\tau_{k+1}]$ are given by
\begin{align*}
\beta(F,S)
=&
\lim_{\lambda_1 \rightarrow \pm \infty}\bg_{d-k} \left( -(\tau_1^2 + \ldots + \tau_d^2)-\lambda_1; \lambda_1, \tau_{k+2}^2, \ldots, \tau_d^2; \sgn(\lambda_1),1,\ldots, 1 \right)\\
=&
\frac{1}{2^{d-k}} + \frac{1}{\pi} \int_0^{\infty} \frac{\eee^{-(\tau_1^2+\ldots + \tau_d^2)y^2/2}}{y} \Im\left( \prod_{j=k+2}^d \Phi\left(-i \tau_j y\right) \right) \dd y,
\\
\gamma(F,S)
=&
\lim_{\lambda_1 \rightarrow \pm \infty} \bg_{d-k}\left( \tau_1^2 +\ldots +\tau_{k+1}^2; \lambda_1, \tau_{k+2}^2, \ldots, \tau_d^2; 1, \ldots, 1 \right)\\
=&
\frac{1}{\sqrt{2\pi}} \int_0^{\infty} \prod_{j=k+2}^d \Phi\left( \frac{\tau_j}{\sqrt{\tau_1^2 + \ldots + \tau_{k+1}^2}}x \right) \eee^{-x^2/2} \dd x.
\end{align*}
\end{enumerate}
\end{theorem}

Proofs of Theorems~\ref{theo:orthocentr_simpl_relint_tangent_normal_angles_intro}, \ref{theo:cones_angles_orthoc_simplices_outside_intro} and~\ref{theo:angles_rectangular_simpl_intro} will be given in Section~\ref{sec:angles_orthocentric_cones_and_oblique_simplices} (Theorems~\ref{theo:orthocentr_simpl_relint_tangent_normal_angles} and~\ref{theo:cones_angles_orthoc_simplices_outside}) and Sections~\ref{sec:explicit_formula_g_d}, \ref{sec:right_angled_simplices}.

\section{Facts from convex geometry} \label{sec:facts_convex_geometry}

\subsection{General notation}

For $d\geq 1$ we let $\R^d$ be the $d$-dimensional Euclidean space with the standard scalar product $\lan\,\cdot\,,\,\cdot\,\ran$ and the associated Euclidean norm $\|\,\cdot\,\|$. We let $\BB^d=\{x\in\R^d:\|x\|\leq 1\}$ be the Euclidean unit ball and $\lambda_d$ denote the $d$-dimensional Lebesgue measure.

The \textit{convex} (respectively, \textit{positive}, \textit{linear}, \textit{affine}) \textit{hull} of a set $A\subseteq \R^d$ is the smallest convex set (respectively, convex cone, linear subspace, affine subspace) containing the set $A$ and is denoted by $\conv A$ (respectively, $\pos A$, $\lin A$, $\aff A$). The convex hull of  finitely  many points $x_1,\ldots,x_n$ is also denoted by $[x_1,\ldots,x_n]$. For background information on convex geometry we refer, for example, to~\cite{barvinok_book} and~\cite{SchneiderBook}.

For a set $A\subseteq \R^d$ we denote by $\relint A$ its relative interior, that is the interior taken with respect to $\aff A$ as the ambient space. Further,  $\relbd A:= A\bsl \relint A$ denotes the relative boundary of $A$.

As usual, $\P[B]$ and $\E X$ denote the probability of the event $B$ and the expectation of a random vector $X$, respectively. For two random vectors $X$ and $Y$ we write $X\overset{d}{=}Y$ if $X$ and $Y$ have the same probability law.

\subsection{Polytopes and their faces}
A \textit{polytope} is a convex hull of finitely many points, while a \textit{polyhedron} is a (non-empty) intersection of finitely many closed half-spaces in $\R^d$. Equivalently, a polytope can be defined as a bounded polyhedron. The dimension $\dim P$ of a polyhedron $P$ is the dimension of its affine hull $\aff P$.
The \textit{$f$-vector} of a $d$-dimensional polyhedron $P\subseteq\R^d$ is defined by
$$
\mathbf{f} (P) := (f_0(P), \ldots,f_{d-1}(P)),
$$
where $f_k(P)$ is the number of $k$-dimensional faces of $P$.
The set of $k$-dimensional faces of a polyhedron $P$ is denoted by $\cF_k(P)$, so that $f_k(P)$ is the cardinality of $\cF_k(P)$. For background information on polytopes, polyhedra, and related topics, we refer the reader to the books~\cite{GruenbaumBook,ziegler_book_lec_on_poly} as well as to~\cite{henk_richter_gebbert_polytopes_review}.

\subsection{Cones and solid angles}
In this paper, the term cone always refers  to a \textit{polyhedral cone}, that is an intersection of finitely many closed half-spaces whose boundaries pass through the origin. In particular, any polyhedral cone is a polyhedron. Equivalently, a polyhedral cone can be defined as a positive hull $\pos (v_1,\ldots, v_n)$ of finitely many vectors $v_1,\ldots, v_n\in \R^d$. The \textit{solid angle} of a cone $C\subseteq \R^d$ is defined as
$$
\alpha(C) := \P[\xi\in C],
$$
where $\xi$ is a random vector having any rotationally invariant distribution on $\lin C$, the linear hull of $C$. Examples of such distributions include the uniform distribution on the unit ball $\BB^d \cap \lin C$ in $\lin C$, the uniform distribution on the unit sphere in $\lin C$ or the standard normal distribution on $\lin C$.   The \textit{polar cone} of $C$ is defined by
$$
C^\circ := \{v\in \R^d\colon \langle v, z\rangle\leq 0 \text{ for all } z\in C\}.
$$
The \textit{tangent cone} $T(F,P)$ at a face $F$ of a polyhedron $P\subseteq\R^d$ is defined as
$$
T(F,P) := \{v\in\R^d\colon  x_0 + v\eps \in P \text{ for some } \eps>0\},
$$
where $x_0$ is any point in the relative interior of $F$ (the definition does not depend on the choice of $x_0$). The \textit{normal cone} of $F$ is the polar to the tangent cone, that is
$$
N(F,P) := T^\circ (F,P) = \{v\in \R^d\colon \langle v, z-x_0\rangle \leq 0 \text{ for all } z\in P\}.
$$
The \textit{internal} and \textit{external} angles at a face $F$ of $P$ are defined as the solid angles of the tangent and the normal cones, respectively:
$$
\beta(F,P) := \alpha(T(F,P)), \qquad \gamma(F,P) := \alpha (N(F,P)).
$$

The normal cone can be described as the set of all outer normal vectors. Consider a polyhedron $P \subseteq \R^d$ and a point $x$ on its relative boundary. A vector $u\in \R^d$ is called an \emph{outer normal vector} of $P$ at $x$, if it satisfies $\lan u, y-x \ran \leq 0$ for all $y \in P$. Let $F\neq P$ be a face of $P$. Then, for every two points $x,\widetilde{x} \in \relint F$, each outer normal vector of $P$ at $x$ is also an outer normal vector of $P$ at $\widetilde{x}$. An outer normal vector of $P$ at the face $F$ is the normal vector of $P$ at any (and hence all) points $x \in \relint F$. Then, the normal cone $N(F,P)$ is the set of all outer normal vectors of $P$ at $F$.

\subsection{Conic intrinsic volumes}
Let $C\subseteq \R^N$ be a $d$-dimensional polyhedral cone. For all $x\in \R^N$ let $\pi_C(x)$ be the unique point in $C$ for which the function $C \rightarrow [0,\infty), y\mapsto \|y-x\|$ attains its minimum. For a face $F$ of $C$, define
$$
\upsilon_F(C) := \P[\pi_C(g) \in \relint F]
$$
where $g$ is a standard Gaussian random vector in $\R^N$.

\begin{definition}[Conic intrinsic volumes]
For a $d$-dimensional polyhedral cone $C$, the conic intrinsic volumes $\upsilon_0(C), \ldots, \upsilon_d(C)$ are defined as
\begin{equation}\label{eq:conic_intr_vol}
\upsilon_k(C)
=
\sum_{F\in \cF_k(C)} \upsilon_F(C)
=
\sum_{F\in \cF_k(C)} \alpha(F) \alpha(N(F, C)),
\quad k\in \{0, \ldots, d\}.
\end{equation}
\end{definition}
It follows that $(\upsilon_0(C)), \ldots, \upsilon_d(C))$ is a probability distribution on $\{0, \ldots, d\}$. For $i\notin\{0,\ldots, d\}$, one defines $\upsilon_i(C):=0$.
For further details on conic intrinsic volumes, we refer the reader to~\cite[Chapter 2.3]{SchneiderConvexCones}, as well as to~\cite{AmelunxenLotzDCG17} and~\cite{amelunxen_lotz_mccoy_tropp_living_on_the_edge}.

\section{Orthocentric simplices}\label{sec:orthocentric_simpl}
Let $v_0,v_1,\ldots, v_d \in \R^N$ be $d+1$ affinely independent points in $\R^N$, where $N\in \{2,3,\ldots\}$ and $d\in \{2,\ldots, N\}$.  Consider the $d$-dimensional simplex $S=[v_0,v_1, \ldots, v_d]$. The \emph{altitude} of a vertex $v_i$ is the line ($=$one-dimensional affine subspace) that contains $v_i$ and is orthogonal to $\aff(v_0, \ldots, v_{i-1}, v_{i+1}, \ldots, v_d)$, the affine hull of the remaining vertices. If all altitudes of $S$ intersect in a single point, this point is called the \emph{orthocenter} of $S$ and $S$ is called an \emph{orthocentric simplex}. Note that the orthocenter is always contained in the affine hull of $S$.

\subsection{Classification of orthocentric simplices}
The next proposition collects some essential   properties of orthocentric simplices.
\begin{proposition}\label{prop:classification_orthocentr_simplices}
Let $v_0,\ldots, v_d \in \R^N$ be affinely independent points, $N\geq d\geq 2$.
\begin{itemize}
\item[(a)] The simplex $S:=[v_0,v_1, \ldots, v_d]$ is orthocentric with orthocenter $w\in \aff S$ if and only if the value of $\lan v_i-w , v_j-w \ran$ does not depend on the choice of $0\leq i \neq j \leq d$.
\end{itemize}
Now, assume that $S$ is orthocentric with orthocenter $w\in \aff S$ and $c\in \R$ satisfies $\lan v_i-w , v_j-w \ran = c$ for  all $0\leq i \neq j \leq d$. Then, the following hold.
\begin{itemize}
\item[(b)] If $w\in \relint S$, then $c<0$.
\item[(c)] If $w\in \relbd S$, then $c=0$ and $w$ is one of the vertices of $S$.
\item[(d)] If $w\in \aff S \setminus S$, then $c>0$. Moreover, $\|v_{i_0}-w\|^2<c$ for exactly one $i_0 \in \{0,\ldots,d\}$ and $\|v_i-w\|^2>c$ for all other $i\in \{0,\ldots, d\}\bsl\{i_0\}$.
\item[(e)] In cases (b) and (d) (that is, if $c\neq 0$) the orthocenter $w$ can be represented as the following affine combination of $v_0,\ldots, v_d$:
\begin{equation}\label{eq:orthocenter_explicit}
w
=
\sum_{i=0}^d \frac{c}{c-\|v_i-w\|^2} v_i,
\qquad
\sum_{i=0}^d \frac{c}{c-\|v_i-w\|^2} = 1.
\end{equation}
\end{itemize}
\end{proposition}

In the setting of~(c), the edges connecting \( w \) to the other vertices of \( S \) are pairwise orthogonal; such simplices are called \emph{rectangular}. The orthocentric simplices described in Parts~(b) and~(d) are referred to as \emph{acute} and \emph{obtuse}, respectively~\cite{Edmonds2005}. Orthocentric simplices that are not rectangular  will be called \emph{oblique}. The number \( c \) is known as the \emph{obtuseness} of the simplex~\cite[Definition~3.2]{Edmonds2005}. Part~(a) is established in~\cite[Theorem 3.1(b)]{Edmonds2005}. The remaining parts are essentially contained in~\cite[Theorems 3.3, 3.4, and 3.8]{Edmonds2005}, but we provide a proof here, as some elements of this proof will be needed later. The proof of~(d) relies on the following lemma.

\begin{lemma}\label{lemma:condition_for_ex_of_orthocentric_cone}
Let $d\in \N$ and $\lambda_0, \lambda_1, \ldots, \lambda_d \in \R \setminus \{0\}$. A necessary and sufficient condition for the positive definiteness of the $d\times d$-matrix
\begin{align*}
G
=
\begin{pmatrix}
\frac{1}{\lambda_1}+\frac{1}{\lambda_0} & \frac{1}{\lambda_0} & \cdots & \frac{1}{\lambda_0} \\
\frac{1}{\lambda_0} & \frac{1}{\lambda_2} + \frac{1}{\lambda_0} & \cdots & \frac{1}{\lambda_0} \\
\vdots  & \vdots  & \ddots & \vdots  \\
\frac{1}{\lambda_0} & \frac{1}{\lambda_0} & \cdots & \frac{1}{\lambda_d} + \frac{1}{\lambda_0}
\end{pmatrix}
\end{align*}
is that one of the following cases occurs.
\begin{enumerate}
\item[(A)] All numbers $\lambda_0, \lambda_1, \ldots, \lambda_d$ are positive.
\item[(B)] One of $\lambda_0, \lambda_1, \ldots,\lambda_d$ is negative, all others are positive, and $\lambda_0 + \lambda_1 + \ldots + \lambda_d < 0$.
\end{enumerate}
\end{lemma}

\begin{proof}
By Sylvester's criterion, the matrix $G$ is positive definite if and only if all principal minors of $G$ are positive. To compute the principal minors,  let $k \in \{1,\ldots , d\}$ and $I=\{i_1, \ldots, i_k\} \subseteq \{1,\ldots , d\}$ with $i_1 < \ldots < i_k$. Using elementary matrix transformations, the corresponding minor can be computed as
\begin{multline}\label{eq:det_G}
\det
\begin{pmatrix}
\frac{1}{\lambda_{i_1}}+\frac{1}{\lambda_0} & \frac{1}{\lambda_0} & \frac{1}{\lambda_0} & \cdots & \frac{1}{\lambda_0} \\
\frac{1}{\lambda_0} & \frac{1}{\lambda_{i_2}} + \frac{1}{\lambda_0} & \frac{1}{\lambda_0} & \cdots & \frac{1}{\lambda_0} \\
\frac{1}{\lambda_0} & \frac{1}{\lambda_0} & \frac{1}{\lambda_{i_3}} + \frac{1}{\lambda_0} &  \cdots & \frac{1}{\lambda_0} \\
\vdots  & \vdots & \vdots  & \ddots & \vdots  \\
\frac{1}{\lambda_0} & \frac{1}{\lambda_0} & \frac{1}{\lambda_0} & \cdots & \frac{1}{\lambda_{i_k}} + \frac{1}{\lambda_0}
\end{pmatrix}
\\
= \det
\begin{pmatrix}
\frac{1}{\lambda_0}+\frac{1}{\lambda_{i_1}} + \frac{\lambda_{i_2}}{\lambda_{i_1} \lambda_0} + \ldots \frac{\lambda_{i_k}}{\lambda_{i_1} \lambda_0} & 0 & 0 & \cdots & 0 \\
\frac{1}{\lambda_0} & \frac{1}{\lambda_{i_2}} & 0 & \cdots & 0 \\
\frac{1}{\lambda_0} & 0 & \frac{1}{\lambda_{i_3}} &  \cdots & 0 \\
\vdots  & \vdots  & \vdots & \ddots & \vdots \\
\frac{1}{\lambda_0} & 0 & 0 & \cdots & \frac{1}{\lambda_{i_k}}
\end{pmatrix}
=
\frac{\lambda_0 + \sum_{i \in I} \lambda_i}{\lambda_0 \prod_{i \in I} \lambda_i}.
\end{multline}
In the process, the first column is subtracted from all other columns. Secondly, in the resulting matrix for all $\ell \in \{2, \ldots ,k\}$, the $\ell$-th row multiplied with $\frac{\lambda_{i_{\ell}}}{\lambda_{i_1}}$ is added to the first row.

It is left to show that~\eqref{eq:det_G} is positive for all $\varnothing \neq I \subseteq \{1, \ldots, d\}$ if and only if either (A) or (B) occurs. If all $\lambda_i$'s are positive, then also~\eqref{eq:det_G} is always positive. Assume that some of the $\lambda_i$'s are negative.

\vspace*{2mm}
\noindent
\emph{Case 1:} $\lambda_0 < 0$. Assume~\eqref{eq:det_G} is positive for all $I\subseteq \{1,\ldots , d\}$. Taking $I=\{1\}, \ldots ,I=\{d\}$ in \eqref{eq:det_G}, this implies $\lambda_1, \ldots , \lambda_d > 0$. Consequently, taking $I=\{1,\ldots , d\}$ implies that $\lambda_0 + \lambda_1 + \ldots + \lambda_d < 0$, which proves~(B).  Conversely, assuming $\lambda_1, \ldots , \lambda_d >0$ and $\lambda_0 + \lambda_1 + \ldots + \lambda_d < 0$ clearly implies that for all $I\subseteq \{1,\ldots , d\}$, both the numerator and the denominator in~\eqref{eq:det_G} are negative and hence the fraction in~\eqref{eq:det_G} is positive.

\vspace*{2mm}
\noindent
\emph{Case 2:} $\lambda_i<0$ for some $i\in \{1,\ldots , d\}$. Assume~\eqref{eq:det_G} is positive for all $I\subseteq \{1,\ldots , d\}$. Due to symmetry in~\eqref{eq:det_G}, we may assume that $i=1$. Hence, $\lambda_0 > 0$ (otherwise the argument from Case 1 would give a contradiction). Taking $I=\{1\}$ implies $\lambda_0 + \lambda_1 < 0$. Taking $I=\{1,2\}, \ldots , I=\{1,d\}$ entails $\lambda_2 ,\ldots ,\lambda_d > 0$. Finally, taking $I=\{1,\ldots, d\}$ yields $\lambda_0 + \ldots + \lambda_d < 0$. Conversely, assume $\lambda_i < 0$, all other $\lambda_j$ are positive and $\lambda_0 + \ldots + \lambda_d < 0$. Then, in~\eqref{eq:det_G} either both the numerator and the denominator are negative (if $i \in I$) or both are positive (if $i \notin I$). Hence, the fraction in~\eqref{eq:det_G} is positive.
\end{proof}

\begin{proof}[Proof of Proposition~\ref{prop:classification_orthocentr_simplices}]
\emph{Proof of (a),  ``$\Rightarrow$".}  Assume that $w$ is the orthocenter of $S$.
Let $i,j,k \in \{0, \ldots, d\}$ be pairwise distinct. By definition of the orthocenter, $\lan v_i-w, v_j-v_k \ran =0$. Hence,
$$
\lan v_i-w, v_j-w \ran = \lan v_i-w, v_j-v_k + v_k-w \ran = \lan v_i-w, v_k-w \ran.
$$
This proves that $\lan v_i-w, v_j-w \ran = \lan v_{i'}-w, v_{j'}-w\ran$ for all $i\neq j$ and $i'\neq j'$ such that $\{i,j\}\cap \{i',j'\}\neq \varnothing$. To remove the latter condition, consider  $l\in \{0, \ldots, d\}$ such that $l\notin \{i,j,k\}$. The same procedure as above yields
$$
\lan v_i-w, v_k-w \ran = \lan v_l-w, v_k-w \ran.
$$
We conclude that $\lan v_i-w, v_j-w \ran = \lan v_l-w, v_k-w \ran$ and the proof of ``$\Rightarrow$" is complete.

\vspace*{2mm}
\noindent
\emph{Proof of (a),  ``$\Leftarrow$".} Conversely, assume there exists $c\in \R$ such that $\lan v_i-w , v_j-w \ran = c$ for all $0\leq i \neq j \leq d$. Hence, for all pairwise distinct $0\leq i,j,k \leq d$ we have
$$
\lan v_i-w, v_j-v_k \ran
=
\lan v_i-w, v_j-v_k + (w-w) \ran
=
\lan v_i-w, v_j-w \ran - \lan v_i-w, v_k-w \ran
=
c-c
=
0.
$$
Thus, $w$ is the orthocenter of $S$.

For the remaining parts, assume that $S$ is orthocentric with orthocenter $w$ and $c\in \R$ satisfies $c=\lan v_i-w, v_j-w \ran$ for all $0\leq i\neq j \leq d$. We have that $w=\sum_{i=0}^d a_i v_i$ for some $a_0, \ldots, a_d \in \R$ with $a_0+\ldots +a_d=1$, since $w\in \aff S$. Note that
$$
0
=
\lan v_k-w,0 \ran
=
\left\lan v_k-w, \sum_{i=0}^d a_i (v_i-w) \right\ran
=
(a_0 + \ldots + a_{k-1}+a_{k+1} +\ldots + a_d) c + a_k\|v_k-w\|^2
$$
holds for all $k=0,\ldots, d$. Hence,
\begin{equation}\label{eq:pf_classif_orthoc_simpl}
0
=
(1-a_k) c + a_k \|v_k-w\|^2 \qquad \text{ for all } k=0,\ldots, d.
\end{equation}

\vspace*{2mm}
\noindent
\emph{Proof of (b).} From~\eqref{eq:pf_classif_orthoc_simpl} with $k=0$ it follows that
$$
0
=
(a_1 +\ldots + a_d) c + a_0\|v_0-w\|^2.
$$
Since $w\in \relint S$, we have $a_i>0$ and $w\neq v_i$ for all $i=0,\ldots , d$. The above equation hence entails $c<0$.

\vspace*{2mm}
\noindent
\emph{Proof of (c).} Since $w\in \relbd S$, there exists some $i_0 \in \{0, \ldots, d\}$ with $a_{i_0} = 0$. After relabeling, we may assume $i_0=0$. With $k=0$ in~\eqref{eq:pf_classif_orthoc_simpl} we have
$$
0
=
(1-a_0) c + a_0 \|v_0-w\|^2
=
1 \cdot c + 0 \cdot \|v_0-w\|^2
=
c.
$$
Consequently,~\eqref{eq:pf_classif_orthoc_simpl} entails $0 = (1-a_k) c + a_k \|v_k-w\|^2 = a_k \|v_k-w\|^2$ for all $k\in \{1, \ldots, d\}$.
Since $a_1+\ldots + a_d =1$, at least one of $a_1, \ldots, a_d$ is non-zero. Hence, there exists $\ell \in \{1, \ldots, d\}$ with $v_{\ell} = w$. Moreover, this $\ell$ is unique since $v_1, \ldots, v_d$ are pairwise distinct.

\vspace*{2mm}
\noindent
\emph{Proof of (d).}  Let $w\in \aff S \setminus S$. First, we prove that $a_i \neq 0$ for all $i=0,\ldots, d$. Indeed, otherwise, after relabeling, we may assume $a_0=0$. Then, \eqref{eq:pf_classif_orthoc_simpl} implies that $0 = (1-a_0)c + a_0 \|v_0-w\|^2 = c$. Hence, again by~\eqref{eq:pf_classif_orthoc_simpl}, for all $k=1,\ldots, d$, it follows that $a_k \|v_k-w\|^2 = 0$. Since $w\notin S$, we have $v_k \neq w$ for all $k$, and hence  $a_k=0$ for all $k=1, \ldots, d$. We therefore obtain $a_0+a_1+ \ldots +a_d = 0$ which contradicts $a_0 + \ldots + a_d = 1$. Thus, the assumption was wrong, and we have $a_i \neq 0$ for all $i=0,\ldots, d$.

Since $w \in \aff S \setminus S$, there exists $j_0 \in \{0, \ldots, d\}$ with $a_{j_0} < 0$. We may assume $j_0=0$. So, $a_0<0$. By~\eqref{eq:pf_classif_orthoc_simpl} we have $0=(1-a_0)c+a_0\|v_0-w\|^2$, and since $v_i \neq w$, we obtain $c>0$. This also shows that
$$
\|v_0-w\|^2
=
\frac{a_0-1}{a_0} c
=
\frac{|a_0|+1}{|a_0|}c
>
c.
$$

Now, consider $i\in \{1, \ldots, d\}$. With~\eqref{eq:pf_classif_orthoc_simpl} and $a_i \neq 0$ we obtain
$$
0
=
\frac{1-a_i}{a_i} c + \|v_i-w\|^2
\quad \Rightarrow \quad
\| v_i-w \|^2 = \frac{a_i-1}{a_i} c
\qquad \text{ for all } i=1,\ldots, d.
$$
Since $c>0$, this entails $(a_i-1)/a_i>0$. Hence for all $i=1, \ldots, d$ we have either $a_i<0$ or $a_i>1$. Since $a_0 +a_1 +  \ldots + a_d = 1$ and $a_0<0$, there exists an index $i_0 \in \{1, \ldots, d\}$ with $a_{i_0} >1$ and therefore also
$
\|v_{i_0}-w\|^2
=
\frac{a_{i_0}-1}{a_{i_0}} c
<
c.
$
The points $w, v_1, \ldots, v_d$ are affinely independent. Indeed, the points $v_0,v_1, \ldots, v_d$ are affinely independent by assumption, and $w$ has a representation as the affine combination $w=a_0v_0 + \ldots + a_dv_d$, where $a_0 \neq 0$ by assumption. Since $v_0, \ldots, v_d$ are affinely independent, this affine combination is unique, so $w \notin \aff(v_1, \ldots, v_d)$. Thus, $w, v_1, \ldots, v_d$ are affinely independent.

Hence the points $v_1-w, \ldots, v_d-w$ are linearly independent and hence their Gram matrix
$$
\left( \left\lan v_i-w, v_j-w \right\ran \right)_{i,j=1}^d
=
\left( \frac{1}{\lambda_0} + \frac{\delta_{ij}}{\lambda_i} \right)_{i,j=1}^d
\; \text{ with } \; \lambda_0 = \frac{1}{c} \;\text{ and } \; \lambda_i = \frac{1}{\|v_i-w\|^2 -c}, \;  i=1, \ldots, d,
$$
is positive definite. By Lemma~\ref{lemma:condition_for_ex_of_orthocentric_cone}, not more than one of $\lambda_1,\ldots, \lambda_d$ is negative. We have already seen that $\lambda_{i_0}<0$.  Thus, $\|v_i-w\|^2 > c$ for all $i\in \{1, \ldots, d\}\bsl\{i_0\}$. This completes the proof of (d).

\vspace*{2mm}
\noindent
\emph{Proof of (e).} Recall that in the above proof we considered the affine combination $w=a_0v_0 + \ldots + a_dv_d$ for the orthocenter.
For all $i=0, \ldots, d$,~\eqref{eq:pf_classif_orthoc_simpl} says
$$
0
=
(1-a_i)c + a_i\|v_i-w\|^2
=
-a_i\left(c-\|v_i-w\|^2\right)+c,
\quad
a_i
=
\frac{c}{c-\|v_i-w\|^2}.
$$
Here we used that $\|v_i-w\|^2 -c\neq 0$. Indeed, if $w\in \relint S$, this is the case, since $c<0$ by Part~(b). If $w\notin S$, this is true by Part~(d). Recalling that $w=a_0v_0 + \ldots + a_dv_d$ and $a_0+\ldots+a_d = 1$ gives~\eqref{eq:orthocenter_explicit}.
\end{proof}

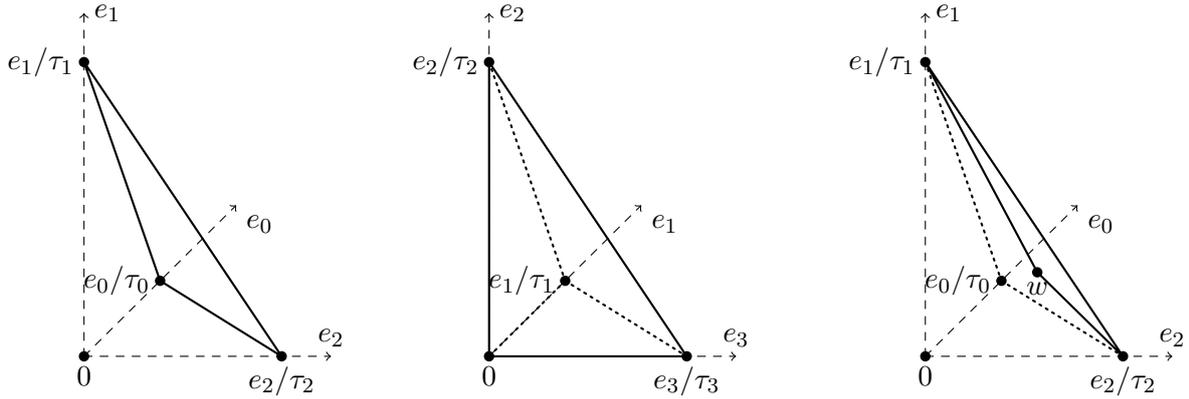
\begin{figure}[h]
    \centering
    \begin{subfigure}{0.3\textwidth}
        \centering
		\begin{tikzpicture}[scale=1.3, line join=round, line cap=round]
            \draw[dashed,->] (0,0,0) -- (2.5,0,0) node[above] {$e_2$};
            \draw[dashed,->] (0,0,0) -- (0,3.5,0) node[right] {$e_1$};
            \draw[dashed,->] (0,0,0) -- (0,0,-4) node[below right] {$e_0$};

            \coordinate (A) at (2,0,0); 
            \coordinate (B) at (0,3,0); 
            \coordinate (C) at (0,0,-2); 

            \draw[thick] (A) -- (C) -- (B) -- (A);

            \fill[black] (A) circle (1.5pt) node[below] {$e_2/\tau_2$};
            \fill[black] (B) circle (1.5pt) node[left] {$e_1/\tau_1$};
            \fill[black] (C) circle (1.5pt) node[left] {$e_0/\tau_0$};
            \fill[black] (0,0,0) circle (1.5pt) node[below] {$0$};
        \end{tikzpicture}
    \end{subfigure}
    \hfill
    \begin{subfigure}{0.3\textwidth}
        \centering
        \begin{tikzpicture}[scale=1.3, line join=round, line cap=round]
            \draw[dashed,->] (0,0,0) -- (2.5,0,0) node[above] {$e_3$};
            \draw[dashed,->] (0,0,0) -- (0,3.5,0) node[right] {$e_2$};
            \draw[dashed,->] (0,0,0) -- (0,0,-4) node[below right] {$e_1$};

            \coordinate (A) at (2,0,0); 
            \coordinate (B) at (0,3,0); 
            \coordinate (C) at (0,0,-2); 
            \coordinate (0) at (0,0,0); 

            \draw[thick] (A) -- (0) -- (B) -- (A);
            \draw[dotted,thick] (C) -- (0);
            \draw[dotted,thick] (B) -- (C);
            \draw[dotted,thick] (C) -- (A);

            \fill[black] (A) circle (1.5pt) node[below] {$e_3/\tau_3$};
            \fill[black] (B) circle (1.5pt) node[left] {$e_2/\tau_2$};
            \fill[black] (C) circle (1.5pt) node[left] {$e_1/\tau_1$};
            \fill[black] (0) circle (1.5pt) node[below] {$0$};
        \end{tikzpicture}
    \end{subfigure}
    \hfill
	\begin{subfigure}{0.35\textwidth}
        \centering
        \begin{tikzpicture}[scale=1.3, line join=round, line cap=round]
            \draw[dashed,->] (0,0,0) -- (2.5,0,0) node[above] {$e_2$};
            \draw[dashed,->] (0,0,0) -- (0,3.5,0) node[right] {$e_1$};
            \draw[dashed,->] (0,0,0) -- (0,0,-4) node[below right] {$e_0$};

            \coordinate (A) at (2,0,0); 
            \coordinate (B) at (0,3,0); 
            \coordinate (C) at (0,0,-2); 
            \coordinate (H) at (0.818,0.542,-0.818); 

            \draw[thick] (H) -- (B) -- (A) -- (H);
            \draw[dotted,thick] (B) -- (C);
            \draw[dotted,thick] (C) -- (A);

            \fill[black] (A) circle (1.5pt) node[below] {$e_2/\tau_2$};
            \fill[black] (B) circle (1.5pt) node[left] {$e_1/\tau_1$};
            \fill[black] (C) circle (1.5pt) node[left] {$e_0/\tau_0$};
            \fill[black] (H) circle (1.5pt) node[below] {$w$};
            \fill[black] (0,0,0) circle (1.5pt) node[below] {$0$};
        \end{tikzpicture}
    \end{subfigure}
    \caption{Canonical examples of orthocentric simplices. Left: acute, Example~\ref{example:e_i/tau_i}. Middle:  rectangular,  Example~\ref{example:degenerate_orthoc_simplex}. Right: obtuse, Example~\ref{example:orthocenter_outside}.}
    \label{figure:three_orthocentric_simplices}
\end{figure}

\subsection{Acute orthocentric simplices}
Next, we present canonical forms for orthocentric simplices and show that, using an isometry, any orthocentric simplex can be transformed into one of these forms. We begin with the acute case. Let $e_0,\ldots, e_d$ be the standard orthonormal basis of $\R^{d+1}$.
\begin{example}\label{example:e_i/tau_i}
Take some positive numbers $\tau_0,\ldots,\tau_d>0$ and consider the $d$-dimensional simplex $S\subseteq \R^{d+1}$ defined by
$$
S
:=
\left[\frac{e_0}{\tau_0},\ldots, \frac{e_d}{\tau_d}\right]
=
\Bigg\{ (x_0, \ldots, x_d) \in \R^{d+1}: x_0,\ldots,x_d\geq 0,\; \sum_{i=0}^d \tau_i x_i = 1\Bigg\};
$$
see the left panel of Figure~\ref{figure:three_orthocentric_simplices}. The affine hull of $S$ is given by
$$
\aff S = \Bigg\{x\in \R^{d+1}:\sum_{i=0}^d \tau_i x_i = 1\Bigg\}.
$$
The simplex \( S \) is orthocentric, and its orthocenter---which is also the projection of the origin onto \( \operatorname{aff} S \)---is the point
$$
w=
\frac{\tau_0e_0+\tau_1e_1 +\ldots +\tau_de_d}{\tau_0^2+\tau_1^2+\ldots + \tau_d^2} =   \sum_{i=0}^{d} \frac{\tau_i^2}{\tau_0^2 + \ldots + \tau_d^2} \cdot \frac{e_i}{\tau_i}.
$$
Indeed, the second formula implies  $w \in \relint S \subseteq \aff S$ and, for all pairwise distinct $i, j, k \in \{0, \ldots, d\}$, we have
$$
\left\lan w-\frac{e_i}{\tau_i}, \frac{e_j}{\tau_j}-\frac{e_k}{\tau_k} \right\ran = 0.
$$
\end{example}

Two simplices  $[v_0,\ldots, v_d]\subseteq\R^{N}$ and $[v_0',\ldots, v_d']\subseteq \R^{N'}$ are called \emph{isometric} if (after relabeling the vertices, if necessary) $\|v_i - v_j\| = \|v_{i}' - v_{j}'\|$ for all $0\leq i,j \leq d$.
The following proposition states that every orthocentric simplex whose orthocenter lies in the relative interior is isometric to a simplex  described in Example~\ref{example:e_i/tau_i}.

\begin{proposition}\label{prop:orthoc_simpl_relint}
Let $v_0, \ldots, v_d \in \R^N$ be affinely independent such that $S=[v_0, \ldots, v_d]$ is orthocentric with orthocenter $w\in \relint S$. As we already know, there exists $c<0$ with $c=\lan v_i-w, v_j-w \ran$ for all $0\leq i \neq j \leq d$. Define
$$
\tau_i
=
\frac{1}{\sqrt{\|v_i-w\|^2 -c}},
\qquad  i=0, \ldots, d.
$$
Then, $S$ is isometric to $[e_0/\tau_0, \ldots, e_d/\tau_d]$.
\end{proposition}
\begin{proof}
The existence of such a $c<0$ is ensured by Proposition~\ref{prop:classification_orthocentr_simplices} (a), (b). For all $0\leq i \neq j \leq d$, observe that
$$
\left\| v_i-v_j \right\|^2
=
\left\| (v_i-w) - (v_j-w)\right\|^2
=
\|v_i-w\|^2 + \|v_j-w\|^2 -2c
=
\frac{1}{\tau_i^2} + \frac{1}{\tau_j^2}
=
\left\| \frac{e_i}{\tau_i} - \frac{e_j}{\tau_j} \right\|^2.
$$
This implies that the two simplices are isometric.
\end{proof}

\subsection{Rectangular orthocentric simplices}
Next we are going to describe a ``canonical form'' for orthocentric simplices with an orthocenter at one of the vertices.
\begin{example}\label{example:degenerate_orthoc_simplex}
Let $e_1,\ldots, e_d$ be the standard orthonormal basis of $\R^d$. Let $\tau_1, \ldots, \tau_d >0$ and consider the $d$-dimensional simplex
$$
S=[0,e_1/\tau_1, \ldots, e_d/\tau_d]\subseteq \R^d;
$$
see the middle  panel of Figure~\ref{figure:three_orthocentric_simplices}.
It is orthocentric with orthocenter $0$. Indeed, for all pairwise different $1\leq i,j,k \leq d$, we have $\lan \frac{e_i}{\tau_i} - 0, \frac{e_j}{\tau_j} - \frac{e_k}{\tau_k} \ran = 0$.
\end{example}

\begin{proposition}\label{prop:classif_degenerate_orthoc_simplex}
Let $v_0, \ldots, v_d \in \R^N$ be affinely independent such that the simplex $S=[v_0, \ldots, v_d]$ is orthocentric with orthocenter $v_0$. Then, $S$ is isometric to $[0, \|v_1-v_0\| e_1, \ldots, \|v_d-v_0\| e_d]$.
\end{proposition}
\begin{proof}
By Proposition~\ref{prop:classification_orthocentr_simplices} (c), we have $\lan v_i-v_0, v_j-v_0 \ran = 0$ for all $1\leq i \neq j \leq d$. Hence, for $1\leq i\neq j \leq d$,
$$
\|v_i-v_j\|^2
=
\|(v_i-v_0) - (v_j-v_0) |^2
=
\|v_i-v_0\|^2 + \|v_j-v_0\|^2
=
\left\| \|v_i-v_0\|e_i - \|v_j-v_0\|e_j \right\|^2.
$$
Moreover, for $i=1, \ldots, d$, we have $\|v_i-v_0\|^2 = \| \|v_i-v_0\|e_i - 0\|^2$. Together, this implies that the two simplices are isometric.
\end{proof}

\subsection{Obtuse orthocentric simplices}
Finally, we describe, up to isometry, orthocentric simplices whose orthocenter lies outside the simplex. Their ``canonical form'' is somewhat more tricky than in the previous two examples.

\begin{example}\label{example:orthocenter_outside}
Let $\tau_0, \tau_1, \ldots, \tau_d>0$. Consider the $d$-dimensional simplex
$$
S
=
\left[ \frac{\tau_0e_0+\tau_1e_1 +\ldots +\tau_de_d}{\tau_0^2+\tau_1^2+\ldots + \tau_d^2}, \frac{e_1}{\tau_1}, \ldots, \frac{e_d}{\tau_d} \right]\subseteq \R^{d+1};
$$
see the right panel of Figure~\ref{figure:three_orthocentric_simplices}. This simplex is orthocentric and its orthocenter is $e_0/\tau_0$. This can be proven in several different ways. Let $v_0, \ldots, v_d$ denote the vertices of $S$. One way is to verify $\lan v_i-e_0/\tau_0, v_j-v_k \ran = 0$ for all pairwise different $0 \leq i,j,k \leq d$. Alternatively, one can check that
$$
\left\lan v_i-\frac{e_0}{\tau_0}, v_j - \frac{e_0}{\tau_0} \right\ran
=
\frac{1}{\tau_0^2}
=:
c
$$
for all $0\leq i \neq j \leq d$, which does not depend on the choice of $i\neq j$, so Proposition~\ref{prop:classification_orthocentr_simplices} (a) entails that the simplex is orthocentric with orthocenter $e_0/\tau_0$. Finally, the orthocentric property can be seen as a consequence of Remark~\ref{rem:egervary_orthocentric_collection} combined with Example~\ref{example:e_i/tau_i}.

Note that $e_0/\tau_0 \notin S$ since $c>0$. The next proposition states that every obtuse orthocentric simplex is of the above form, up to isometry.
\end{example}

\begin{proposition}\label{prop:classification_orthocenter_outside}
Let $v_0, \ldots, v_d \in \R^N$ be affinely independent such that $S=[v_0, \ldots, v_d]$ is orthocentric with orthocenter $w\notin S$. By Proposition~\ref{prop:classification_orthocentr_simplices}~(d), there exists $c>0$ with $c=\lan v_i-w, v_j-w \ran$ for all $0\leq i \neq j \leq d$ and, without loss of generality, let  $\|v_{0}-w\|^2<c$ and $\|v_i-w\|^2>c$ for all $i\in \{1, \ldots, d\}$. Define
$$
\tau_i
=
\frac{1}{\sqrt{\|v_i-w\|^2-c}}
\; \text{ for } i\in \{1, \ldots, d\}
\;\text{ and }\;
\tau_{0}
=
\sqrt{\sum_{j=0}^d \frac{1}{c-\|v_j-w\|^2}} = \sqrt{\frac 1c}.
$$
Then, $S$ is isometric to
$$
\left[\frac{\tau_0e_0+\tau_1e_1 +\ldots +\tau_de_d}{\tau_0^2+\tau_1^2+\ldots + \tau_d^2}, \frac{e_1}{\tau_1}, \ldots, \frac{e_d}{\tau_d} \right].
$$
\end{proposition}
\begin{proof}
The equivalence of both formulas for $\tau_0$ follows from the second identity in~\eqref{eq:orthocenter_explicit}.
For all $1\leq i \neq j \leq d$  observe that
$$
\left\| v_i-v_j \right\|^2
=
\left\| (v_i-w) - (v_j-w)\right\|^2
=
\|v_i-w\|^2 + \|v_j-w\|^2 -2c
=
\frac{1}{\tau_i^2} + \frac{1}{\tau_j^2}
=
\left\| \frac{e_i}{\tau_i} - \frac{e_j}{\tau_j} \right\|^2.
$$
Moreover, note that
$$
\tau_0^2
=
-\frac{1}{\|v_{0}-w\|^2-c}-(\tau_1^2 + \ldots  + \tau_d^2),
\qquad
\|v_{0}-w\|^2-c
=
-\frac{1}{\tau_0^2 + \ldots + \tau_d^2}.
$$
For all $j\in \{1, \ldots, d\}$ it hence follows that
$$
\left\| v_{0}-v_j \right\|^2
=
\|v_{0}-w\|^2 + \|v_j-w\|^2 -2c
=
-\frac{1}{\tau_0^2 + \ldots + \tau_d^2} +\frac{1}{\tau_j^2}
=
\left\| \frac{\tau_0e_0+\tau_1e_1 +\ldots +\tau_de_d}{\tau_0^2+\tau_1^2+\ldots + \tau_d^2} - \frac{e_j}{\tau_j} \right\|^2.
$$
Together, it follows that the two simplices are isometric.
\end{proof}

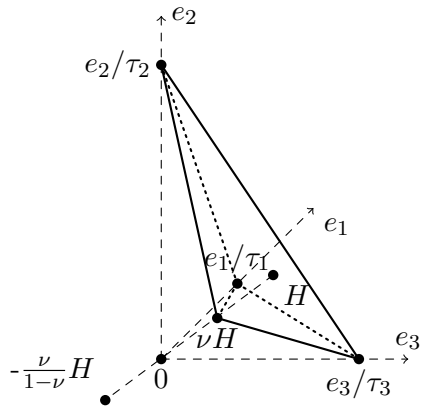
\begin{figure}[h]
		\centering
		\begin{tikzpicture}[scale=1.3, line join=round, line cap=round]
			\draw[dashed,->] (0,0,0) -- (2.5,0,0) node[above] {$e_3$};
			\draw[dashed,->] (0,0,0) -- (0,3.5,0) node[right] {$e_2$};
			\draw[dashed,->] (0,0,0) -- (0,0,-4) node[below right] {$e_1$};
			
			\coordinate (A) at (2,0,0); 
			\coordinate (B) at (0,3,0); 
			\coordinate (C) at (0,0,-2); 
			\coordinate (H) at (0.818,0.542,-0.818); 
			\coordinate (nuH) at (0.409,0.261,-0.409); 
			\coordinate (O) at (-0.409,-0.261,0.409); 
			
			\draw[thick] (A) -- (nuH) -- (B) -- (A);
			\draw[dotted,thick] (C) -- (nuH);
			\draw[dotted,thick] (B) -- (C);
			\draw[dotted,thick] (C) -- (A);
			
			\fill[black] (A) circle (1.5pt) node[below] {$e_3/\tau_3$};
			\fill[black] (B) circle (1.5pt) node[left] {$e_2/\tau_2$};
			\fill[black] (C) circle (1.5pt) node[above] {$e_1/\tau_1$};
			\fill[black] (H) circle (1.5pt) node[below right] {$H$};
			\fill[black] (nuH) circle (1.5pt) node[below] {$\nu H$};
			\fill[black] (O) circle (1.5pt) node[above left] {-$\frac{\nu}{1-\nu} H$};
			\fill[black] (0,0,0) circle (1.5pt) node[below] {$0$};
			
			\draw[black,dashed] (H) -- (O);
		\end{tikzpicture}
	\caption{Simplex from Example~\ref{exam:orthocenter_outside_alternative}.}
	\label{figure:orthocentric_simplices_nu}
\end{figure}

\subsection{All three families together}
All three subclasses of orthocentric simplices—acute, obtuse, and rectangular—can be encompassed within a single parametric family.

\begin{example}\label{exam:orthocenter_outside_alternative}
Let $e_1, \ldots, e_d$ denote the standard orthonormal basis of $\R^d$, and let $\tau_1, \ldots, \tau_d > 0$.
For $\nu \in \R \setminus \{1\}$, consider the simplex
\[
T := \left[\nu H, \frac{e_1}{\tau_1}, \ldots, \frac{e_d}{\tau_d} \right]\subseteq \R^d, \quad \text{where} \quad H := \frac{\tau_1 e_1 + \ldots + \tau_d e_d}{\tau_1^2 + \ldots + \tau_d^2};
\]
see Figure~\ref{figure:orthocentric_simplices_nu}. This simplex is orthocentric, with orthocenter at
$
w = -\frac{\nu}{1 - \nu} H.
$

This can be verified in two ways. Let $v_0, \ldots, v_d$ be the vertices of $T$, with $v_0 = \nu H$. One may check that $\langle v_i - w, v_j - v_k \rangle = 0$ for all distinct $0 \leq i, j, k \leq d$, so $w$ satisfies the defining property of an orthocenter.
Alternatively, one can compute
\[
\langle v_i - w, v_j - w \rangle = \frac{\nu (2 - \nu)}{(1 - \nu)^2} \cdot \frac{1}{\tau_1^2 + \ldots + \tau_d^2} =: c
\]
for all $0 \leq i \neq j \leq d$,  and thus, by Proposition~\ref{prop:classification_orthocentr_simplices}~(a), $T$ is orthocentric with orthocenter $w$.

Since reflection of $T$ across the hyperplane $\aff(e_1/\tau_1, \ldots, e_d/\tau_d)$ corresponds to the transformation $\nu \mapsto 2 - \nu$, it suffices to consider $\nu < 1$.
Then the classification follows from the sign of $c$: for $0 < \nu < 1$, we have $c > 0$ and $T$ is obtuse; for $\nu < 0$, we have $c < 0$ and $T$ is acute; and for $\nu = 0$, we have $c = 0$ and $T$ is rectangular.
In the obtuse case, choosing \(\nu = 1 - \tau_0 / (\tau_0^2 + \ldots + \tau_d^2)^{1/2}\), makes the simplex \(T\)  isometric to the simplex \(S\) from Example~\ref{example:orthocenter_outside}. In the acute case, setting \(\nu = 1 - (\tau_0^2 + \ldots + \tau_d^2)^{1/2} / \tau_0\) yields a simplex \(T\) isometric to the simplex \(S\) from Example~\ref{example:e_i/tau_i}.
\end{example}

\section{Orthocentric cones}
\subsection{Definition of orthocentric cones}\label{subsec:def_orthocentric_cones}
Our goal is to compute internal and external angles of orthocentric simplices. To this end, we need to describe their tangent and normal cones. We start by introducing a class of polyhedral cones which, as we shall show later, appear as tangent and normal cones of orthocentric simplices.

\begin{definition}\label{def:ortho_cone}
Let $v_1,\ldots,v_d$ be linearly independent vectors in $\R^N$ (with $N\geq d$) satisfying
\begin{equation}\label{eq:v_i_scalar_prod}
\langle v_i, v_j\rangle
=
\frac{1}{\lambda_0} + \frac{\delta_{ij}}{\lambda_i}
=
\begin{cases}
\frac 1 {\lambda_0}, &\text{ if } i\neq j,\\
\frac 1 {\lambda_i} + \frac 1 {\lambda_0}, & \text{ if } i=j,
\end{cases}
\qquad
i,j\in \{1,\ldots,d\},
\end{equation}
for some real numbers $\lambda_0,\lambda_1,\ldots,\lambda_d \in  \R\backslash\{0\}$. Moreover, let $\eps_1, \ldots, \eps_d \in \{\pm 1\}$. Then, the positive hull of the vectors $\eps_1 v_1, \ldots, \eps_d v_d$ is referred to as an \emph{orthocentric cone} and denoted by
$$
C_d(\lambda_0; \lambda_1,\ldots,\lambda_d; \eps_1, \ldots, \eps_d) := \pos (\eps_1 v_1,\ldots,\eps_d v_d).
$$
Note that this cone is defined up to isometry only and the concrete choice of $v_1,\ldots, v_d$ and the ambient dimension $N$ is of no relevance as long as relations~\eqref{eq:v_i_scalar_prod} are satisfied. Recall from Lemma~\ref{lemma:condition_for_ex_of_orthocentric_cone} that $C_d(\lambda_0; \lambda_1,\ldots,\lambda_d; \eps_1, \ldots, \eps_d)$ exists in two cases:
\begin{enumerate}
\item[(A)] All numbers $\lambda_0, \lambda_1, \ldots, \lambda_d$ are positive.
\item[(B)] One of $\lambda_0, \lambda_1, \ldots,\lambda_d$ is negative, all others are positive, and $\lambda_0 + \lambda_1 + \ldots + \lambda_d < 0$.
\end{enumerate}
\end{definition}

\subsection{Dual cones of orthocentric cones}
In the next theorem we show that the polar of an orthocentric cone is again orthocentric.

\begin{theorem}
[The polar of $C_d(\lambda_0;\lambda_1,\ldots,\lambda_d; \eps_1, \ldots, \eps_d)$]
\label{theo:orthocentric_cone_dual}
Let $v_1,\ldots,v_d\in \R^d$ be linearly independent vectors satisfying~\eqref{eq:v_i_scalar_prod} for some $\lambda_0,\lambda_1,\ldots,\lambda_d\in \R\setminus\{0\}$ and let $\eps_1, \ldots, \eps_d \in \{ \pm 1\}$. Then, the polar cone of $\pos(\eps_1 v_1,\ldots, \eps_d v_d)$ is isometric to
$$
C_d(-\lambda_0 - \lambda_1 - \ldots - \lambda_d ;\lambda_1,\ldots,\lambda_d; \eps_1 \sgn(\lambda_1), \ldots, \eps_d \sgn(\lambda_d)).
$$
\end{theorem}
\begin{proof}
First, we prove that the polar of the cone $\pos(v_1, \ldots, v_d)$ is isometric to\linebreak $C_d(-\lambda_0 - \lambda_1 - \ldots - \lambda_d ;\lambda_1,\ldots,\lambda_d; \sgn(\lambda_1), \ldots, \sgn(\lambda_d))$.

The $(d-1)$-dimensional faces of $\pos(v_1,\ldots, v_d)$ are of the form  $F_i = \pos(v_1,\ldots, v_{i-1}, v_{i+1}, \ldots, v_d)$, $i=1,\ldots, d$. By conic duality, each such $(d-1)$-dimensional  face of $\pos(v_1,\ldots, v_d)$ corresponds to a one-dimensional face of the polar cone. This one-dimensional face is a ray spanned by some vector $w_i$ which is orthogonal to all $v_j$ with $j \neq i$ and satisfies $\langle w_i, v_i \rangle < 0$. After multiplying $w_i$ by a positive scalar, we may assume that $\langle w_i, v_i \rangle = -1$.
Thus, the polar of the cone $\pos(v_1,\ldots, v_d)$ is the positive hull of the vectors $w_1,\ldots,w_d$ with the property that
$$
\langle v_i, w_j\rangle = - \delta_{ij},
\qquad
i,j\in \{1,\ldots,d\}.
$$
Let $V$ be a non-singular matrix whose columns are $v_1,\ldots,v_d$ and, similarly, let $W$ be a matrix with columns $w_1,\ldots,w_d$. Then, the above relation means that $-W^\top V$ is the $d\times d$ identity matrix. Consequently,
$$
W
=
(-V^{-1})^{\top}.
$$
The Gram matrix of $w_1,\ldots, w_d$, whose entries are $\langle w_i, w_j\rangle$,  is then given by
$$
W^\top W
=
(-V^{-1}) (-V^{-1})^\top = (V^\top V)^{-1},
$$
i.e.\ it is the inverse of the Gram matrix of $v_1,\ldots,v_d$. The inverse of the Gram matrix with entries~\eqref{eq:v_i_scalar_prod} is given by
$$
\langle w_i, w_j \rangle
=
- \frac{\lambda_i \lambda_j}{\lambda_0 + \lambda_1+ \ldots +\lambda_d} + \lambda_i \delta_{ij},
\qquad
i,j\in \{1,\ldots,d\}.
$$
This can be seen by multiplying the matrices.
Consequently, the polar cone of $\pos(v_1,\ldots, v_d)$ is $\pos(w_1, \ldots, w_d) = \pos(\sgn(\lambda_1)\widetilde{w}_1, \ldots, \sgn(\lambda_d)\widetilde{w}_d)$, with the notation $\widetilde{w}_i = w_i/\lambda_i$ for $i=1,\ldots, d$, and with
\begin{equation}\label{eq:w_i_scalar_prod}
\lan \widetilde{w}_i, \widetilde{w}_j \ran
=
- \frac{1}{\lambda_0 + \lambda_1+ \ldots +\lambda_d} +  \frac{\delta_{ij}}{\lambda_i},
\qquad
i,j\in \{1,\ldots,d\}.
\end{equation}
Finally, we have that $\pos(\eps_1 v_1, \ldots, \eps_d v_d)^{\circ} = \pos(\eps_1 \sgn(\lambda_1) \widetilde{w}_1, \ldots, \eps_d \sgn(\lambda_d) \widetilde{w}_d)$, since
$$
\lan \eps_i v_i, \eps_j \sgn(\lambda_j) \widetilde{w}_j\ran
=
\eps_i \eps_j  \sgn(\lambda_j) \frac{1}{\lambda_j} \lan v_i, w_j \ran
\begin{cases}
=0, &\text{ if } i\neq j,\\
<0, & \text{ if } i=j,
\end{cases}
\qquad
i,j\in \{1,\ldots,d\}.
$$
By Definition~\ref{def:ortho_cone}, this cone is isometric to $C_d(-\lambda_0 - \lambda_1-\ldots - \lambda_d ;\lambda_1,\ldots,\lambda_d; \eps_1 \sgn(\lambda_1), \ldots, \eps_d \sgn(\lambda_d))$.
\end{proof}

\subsection{Tangent and normal cones of orthocentric cones and simplices}
The next theorem shows that the tangent cones at faces of an orthocentric cone become orthocentric after factoring out the lineality space.

\begin{theorem}[Tangent and normal cones of orthocentric cones]\label{theo:orthocentr_cones_tangent_normal}
Take linearly independent vectors $v_1,\ldots, v_d\in \R^d$  satisfying~\eqref{eq:v_i_scalar_prod} for some $\lambda_0,\lambda_1,\ldots,\lambda_d\in \R\backslash\{0\}$ and with $\eps_1, \ldots, \eps_d \in \{\pm 1\}$, consider the cone $C:=\pos(\eps_1v_1,\ldots,\eps_dv_d)$. Let $k\in \{1,\ldots,d-1\}$. Then, the normal cone $N(F,C)$ of $C$ at its face $F:= \pos(\eps_1 v_1,\ldots, \eps_k v_k)$ is isometric to
\begin{equation}\label{eq:orthocentr_normal_cone}
C_{d-k}( - \lambda_0 - \lambda_1 - \ldots - \lambda_d ;\lambda_{k+1},\ldots,\lambda_d; \eps_{k+1} \sgn(\lambda_{k+1}), \ldots, \eps_d \sgn(\lambda_d)).
\end{equation}
Furthermore, the tangent cone of the cone $C$ at $F$  can be represented as a direct orthogonal sum $T(F,C) = L \oplus D$ of its lineality space $L = \lin(v_1,\ldots, v_k)$ and a pointed cone $D$  isometric to
\begin{equation}\label{eq:orthocentr_tangent_cone}
C_{d-k}(\lambda_0 + \lambda_{1}+\ldots+\lambda_k;\lambda_{k+1},\ldots,\lambda_d; \eps_{k+1}, \ldots, \eps_d).
\end{equation}
\end{theorem}

\begin{proof}
We continue to use the notation from the proof of Theorem~\ref{theo:orthocentric_cone_dual}.
The normal cone at $F$ is isometric to the positive hull of the vectors $\eps_{k+1} \sgn(\lambda_{k+1}) \widetilde{w}_{k+1}$, $\ldots$, $\eps_d \sgn(\lambda_d) \widetilde{w}_d$, where the scalar products of the $\widetilde{w}_i$'s with each other are given by~\eqref{eq:w_i_scalar_prod}. Hence, the normal cone $N(F,C)$ of $C$ at $F$ is isometric to the one given in~\eqref{eq:orthocentr_normal_cone}. The tangent cone $T(F,C)$ is just the polar of $N(F,C)$. Note that $\lin N(F,C)$ is the orthogonal complement of $\lin (v_1,\ldots,v_k)$.  Therefore, $T(F,C)$ is the direct orthogonal sum of the linear space $\lin (v_1,\ldots, v_k)$ and the cone $D$ defined as the polar cone of $N(F,C)$ taken with respect to the ambient space $\lin N(F,C)$. By Theorem~\ref{theo:orthocentric_cone_dual}, the cone $D$ is isometric to
$$
C_{d-k}(\lambda'; \lambda_{k+1}, \ldots, \lambda_d; \eps_{k+1} \sgn(\lambda_{k+1})^2, \ldots, \eps_d \sgn(\lambda_d)^2),
$$
with
$$
\lambda'
=
-(-\lambda_0-\lambda_1-\ldots -\lambda_d)-\lambda_{k+1} -\ldots -\lambda_d
=
\lambda_0+\lambda_1+\ldots +\lambda_k.
$$
This yields~\eqref{eq:orthocentr_tangent_cone}.
\end{proof}

In the next theorem we show that the tangent and normal cones of an oblique orthocentric simplex are orthocentric cones, thus justifying their name. (For tangent cones we again have to factor out the lineality space.) The case of rectangular simplices will be studied somewhat later in Proposition~\ref{prop:degenerate_simpl_cones}.

\begin{theorem}[Tangent and normal cone of oblique orthocentric simplices]
\label{theo:cones_of_orthoc_simpl_nondegenerate_case}
For $d\geq 2$ let $v_0, \ldots, v_d \in \R^d$ be affinely independent and $S=[v_0, \ldots, v_d]$ an orthocentric simplex with orthocenter $w\in \relint S$ or $w \in \aff S \setminus S$. By Proposition~\ref{prop:classification_orthocentr_simplices} there exists $c\neq 0$ with $\lan v_i-w, v_j-w \ran = c$ for all $0\leq i \neq j \leq d$ and the numbers
$$
\mu_i = \frac{1}{\| v_i-w\|^2-c},  \qquad i=0, \ldots, d,
$$
are well-defined. Let $k\in \{0,\ldots, d-1\}$.
\begin{enumerate}
\item The tangent cone of $S$  at the face $F:= [v_0, \ldots, v_k]$ can be represented as a direct orthogonal sum $T(F,S) = L \oplus D$ of the $k$-dimensional lineality space $L$ of $F$ and a pointed cone $D$ isometric to
\begin{multline*}
C_{d-k}\left(-\frac{1}{c}-\mu_{k+1}-\ldots -\mu_d; \mu_{k+1}, \ldots, \mu_d;1,\ldots, 1\right)
\\
=
C_{d-k}(\mu_0 + \ldots + \mu_k; \mu_{k+1}, \ldots, \mu_d;1,\ldots, 1).
\end{multline*}
\item The normal cone $N(F,S)$ of $S$ at $F$ is isometric to
\begin{multline*}
C_{d-k}\left(\frac{1}{c}; \mu_{k+1}, \ldots, \mu_d; \sgn(\mu_{k+1}), \ldots, \sgn(\mu_d)\right)
\\
=
C_{d-k}(-\mu_0-\ldots-\mu_d; \mu_{k+1}, \ldots, \mu_d; \sgn(\mu_{k+1}), \ldots, \sgn(\mu_d)).
\end{multline*}

\end{enumerate}
\end{theorem}

We give two proofs of the above theorem. In the first proof, the normal cone is considered first, using outer normal vectors. In the second proof, the representation for tangent cone is proven first. We see that it is the same as the tangent cone of a certain orthocentric cone.

\begin{proof}[First proof of Theorem~\ref{theo:cones_of_orthoc_simpl_nondegenerate_case}]
First, note that both in (a) and (b), the two representations of the cone agree by the second formula in~\eqref{eq:orthocenter_explicit}. It suffices to show that the tangent (resp.\ normal) cone is given by one of them.

The normal cone of $S$ at $F$ is spanned by $n_{k+1}, \ldots, n_d$, where $n_i$ denotes an outer normal vector of $S$ at the facet $H_i = [v_0, \ldots, v_{i-1}, v_{i+1}, \ldots, v_d]$, for all $i=0, \ldots, d$. If $w\in \relint S$, then it is clear from the definition of the orthocenter that we can take $w-v_i$ as an outer normal vector at $H_i$. Since the case when $w$ is outside $S$ is not so clear, let us present an argument valid both for $w\in \relint S$ and $w \in \aff S \setminus S$.
We have
$$
n_i \in \lin(v_j-v_k: j,k \in \{0, \ldots, i-1, i+1, \ldots, d\}, j\neq k)^{\perp},
$$
which is $1$-dimensional. Since $w$ is the orthocenter of $S$, we have $(v_i-w) \perp (v_j-v_k)$ for all $j,k \neq i, j\neq k$. Hence, either $v_i-w$ or $-(v_i-w)$ is an outer normal vector at $H_i$. In order to be \emph{outer} normal vector, $n_i$ has to satisfy $\lan n_i, x-v_i \ran >0$ for all $x\in H_i$. For any $j \neq i$, we have that
$$
\lan w-v_i, v_j-v_i \ran = \lan -(v_i-w), (v_j-w) - (v_i-w) \ran = -c+\|v_i-w\|^2,
$$
which has the same sign as $\mu_i$.
Consequently, an outer normal vector at the facet $H_i$ is $n_i = \sgn(\mu_i)(w-v_i)$. Hence, we obtain
$N(F,S) = \pos(\sgn(\mu_{k+1})(w-v_{k+1}), \ldots, \sgn(\mu_d)(w-v_d) )$. Calculating the scalar products of these vectors yields (b). The tangent cone $T(F,C)$ is the polar cone of $N(F,C)$, hence Theorem~\ref{theo:orthocentric_cone_dual} allows to conclude (a) from this.
\end{proof}

\begin{proof}[Second proof of Theorem~\ref{theo:cones_of_orthoc_simpl_nondegenerate_case}]
First, note that by Proposition~\ref{prop:classification_orthocentr_simplices}, the numbers $\mu_i$ are indeed well-defined.  As a first case, consider $k=0$. Then, $T(F,S)=\pos(v_1-v_0, \ldots, v_d-v_0)$ and for $1\leq i,j \leq d$ we have
$$
\lan v_i - v_0, v_j - v_0 \ran
=
\lan (v_i-w) - (v_0-w), (v_j-w) - (v_0-w) \ran
=
\|v_0-w\|^2 - c + \delta_{ij}(\|v_i-w\|^2 -c)
=
\frac{1}{\mu_0} + \frac{\delta_{ij}}{\mu_i}.
$$
This yields (a) in the case $k=0$, which then implies  (b) by Theorem~\ref{theo:orthocentric_cone_dual}.

Now, let $k\in \{1, \ldots ,d-1\}$. Define
$$
C = v_0 + \pos(v_1-v_0, \ldots, v_d-v_0) \quad \text{ and } \quad G = v_0 + \pos(v_1-v_0, \ldots, v_k-v_0).
$$
Then $G$ is a face of $C$. We claim that $T(F,S) = T(G,C)$. Before proving this, let us show how it implies the theorem. Note that $C-v_0$ is isometric to $C_d(\mu_0; \mu_1, \ldots, \mu_d)$. Hence, (a) and (b) follow by Theorem~\ref{theo:orthocentr_cones_tangent_normal}.

It remains to show that $T(F,S) = T(C,G)$. 
We can write $S$ and $F$ as
\begin{align*}
	S&=v_0 + \left\{ a_1\left( v_1-v_0 \right) +\ldots + a_d\left( v_d-v_0 \right): a_1,\ldots, a_d\geq 0, a_1 + \ldots + a_d\leq 1 \right\},\\
	F&=v_0 + \left\{ a_1\left( v_1-v_0 \right) +\ldots + a_k\left( v_k-v_0 \right): a_1,\ldots, a_k\geq 0, a_1 + \ldots + a_k\leq 1 \right\}.
\end{align*}
In particular, this shows $S \subseteq C$ and $F \subseteq G$.

Since $F$ and $G$ both have a $k$-dimensional affine hull, we have $\relint F \subseteq \relint G$. Hence, we can choose some $x_0\in \relint F$ such that
$$
T(F,S) = \{v\in \R^d : x_0+\eps v \in S \text{ for some } \eps>0\}, \; T(G,C) = \{v\in \R^d : x_0+\eps v \in C \text{ for some } \eps>0\}.
$$
The inclusion $S\subseteq C$ immediately implies $T(F,S)\subseteq T(G,C)$. For the reverse inclusion, note that
$$
x_0 = v_0 + b_1\left( v_1-v_0 \right) +\ldots + b_k\left( v_k-v_0 \right)
$$
for some $b_1,\ldots, b_k >0$ with $b_1 + \ldots + b_k<1$. This holds true, since $x_0$ is an element of $\relint F$, using the characterization of the relative interior of a simplex from \cite[Lemma 1.1.11]{SchneiderBook}.

Let $v\in T(G,C)$. Then, there exists some $\eps>0$ with $x_0+\eps v \in C$. Further, $x_0\in C$, and so for all $0\leq \delta \leq \eps$, by convexity of $C$, we have $x_0+\delta v \in C$.  Hence,
$$
x_0+\delta v = v_0 + a_{1,\delta} \left( v_1-v_0 \right) +\ldots + a_{d,\delta} \left( v_d-v_0 \right)
$$
for some $a_{1,\delta}, \ldots, a_{d, \delta} \geqslant 0$. As $\delta \searrow 0$, we have $x_0 + \delta v \rightarrow x_0$ which implies that $a_{i,\delta} \rightarrow b_i$ for $i\in \{1,\ldots, k\}$ and $a_{i,\delta} \rightarrow 0$ for $i\in \{k+1,\ldots, d\}$. Indeed, take $(v_i-v_0: 1\leq i \leq d)$ as a basis of $(\aff C) -v_0$ and use that convergence in this vector space is equivalent to convergence of all coordinates. Consequently, there exists some $\delta>0$ such that $a_{1,\delta} + \ldots + a_{d,\delta} \leq 1$, which implies $x_0+\delta v \in S$. Thus, $v\in T(F,S)$. This implies $T(G,C)\subseteq T(F,S)$, which concludes the proof.
\end{proof}

\section{Angles of orthocentric cones and oblique orthocentric simplices} \label{sec:angles_orthocentric_cones_and_oblique_simplices}

\subsection{Angles of orthocentric cones}
Our goal is to compute the angles of orthocentric cones, as defined in Definition~\ref{subsec:def_orthocentric_cones}.
This analysis will enable us to compute the angles of orthocentric simplices, excluding the rectangular case, which will be treated separately as a limiting case. The resulting angle expressions will first be given in terms of the following function.

\begin{definition}[The function \( \bg_d \)]\label{def:g_d_function}
Let $d\geq 1$ and $\lambda_0, \lambda_1, \ldots, \lambda_d \neq 0$ be such that one of the following cases occurs
\begin{enumerate}
\item[(A)] All numbers $\lambda_0, \lambda_1, \ldots, \lambda_d$ are positive.
\item[(B)] One of $\lambda_0, \lambda_1, \ldots,\lambda_d$ is negative, all others are positive, and $\lambda_0 + \lambda_1 + \ldots + \lambda_d < 0$.
\end{enumerate}
Moreover, let $\eps_1, \ldots, \eps_d \in \{\pm 1\}$. Then, define
$$
\bg_d(\lambda_0; \lambda_1, \ldots, \lambda_d;\eps_1, \ldots, \eps_d)
:=
\P\left[ \eps_1\eta_1<0, \ldots, \eps_d\eta_d<0 \right],
$$
where $(\eta_1, \ldots, \eta_d)$ is a multivariate Gaussian random vector with mean $0$ and covariance matrix
\begin{equation}\label{eq:covariance_matrix}
\left( \frac{1}{\lambda_0} + \frac{\delta_{ij}}{\lambda_i} \right)_{i,j=1}^d =: \Sigma
\end{equation}
Finally, for $d=0$, we  define $\bg_0(\lambda_0) = 1$ for all $\lambda_0 \in \R$.
\end{definition}

The matrix $\Sigma$ is indeed a valid covariance matrix since it is positive definite by Lemma~\ref{lemma:condition_for_ex_of_orthocentric_cone}. Somewhat later we shall derive explicit formulas for the function $\bg_d$. But first we show that various quantities related to orthocentric cones can be expressed in terms of $\bg_d$.

\begin{theorem}[Angle of $C_d(\lambda_0; \lambda_1,\ldots,\lambda_d; \eps_1, \ldots, \eps_d)$]
\label{theo:angle_pos(eps_i_v_i)_via_g_d(lambda_i,eps_i)}
Let $d\geq 1$, $v_1,\ldots,v_d$ satisfy $\lan v_i, v_j\ran = 1/\lambda_0 + \delta_{ij}/\lambda_i$ for all $i,j\in \{1,\ldots,d\}$, for some real numbers $\lambda_0,\lambda_1,\ldots,\lambda_d \in  \R\setminus \{0\}$ and $\eps_1, \ldots, \eps_d \in \{\pm 1\}$. Then, the solid angle of $\pos(\eps_1 v_1, \ldots, \eps_d v_d)$ is given by
$$
\alpha(\pos(\eps_1 v_1, \ldots, \eps_d v_d))
=
\bg_d(-\lambda_0-\lambda_1-\ldots -\lambda_d; \lambda_1,\ldots, \lambda_d; \eps_1 \sgn(\lambda_1), \ldots, \eps_d \sgn(\lambda_d)).
$$
\end{theorem}

\begin{proof}
The solid angle does not depend on the dimension of the ambient space. Hence, we may assume $v_1,\ldots, v_d\in \R^d$.
Let $D=\pos(\eps_1v_1,\ldots,\eps_dv_d)^{\circ}$. By Theorem~\ref{theo:orthocentric_cone_dual} and its proof, $D$ has the representation $D=\pos(\eps_1 \sgn(\lambda_1) w_1, \ldots, \eps_d \sgn(\lambda_d) w_d)$ for some $w_1, \ldots, w_d \in \R^d$ satisfying
\begin{equation}\label{eq:prop_solid_angle_ortho_cone_wi*wj}
\begin{aligned}
\lan w_i, w_j \ran
=
\begin{cases}
-\frac{1}{\lambda_0 + \lambda_1 + \ldots + \lambda_d} + \frac{1}{\lambda_i}, & \text{if } i = j \\
-\frac{1}{\lambda_0 + \lambda_1 + \ldots + \lambda_d}, & \text{if } i \neq j
\end{cases}
\end{aligned}
\qquad
i, j \in \{ 1 \ldots d \}.
\end{equation}
(What we call $w_i$ was denoted by $\widetilde w_i$ in that proof.) Hence,
$$
\pos(\eps_1 v_1, \ldots, \eps_d v_d)
=
D^{\circ}
=
\{x\in \R^d:\lan x, \eps_1\sgn(\lambda_1)w_1\ran \leq 0, \ldots , \lan x, \eps_d\sgn(\lambda_d)w_d \ran\leq 0\}.
$$
Let $\xi = (\xi_1, \ldots, \xi_d)^\top$ be a standard Gaussian random vector on $\R^d$. Then, by the definition of the solid angle, we have that
\begin{align*}
\alpha \left(\pos(\eps_1v_1,\ldots, \eps_dv_d) \right)
&=
\P[\eps_1\sgn(\lambda_1) \lan \xi,w_1\ran \leq 0, \ldots ,\eps_d\sgn(\lambda_d) \lan \xi,w_d \ran \leq 0]\\
&=
\P[\eps_1\sgn(\lambda_1) \eta_1 \leq 0, \ldots, \eps_d \sgn(\lambda_d)\eta_d \leq 0],
\end{align*}
setting $\eta_i = \lan\xi, w_i\ran$ for $i=1,\ldots, d$. Here, $(\eta_1, \ldots, \eta_d)^T = W^T \xi$, where $W\in \R^{d\times d}$ is the matrix with $w_1,\ldots, w_d$ as its columns. Hence, $(\eta_1, \ldots, \eta_d)$ is zero-mean Gaussian with covariance matrix $W^T(W^T)^T = W^T W$, which is the Gram matrix of $w_1, \ldots, w_d$ given in~\eqref{eq:prop_solid_angle_ortho_cone_wi*wj}. By the definition of $\bg_d$, the desired formula follows.
\end{proof}

\begin{corollary}[Conic intrinsic volumes of orthocentric cones] \label{cor:intrinsic_volumes_orthocentr_simplices}
For $d\geq 1$ let $\lambda_0, \lambda_1, \ldots, \lambda_d \in \R \setminus \{0\}$, $\eps_1, \ldots, \eps_d \in \{\pm 1\}$ be as in Definition~\ref{def:g_d_function}.  Then, for every $k\in \{0,\ldots, d\}$, the $k$-th conic intrinsic volume of  $C_d(\lambda_0; \lambda_1, \ldots, \lambda_d; \eps_1, \ldots, \eps_d)$ is given by
\begin{align*}
\upsilon_k(C_d(\lambda_0;\lambda_1,\ldots,\lambda_d;\eps_1, \ldots, \eps_d))&\\
=
\sum_{1\leq i_1<\ldots<i_k\leq d}\bg_k&(-\lambda_0-\lambda_{i_1}-\ldots-\lambda_{i_k};\lambda_{i_1},\ldots,\lambda_{i_k}; \eps_{i_1}\sgn(\lambda_{i_1}), \ldots, \eps_{i_k} \sgn(\lambda_{i_k})) \times\\
\times\bg_{d-k}&(\lambda_0+\lambda_{i_1}+\ldots+\lambda_{i_k};\lambda_{j_1},\ldots,\lambda_{j_{d-k}}; \eps_{j_1}, \ldots, \eps_{j_{d-k}}).
\end{align*}
Here, $j_1, \ldots, j_{d-k}$ are defined via $\{i_1,\ldots,i_k\} \cup \{j_1, \ldots, j_{d-k}\}=\{1,\ldots, d\}$.\\
Since $\bg_0(x)=1$, in the cases $k=0$ and $k=d$, this means
\begin{align*}
\upsilon_0(C_d(\lambda_0;\lambda_1,\ldots,\lambda_d; \eps_1, \ldots, \eps_d))
&=
\bg_d(\lambda_0;\lambda_1,\ldots,\lambda_d; \eps_1, \ldots, \eps_d)
\qquad \text{ and }\\
\upsilon_d(C_d(\lambda_0;\lambda_1,\ldots,\lambda_d; \eps_1, \ldots, \eps_d))
&=
\bg_d(-\lambda_0-\lambda_1-\ldots-\lambda_d;\lambda_1,\ldots,\lambda_d; \eps_1\sgn(\lambda_1), \ldots, \eps_d \sgn(\lambda_d)).
\end{align*}
\end{corollary}

\begin{proof}
Let $0\leq k \leq d$. Recall from~\eqref{eq:conic_intr_vol} that the $k$-th conic intrinsic volume of the cone $C:=C_d(\lambda_0; \lambda_1, \ldots, \lambda_d; \eps_1, \ldots, \eps_d)$ is given by
$$
\upsilon_k(C)=\sum_{F\in \cF_k(C)} \alpha(F)\alpha(N(F,C)).
$$
For $k\in \{1,\ldots,d-1\}$ we can argue as follows. Every face $F\in \cF_k (C)$ is isometric to a cone $C_k(\lambda_0;\lambda_{i_1},\ldots,\lambda_{i_k};\eps_{i_1}, \ldots, \eps_{i_k})$ for certain $1\leq i_1<\ldots <i_k\leq d$. A formula for the solid angle of such a face is given in Theorem~\ref{theo:angle_pos(eps_i_v_i)_via_g_d(lambda_i,eps_i)}:
$$
\alpha(F)=\bg_k(-\lambda_0-\lambda_{i_1}-\ldots-\lambda_{i_k};\lambda_{i_1},\ldots ,\lambda_{i_k}; \eps_{i_1} \sgn(\lambda_{i_1}), \ldots, \eps_{i_k} \sgn(\lambda_{i_k})).
$$
By Theorem~\ref{theo:orthocentr_cones_tangent_normal}, the normal cone of $C$ at such a face is isometric to
$$
C_{d-k}(-\lambda_0-\lambda_1-\ldots-\lambda_d;\lambda_{j_1},\ldots,\lambda_{j_{d-k}}; \eps_{j_1} \sgn(\lambda_{j_1}), \ldots, \eps_{j_{d-k}} \sgn(\lambda_{j_{d-k}})).
$$
Hence, Theorem~\ref{theo:angle_pos(eps_i_v_i)_via_g_d(lambda_i,eps_i)} implies
$$
\alpha(N(F,C))
=
\bg_{d-k}(\lambda_0+\lambda_{i_1}+\ldots+\lambda_{i_k};\lambda_{j_1},\ldots,\lambda_{j_{d-k}}; \eps_{j_1}, \ldots, \eps_{j_{d-k}}).
$$
Together, this implies the formula for $\upsilon_k(C)$.

In the case $k = 0$, note that the only $0$-dimensional face of $C$ is $\{0\}$ and
\begin{align*}
N(\{0\}, C)
&=
C_d(\lambda_0; \lambda_1, \ldots, \lambda_d; \eps_1,\ldots, \eps_d)^{\circ}
\\
&=
C_d(-\lambda_0-\lambda_1 - \ldots - \lambda_d; \lambda_1, \ldots, \lambda_d; \eps_1 \sgn(\lambda_1), \ldots, \eps_d \sgn(\lambda_d))
\end{align*}
by Theorem~\ref{theo:orthocentric_cone_dual}.
Theorem~\ref{theo:angle_pos(eps_i_v_i)_via_g_d(lambda_i,eps_i)} hence entails
$$
\upsilon_0(C)
=
\alpha(\{0\}) \cdot \alpha(N(\{0\}, C))
=
1 \cdot \bg_d(\lambda_0;\lambda_1,\ldots,\lambda_d; \eps_1, \ldots, \eps_d).
$$
In the case $k=d$, note that the only $d$-dimensional face of $C$ is $F=C$ itself and its normal cone is $\{0\}$. Then, with Theorem~\ref{theo:angle_pos(eps_i_v_i)_via_g_d(lambda_i,eps_i)} we find
$$
\upsilon_d(C)
=
\alpha(C) \cdot \alpha(\{0\})
=
\bg_d(-\lambda_0 - \lambda_1 - \ldots - \lambda_d; \lambda_1, \ldots, \lambda_d; \eps_1 \sgn(\lambda_1), \ldots, \eps_d \sgn(\lambda_d)) \cdot 1.
$$
This concludes the proof.
\end{proof}

\subsection{Angles of oblique orthocentric simplices}
In this section, we express the angles of the oblique orthocentric simplices in terms of the function $\bg_d$. We start with the acute case. Recall that by Example~\ref{example:e_i/tau_i} and Proposition~\ref{prop:orthoc_simpl_relint}, acute orthocentric simplices (i.e.\ those containing the orthocenter in their relative interior) are, after applying an isometry, exactly the simplices $\left[e_0/\tau_0, \ldots, e_d/\tau_d\right]$ with $\tau_0, \ldots, \tau_d > 0$. According to Example~\ref{example:e_i/tau_i}, the orthocenter is
$$
w
=
\frac{\tau_0e_0+\tau_1e_1 +\ldots +\tau_de_d}{\tau_0^2+\tau_1^2+\ldots + \tau_d^2}.
$$
The angles of such simplices are given in the following theorem.
\begin{theorem}[Angles of acute orthocentric simplices]\label{theo:orthocentr_simpl_relint_tangent_normal_angles}
Let $d\geq 2$ and fix some $\tau_0,\ldots,\tau_d>0$. Consider the simplex $S:= [e_0/\tau_0,\ldots, e_d/\tau_d] \subseteq \R^{d+1}$ and let $k\in \{0,\dots, d\}$. For the $k$-dimensional face $F:= [e_0/\tau_0,\ldots,e_k/\tau_k]$, the following assertions are satisfied.
\begin{enumerate}
\item For $k\in \{0,\ldots,  d-1\}$, the tangent cone of $S$ at $F$ can be represented as a direct orthogonal sum $T(F,S) = L \oplus D$ of the $k$-dimensional lineality space $L$ of $F$
and a pointed cone $D$ isometric to
$$
C_{d-k}(\tau_0^2 + \ldots + \tau_k^2 ; \tau_{k+1}^2,\ldots, \tau_d^2; 1, \ldots, 1).
$$
For $k=d$, the tangent cone $T(S,F) = T(S,S)$ of $S$ at $S$ is the lineality space of $S$.
\item For $k\in \{0,\ldots,  d-1\}$, the normal cone of $S$ at $F$ is an orthogonal direct sum of a straight line and a cone isometric to
$$
C_{d-k}( - \tau_0^2 - \ldots  - \tau_d^2 ; \tau_{k+1}^2,\ldots, \tau_d^2; 1, \ldots, 1).
$$
If $k=d$, the normal cone $N(F,S) = N(S,S)$ is a straight line.
\item The internal angle of $S$ at $F$ is
$$
\beta(F,S)
=
\bg_{d-k}(-\tau_0^2-\ldots-\tau_d^2;\tau_{k+1}^2,\dots,\tau_d^2; 1, \ldots, 1).
$$
\item The external angle of $S$ at $F$ is
$$
\gamma(F,S)
=
\bg_{d-k}(\tau_0^2+\ldots+\tau_k^2;\tau_{k+1}^2,\dots,\tau_d^2; 1,\ldots ,1).
$$
\end{enumerate}
\end{theorem}
\begin{remark}
An explicit representation for the function $\bg_d$ at arguments as on the right-hand side of~(c) and~(d) of Theorem~\ref{theo:orthocentr_simpl_relint_tangent_normal_angles} will be given in Theorem~\ref{theo:formula_for_bg_d(lambda_i,eps_i)}.
\end{remark}

\begin{proof}[Proof of Theorem~\ref{theo:orthocentr_simpl_relint_tangent_normal_angles}]
\underline{Case $k=d$:} The assertion (a) follows from the definition of the tangent cone, then (b) follows from (a) with $T(F,S)=N(F,S)^{\circ}$. The angles of these cones are both equal to $1$. Hence, the assertions (c) and (d) follow by the definition of $\bg_0$.

\vspace*{2mm}
\noindent
\underline{Case $k\in \{0, \ldots, d-1\}$:} \emph{Proof of (a) and (b).}
For all $0 \leq i \neq j \leq d$, we have
$$
\left\lan \frac{e_i}{\tau_i} -w, \frac{e_j}{\tau_j} -w \right\ran
=
-\frac{1}{\tau_0^2 + \ldots + \tau_d^2}
=:c
\; \text{ and } \;
\mu_i := \frac{1}{\left\| \frac{e_i}{\tau_i}-w \right\|^2 -c}
=
\frac{1}{\left\lan \frac{e_i}{\tau_i}-w, \frac{e_i}{\tau_i}-\frac{e_j}{\tau_j} \right\ran}
=
\tau_i^2.
$$
The assertions (a) and (b) now follow from Theorem~\ref{theo:cones_of_orthoc_simpl_nondegenerate_case}.

\vspace*{2mm}
\noindent
\emph{Proof of (c) and (d).} Note that $\beta(F,S) = \alpha(T(F,S))$ and $\gamma(F,S)=\alpha(N(F,S))$.
The assertions follow from (a), (b) and Theorem~\ref{theo:angle_pos(eps_i_v_i)_via_g_d(lambda_i,eps_i)}.
\end{proof}

Now we turn to the obtuse simplices.  Recall that, by Example~\ref{example:orthocenter_outside} and Proposition~\ref{prop:classification_orthocenter_outside}, obtuse orthocentric simplices (i.e.\ those not containing their  orthocenter) are, after applying an isometry, exactly the simplices of the form
\begin{equation}\label{eq:orthocentric_simpl_obtuse_canonical}
S:= \left[w, \frac{e_1}{\tau_1}, \ldots, \frac{e_d}{\tau_d} \right],
\qquad
w= \frac{\tau_0e_0+\tau_1e_1 +\ldots +\tau_de_d}{\tau_0^2+\tau_1^2+\ldots + \tau_d^2},
\end{equation}
where $\tau_0, \ldots, \tau_d>0$.
The orthocenter is then at $\frac{e_0}{\tau_0}$. Note that these simplices have two types of faces: those containing the vertex $w$ and those not containing $w$.

\begin{theorem}[Angles of obtuse orthocentric simplices]\label{theo:cones_angles_orthoc_simplices_outside}
	Let $d\geq 2$, $\tau_0, \ldots, \tau_d >0$ and consider the simplex $S$ defined in~\eqref{eq:orthocentric_simpl_obtuse_canonical}.
	\begin{enumerate}
		\item $T(S,S) = \R^d$ and $N(S,S)=\{0\}$.
		\item For $k\in \{0,\ldots, d-1\}$, the tangent cone $T(F,S)=L \oplus C$ of $F=[\frac{e_1}{\tau_1}, \ldots, \frac{e_{k+1}}{\tau_{k+1}}]$ is the direct orthogonal sum of the lineality space $L$ of $F$ and a cone $C$ isometric to
		$$
		C_{d-k}\left( \tau_1^2 + \ldots + \tau_{k+1}^2; -\tau_0^2 - \tau_1^2 - \ldots - \tau_d^2 , \tau_{k+2}^2, \ldots, \tau_d^2; 1, \ldots, 1 \right).
		$$
		The normal cone of $S$ at $F$ is
		$$
		N(F,S)
		=
		C_{d-k}\left( \tau_0^2; -\tau_0^2-\tau_1^2 - \ldots -\tau_d^2, \tau_{k+2}^2, \ldots, \tau_d^2;-1, 1, \ldots, 1 \right).
		$$
		\item For $k\in \{0,\ldots, d-1\}$, the tangent cone $T(F,S)=L \oplus C$ of $F=[w,\frac{e_1}{\tau_1}, \ldots, \frac{e_k}{\tau_k}]$ is the direct orthogonal sum of the lineality space $L$ of $F$ and a cone $C$ isometric to
		$$
		C_{d-k}\left(-\tau_0^2-\tau_{k+1}^2-\ldots -\tau_d^2; \tau_{k+1}^2,\ldots, \tau_d^2; 1,\ldots, 1 \right).
		$$
		The normal cone of $S$ at $F$ is
		$$
		N(F,S)
		=
		C_{d-k}\left( \tau_0^2; \tau_{k+1}^2, \ldots, \tau_d^2; 1, \ldots, 1 \right).
		$$
		\item For $k \in \{0, \ldots, d-1\}$, the internal and external angle of $S$ at $F=[\frac{e_1}{\tau_1}, \ldots, \frac{e_{k+1}}{\tau_{k+1}}]$ are
		\begin{align*}
			\beta(F,S)
			&=
		\bg_{d-k}\left( \tau_0^2; -\tau_0^2 -\ldots - \tau_d^2, \tau_{k+2}^2, \ldots, \tau_d^2; -1,1, \ldots, 1 \right),\\
			\gamma(F,S)
			&=
			\bg_{d-k}\left( \tau_1^2 +\ldots +\tau_{k+1}^2; -\tau_0^2 -\ldots - \tau_d^2, \tau_{k+2}^2, \ldots, \tau_d^2; 1, \ldots, 1 \right).
		\end{align*}
		For $k\in \{0,1,\ldots, d\}$, the internal and external angle of $S$ at $F=[w, \frac{e_1}{\tau_1}, \ldots, \frac{e_k}{\tau_k}]$ are
		\begin{align*}
			\beta(F,S)
			&=
			\bg_{d-k}\left(\tau_0^2; \tau_{k+1}^2, \ldots, \tau_d^2; 1, \ldots, 1 \right),\\
			\gamma(F,S)
			&=
			\bg_{d-k}\left( -\tau_0^2 -\tau_{k+1}^2-\ldots -\tau_d^2; \tau_{k+1}^2, \ldots, \tau_d^2; 1, \ldots, 1 \right).
		\end{align*}
		
	\end{enumerate}
\end{theorem}

\begin{remark}
	Explicit representations for $\bg_d$ at arguments as in the right-hand sides of Theorem~\ref{theo:cones_angles_orthoc_simplices_outside} (d) are given in Theorem~\ref{theo:formula_for_bg_d(lambda_i,eps_i)} and Theorem~\ref{theo:formula_for_bg_d(lambda_i,eps_i)_negative}.
\end{remark}

\begin{proof}[Proof of Theorem~\ref{theo:cones_angles_orthoc_simplices_outside}]
(a) follows by the definition of the tangent cone and $N(S,S)=T(S,S)^{\circ}$.

\vspace*{2mm}
\noindent
(b) and (c) follow by applying Theorem~\ref{theo:cones_of_orthoc_simpl_nondegenerate_case}. To see this, put  $v_0=w$ and  $v_i=\frac{e_i}{\tau_i}$ for $i=1,\ldots, d$. Then, $S=[v_0, v_1, \ldots, v_d]$. By Example~\ref{example:orthocenter_outside}, $\frac{e_0}{\tau_0}$ is the orthocenter of $S$ and
	$
	\lan v_i-\frac{e_0}{\tau_0}, v_j-\frac{e_0}{\tau_0} \ran
	=
	1/\tau_0^2
	=:
	c
	$
	for all $0 \leq i \neq j \leq d$. Let  $\mu_i = 1/(\|v_i-\frac{e_0}{\tau_0}\|^2-c)$ for $i\in \{0, \ldots, d\}$. It is a computation to see that
	$$
	\mu_0 = -\tau_0^2-\tau_1^2 - \ldots - \tau_d^2 <0
	\qquad \text{ and } \qquad
	\mu_i=\tau_i^2 \text{ for } i \in \{1, \ldots, d\}.
	$$
	Inserting these values into  Theorem~\ref{theo:cones_of_orthoc_simpl_nondegenerate_case} yields (b) and (c).

\vspace*{2mm}
\noindent
For (d) and (e), note that $\beta(F,S)=\alpha(T(F,S))$ and $\gamma(F,S)=\alpha(N(F,S))$.

\vspace*{2mm}
\noindent
(d) follows from (b) and Theorem~\ref{theo:angle_pos(eps_i_v_i)_via_g_d(lambda_i,eps_i)}.

\vspace*{2mm}
\noindent
It remains to prove (e). For $F=S$ (i.e.\ $k=d$), this follows from (a) and the definition of $\bg_0$. For $k\leq d-1$, the assertions follow from (c) and Theorem~\ref{theo:angle_pos(eps_i_v_i)_via_g_d(lambda_i,eps_i)}.
\end{proof}

\subsection{Examples for orthocentric cones}
In the following examples, we present concrete realizations of orthocentric cones, covering every admissible combination of parameters \( \lambda_0, \ldots, \lambda_d \).  Recall that  $e_0,\ldots, e_d$ is the standard orthonormal basis in $\R^{d+1}$.

\begin{example}
If $\lambda_0,\lambda_1,\ldots,\lambda_d>0$, then the vectors $v_i := \frac{e_i}{\sqrt{\lambda_i} } -  \frac{e_0}{\sqrt{\lambda_0}}\in \R^{d+1}$, with $i=1,\ldots, d$, satisfy~\eqref{eq:v_i_scalar_prod} and hence $\pos (v_1,\ldots, v_d)$ provides a realization of  $C_d(\lambda_0; \lambda_1,\ldots, \lambda_d; 1,\ldots, 1)$.
\end{example}

\begin{example}
The nonnegative orthant $\pos (e_1,\ldots, e_d)$ can be considered as a limiting case of $C_d(\lambda_0; \lambda_1,\ldots, \lambda_d; 1,\ldots, 1)$ when $\lambda_0\to \pm \infty$.
\end{example}

\begin{example}
Let $\lambda_1,\ldots,\lambda_d>0$ and $\lambda_0<-(\lambda_1 + \ldots + \lambda_d)$. Consider the vectors
$$
v_i
:=
\frac{\sqrt{-(\lambda_0 + \lambda_1 + \ldots + \lambda_d)}}{\lambda_0} e_0 + \frac{\sqrt{\lambda_1} e_1 + \ldots + \sqrt{\lambda_d} e_d}{\lambda_0} + \frac{e_i}{\sqrt{\lambda_i}} \in \R^{d+1},
$$
with $i=1, \ldots, d$. Then, $\pos (v_1,\ldots, v_d)$ is isometric to  $C_d(\lambda_0; \lambda_1,\ldots, \lambda_d; 1,\ldots, 1)$. Instead of verifying~\eqref{eq:v_i_scalar_prod} directly, we put $\tau_1:= \sqrt{\lambda_1}, \ldots,\tau_d:= \sqrt{\lambda_d}$, $\tau_0:= \sqrt{-(\lambda_0 + \lambda_1 + \ldots + \lambda_d)}$ and observe that  the tangent cone of the obtuse simplex $S$ as defined in~\eqref{eq:orthocentric_simpl_obtuse_canonical} at its vertex $w$  is given by  $T(\{w\},S) = \pos (\frac{e_1}{\tau_1} - w,\ldots,\frac{e_d}{\tau_d} - w)= \pos (v_1,\ldots, v_d)$. By Theorem~\ref{theo:cones_angles_orthoc_simplices_outside} (c), this cone is isometric to $C_{d}(-\tau_0^2-\tau_{1}^2-\ldots -\tau_d^2; \tau_{1}^2,\ldots, \tau_d^2; 1,\ldots, 1 )$.
\end{example}

\begin{example}
Let $\lambda_0,\lambda_2,\ldots,\lambda_{d}>0$,  $\lambda_1<0$ and  $\lambda_0 + \ldots + \lambda_{d}<0$. In $\R^{d+1}$ we consider the vectors $v_i := \frac{e_i}{\sqrt{\lambda_i} } -  \frac{e_1}{\sqrt{\lambda_0}}$,
$i=2,\ldots, d$, as well as
$$
v_1
:=
\frac{\sqrt{-(\lambda_0 + \lambda_1 + \ldots + \lambda_{d})}e_0 + \sqrt{\lambda_0} e_1 + \sqrt{\lambda_2} e_2 +  \ldots + \sqrt{\lambda_{d}} e_{d}}{-\lambda_1} - \frac{e_1}{\sqrt{\lambda_0}}.
$$
Then, $\pos (v_1,\ldots, v_d)$ is isometric to  $C_d(\lambda_0; \lambda_1,\ldots, \lambda_d; 1,\ldots, 1)$. To see this, we put $\tau_1:= \sqrt{\lambda_0}, \tau_2:= \sqrt{\lambda_2}, \ldots,\tau_d:= \sqrt{\lambda_d}$, $\tau_0:= \sqrt{-(\lambda_0 + \lambda_1 + \ldots + \lambda_d)}$ and observe that  the tangent cone of the obtuse simplex $S$ as defined in~\eqref{eq:orthocentric_simpl_obtuse_canonical} at its vertex $\frac{e_1}{\tau_1}$  is given by  $T(\{\frac{e_1}{\tau_1}\},S) = \pos (w-\frac{e_1}{\tau_1}, \frac{e_2}{\tau_2} - \frac{e_1}{\tau_1},\ldots,\frac{e_d}{\tau_d} - \frac{e_1}{\tau_1})= \pos (v_1,\ldots, v_d)$. By Theorem~\ref{theo:cones_angles_orthoc_simplices_outside} (b), this cone is isometric to $C_{d}(\tau_1^2; -\tau_0^2-\ldots -\tau_d^2, \tau_{2}^2,\ldots, \tau_d^2; 1,\ldots, 1 )$.
\end{example}

\section{Applications to Gaussian polytopes}

Let \( g_1, \ldots, g_n \) be independent standard Gaussian vectors in \( \mathbb{R}^d \), and let \( \tau_1, \ldots, \tau_n > 0 \) be constants. In this section, we compute the expected number of \( k \)-dimensional faces and the expected volume of the Gaussian random polytope \([g_1/\tau_1, \ldots, g_n/\tau_n]\) in terms of the function \( \bg_d \).

\subsection{Random projections of simplices}
The expected $f$-vector of the random polytope that arises as the image of an original polytope $P$ under a standard Gaussian matrix can be expressed in terms of the internal and external angles of $P$, as follows.
\begin{theorem}
	Let $P\subseteq \R^N$ be a polytope and $d\in \{1, \ldots, \dim P \}$. Let $G$ be a random $d\times N$ matrix with entries that are independent and standard Gaussian distributed. Then, the expected number of $k$-dimensional faces of the random polytope $GP = \{ Gx:x\in P\}\subseteq \R^d$ is given by
	\begin{equation}\label{eq:AffSchn_similar}
		\E f_k(GP) = 2 \sum_{s=0}^\infty \sum_{H\in \cF_{d-1-2s}(P)} \gamma(H,P) \sum_{F\in \cF_k(H)} \beta(F,H)
	\end{equation}
	for all $k\in \{0, \ldots, \dim P\}$.

	We emphasize that the sum on the right-hand side of \eqref{eq:AffSchn_similar} terminates after finitely many non-zero terms.
\end{theorem}

We refer to~\cite{AS92} for a proof of a similar result, where a uniform projection instead of a standard Gaussian projection $G$ is considered. The same proof works for any random linear map, whose kernel is uniformly distributed on the Grassmanian of all $(N-d)$-dimensional linear subspaces of $\R^N$.

\subsection{Intrinsic volumes and Tsirelson's theorem}
The intrinsic volumes $V_0(A), \ldots, V_d(A)$ of a compact convex set $A \subseteq \R^d$  are defined via the Steiner formula
$$
\Vol_d(A+r\BB^d)
=
\sum_{k=0}^d \kappa_{d-k} V_k(A) r^{d-k}, \qquad r>0.
$$
For a polytope $P\subseteq \R^d$, the  intrinsic volumes are given by
\begin{equation}\label{eq:intrinsic_volume_polytope}
	V_k(P)
	=
	\sum_{F\in \cF_k(P)} \Vol_k(F) \gamma(F,P);
\end{equation}
see~\cite[Equation (14.14)]{SW08}.
The expected volume of the Gaussian projection of a convex compact set can be expressed in terms of its intrinsic volumes. The next result was originally established by Tsirelson~\cite{tsirelson_2}, with a more accessible proof later provided by Vitale~\cite{Vitale08}.
\begin{theorem}[Tsirelson]\label{theo:Tsirelson}
	Let $K \subseteq \R^n$ be convex and compact. For $d\in \{1, \ldots, n\}$ let $G$ denote a $d \times n$ matrix with independent standard Gaussian distributed entries. The expected volume of
	$G K=\{ Gx:x\in K \}\subseteq \R^d$ is given by
	$$
	\E \Vol_d(GK)
	=
	\frac{2^{d/2} \Gamma \left( \frac{d+1}{2} \right)}{\sqrt{\pi}} V_d(K).
	$$
\end{theorem}

\subsection{Results on  Gaussian polytopes}
Let $g_1,\dots, g_n$ be independent $\R^d$-valued standard Gaussian random vectors and $\tau_1, \ldots, \tau_d>0$. In the following, we shall compute expected functionals of Gaussian polytopes of the form $[g_1/\tau_1,\ldots, g_n/\tau_n]\subseteq \R^d$. The idea is to view them as Gaussian projections of the simplex $T=[e_1/\tau_1, \ldots, e_n/\tau_n]\subseteq \R^n$, as in Example~\ref{example:e_i/tau_i}, but with slightly different indexing. From Theorem~\ref{theo:orthocentr_simpl_relint_tangent_normal_angles}, recall that the internal and external angles of this simplex at its face $F=[e_1/\tau_1, \ldots, e_k/\tau_k]$ are given by
\begin{align}
\beta(F,T)
=&
\bg_{n-k}(-\tau_1^2-\ldots-\tau_n^2;\tau_{k+1}^2,\dots,\tau_n^2; 1, \ldots, 1),\label{eq:gauss_poly_proof_beta}\\
\gamma(F,T)
=&
\bg_{n-k}(\tau_1^2+\ldots+\tau_k^2;\tau_{k+1}^2,\dots,\tau_n^2; 1,\ldots ,1).\label{eq:gauss_poly_proof_gamma}
\end{align}

\begin{theorem}\label{theo:exp_face_number_gaussian}
For $d \geq 1$ and $n\geq d+1$, let $\tau_1,\dots, \tau_n>0$ and $g_1,\dots, g_n$ be independent $\R^d$-valued standard Gaussian random vectors. Then, for $k\in \{0,\dots,d-1\}$ it holds that
\begin{align*}
\E f_k \left(\left[\frac{g_1}{\tau_1},\ldots, \frac{g_n}{\tau_n}\right] \right)
=
2\sum_{s=0}^{\infty} \sum_{1\leq i_1<\ldots <i_{d-2s}\leq n} \bg_{n-d+2s} (\tau_{i_1}^2 + \ldots + \tau_{i_{d-2s}}^2;\tau_{j_1}^2,\dots, \tau_{j_{n-d-2s}}^2) \times\\
\times \sum_{1\leq m_1 <\ldots <m_{d-2s-k-1}\leq d-2s}
\bg_{d-2s-1-k}(-\tau_{i_1}^2 - \ldots - \tau_{i_{d-2s}}^2;\tau_{i_{m_1}}^2,\ldots, \tau_{i_{m_{d-2s-k-1}}}^2),
\end{align*}
where, in the second sum, $j_1, \ldots, j_{n-d+2s}$ are defined via $\{i_1,\dots, i_{d-2s}\} \cup \{j_1,\dots, j_{n-d+2s}\} = \{1,\dots, n\}$.
\end{theorem}

\begin{proof}
Let $G$ be a $d\times n$ matrix with independent standard Gaussian distributed entries and $T=[e_1/ \tau_1, \ldots, e_n/\tau_n]$. Then, we have that $GT \eqdistr [g_1/\tau_1,\ldots, g_n/\tau_n]$. Hence,~\eqref{eq:AffSchn_similar} entails for $k\in \{0,\dots, d-1\}$ that
$$
\E f_k \left(\left[\frac{g_1}{\tau_1},\ldots, \frac{g_n}{\tau_n}\right] \right) = \E f_k(GT) = 2 \sum_{s=0}^\infty \sum_{H\in \cF_{d-1-2s}(T)} \gamma(H,T) \sum_{F\in \cF_k(H)} \beta(F,H).
$$
For the second sum, note that $H\in \cF_{d-1-2s}(T)$ if and only if there exist $1\leq i_1<\ldots <i_{d-2s}\leq n$ with $H=[e_{i_1}/\tau_{i_1}, \dots, e_{i_{d-2s}}/\tau_{i_{d-2s}}]$.
Similarly, $F$ is  a $k$-dimensional face of $H$ if and only if there exist $1\leq l_1<\ldots <   l_{k+1}\leq d-2s$ such that $F=[e_{i_{l_1}}/\tau_{i_{l_1}}, \dots, e_{i_{l_{k+1}}}/\tau_{i_{l_{k+1}}}]$.

The corresponding angles $\gamma(H, T)$ and $\beta(F, H)$ are taken from~\eqref{eq:gauss_poly_proof_gamma} and~\eqref{eq:gauss_poly_proof_beta}. Finally, the inner sum is taken over all possibilities for $\{m_1,\dots, m_{d-2s-k-1}\}=\{i_1,\dots, i_{d-2s}\} \setminus \{l_1, \dots, l_{k+1}\}$ instead of all possibilities for $\{l_1, \dots, l_{k+1}\}$.
\end{proof}

A closely related work is~\cite{kabluchko_steigenberger_beta_poly_beta_cones}, which computes the expected $f$-vectors of the so-called beta polytopes. As the beta parameters tend to infinity, these expected $f$-vectors converge to those of the Gaussian polytopes considered in Theorem~\ref{theo:exp_face_number_gaussian}.

\begin{theorem}\label{theo:intrinsic_vol_e_i/tau_i}
Let $\tau_0, \ldots, \tau_d >0$. For $k\in \{0, \ldots, d\}$, the $k$-th intrinsic volume of the simplex $S=[e_0/\tau_0, \ldots, e_d/\tau_d]$ is given by
$$
V_k(S)
=
\frac{1}{k!} \sum_{0 \leq i_0 < \ldots < i_k \leq d} \frac{\sqrt{\tau_{i_0}^2 + \ldots + \tau_{i_k}^2}}{\tau_{i_0} \cdots \tau_{i_k}}
\bg_{d-k} \left( \tau_{i_0}^2 + \ldots + \tau_{i_k}^2; \tau_{j_1}^2, \ldots, \tau_{j_{d-k}}^2; 1, \ldots, 1 \right),
$$
where $j_1, \ldots, j_{d-k}$ are defined via $\{i_0,\dots, i_k\} \cup \{j_1,\dots, j_{d-k}\} = \{0,\dots, d\}$.
\end{theorem}

\begin{remark}
An explicit representation of the function $\bg_{d-k}$ appearing in Theorem~\ref{theo:intrinsic_vol_e_i/tau_i} will be proven in Theorem~\ref{theo:formula_for_bg_d(lambda_i,eps_i)}. Together, these two results recover~\cite[Theorem 3.1]{kabluchko_zaporozhets_expected_volumes_18}.
\end{remark}

\begin{proof}[Proof of Theorem~\ref{theo:intrinsic_vol_e_i/tau_i}]
We use~\eqref{eq:intrinsic_volume_polytope}. To apply this formula, note that the set of $k$-dimensional faces of $S$ is
$$
\cF_k(S)
=
\left\{ \left[\frac{e_{i_0}}{\tau_{i_0}}, \ldots, \frac{e_{i_k}}{\tau_{i_k}} \right]: 0 \leq i_0 < \ldots < i_k \leq d \right\}.
$$
The external angles of $S$ at these faces are given by Theorem~\ref{theo:orthocentr_simpl_relint_tangent_normal_angles}(d). Finally, the volume of such a simplex is
$$
\Vol_k\left( \left[\frac{e_{i_0}}{\tau_{i_0}}, \ldots, \frac{e_{i_k}}{\tau_{i_k}} \right] \right)
=
\frac{\sqrt{\tau_{i_0}^2 + \ldots + \tau_{i_k}^2}}{k! \cdot \tau_{i_0} \cdots \tau_{i_k}};
$$
see, e.g., \cite[Theorem~5.1, Eqn.~(52)]{Edmonds2005}. Plugging everything into~\eqref{eq:intrinsic_volume_polytope} concludes the proof.
\end{proof}

\begin{theorem}\label{theo:exp_volume_gaussian_simpl}
For $d\geq 1$ and $n\geq d+1$, let $\tau_1, \ldots, \tau_n>0$ and $g_1, \ldots, g_n$ be independent $\R^d$-valued standard Gaussian random vectors. Then,
\begin{align*}
&\E \Vol_d \left( \left[ g_1/\tau_1, \ldots, g_n/\tau_n \right] \right)\\
&=
\frac{2^{d/2} \Gamma \left( \frac{d+1}{2} \right)}{\sqrt{\pi} d!}
\sum_{1 \leq i_1 < \ldots < i_{d+1} \leq n} \frac{\sqrt{\tau_{i_1}^2 + \ldots + \tau_{i_{d+1}}^2}}{\tau_{i_1} \cdots \tau_{i_{d+1}}}
\bg_{n-d-1} \left( \tau_{i_1}^2 + \ldots + \tau_{i_{d+1}}^2; \tau_{j_1}^2, \ldots, \tau_{j_{n-d-1}}^2; 1,\ldots, 1 \right),
\end{align*}
where $j_1, \ldots, j_{n-d-1}$ are defined via $\{i_1,\dots, i_{d+1}\} \cup \{j_1,\dots, j_{n-d-1}\} = \{1,\dots, n\}$.
\end{theorem}

\begin{proof}
Write $T=[e_1/\tau_1, \ldots, e_n/\tau_n]$ and let $G$ be a $(d \times n)$-matrix with independent standard Gaussian distributed entries. Then, we have $GT$ has the same distribution as $[g_1/\tau_1, \ldots, g_n/\tau_n]$. The intrinsic volumes of $T$ are given, with a slightly different notation, in Theorem~\ref{theo:intrinsic_vol_e_i/tau_i}. Plugging them into Theorem~\ref{theo:Tsirelson} yields the result.
\end{proof}

\begin{remark}
Theorems~\ref{theo:exp_face_number_gaussian} and~\ref{theo:exp_volume_gaussian_simpl} employ the angles of a simplex $T= [e_1/\tau_1,\ldots, e_n/\tau_n]$ (of the same type as in Example~\ref{example:e_i/tau_i}) to compute the expected number of faces and the expected volume of the random polytope $GT$, where $G$ is a random Gaussian matrix. This method extends naturally to the simplices considered in Examples~\ref{example:degenerate_orthoc_simplex}, \ref{example:orthocenter_outside} and~\ref{exam:orthocenter_outside_alternative}, leading to explicit formulas for the expected number of faces and the expected volume of the random polytopes
$$
\left[ 0, \frac{g_1}{\tau_1}, \ldots, \frac{g_n}{\tau_n} \right],
\qquad
\left[\frac{\tau_0 g_0 + \ldots + \tau_n g_n}{\tau_0^2 + \ldots + \tau_n^2}, \frac{g_1}{\tau_1}, \ldots, \frac{g_n}{\tau_n} \right],
\qquad
\left[ \nu \frac{\tau_1 g_1 + \ldots + \tau_n g_n}{\tau_1^2 + \ldots + \tau_n^2}, \frac{g_1}{\tau_1}, \ldots, \frac{g_n}{\tau_n} \right],
$$
where $\tau_0, \tau_1, \ldots, \tau_n >0$, $\nu\neq 1$, and $g_0, g_1, \ldots, g_n$ are independent standard $d$-dimensional Gaussian distributed random vectors, $n\geq d$.
\end{remark}

\section{Explicit formulas for  the function \texorpdfstring{$\bg_d$}{g\_d}}  \label{sec:explicit_formula_g_d}
This section is devoted to deriving explicit formulas for the angles of oblique orthocentric simplices. These angles, along with various other quantities, have already been expressed in terms of the function \( \bg_d \). Therefore, our primary goal is to obtain explicit formulas for \( \bg_d \). The case of rectangular simplices, which requires a separate treatment, will be discussed in Section~\ref{sec:right_angled_simplices}.

\subsection{Preparation}\label{subsec:preparation_orthant_probab_for_orthocentric_simpl}
The following lemma constitutes a key step toward deriving an explicit formula for the function $\bg_d(\lambda_0; \lambda_1,\ldots, \lambda_d; \eps_1,\ldots, \eps_d)$ in the special case where $\lambda_1>0,\ldots, \lambda_d>0$.

\begin{lemma}[An orthant probability]\label{lemma:proba_gaussian_greater_const}
Let $d\in \N$, $\lambda_1, \ldots, \lambda_d>0$  and $r>-1/(\lambda_1+\ldots +\lambda_d)$. Let $(\eta_1, \ldots, \eta_d)$ be multivariate Gaussian distributed with mean $0$ and covariance matrix
\begin{equation}\label{eq:covariance_matrix_r}
\Sigma
:=
\left( r+ \frac{\delta_{ij}}{\lambda_i} \right)_{i,j=1}^d.
\end{equation}
Then, for all $t_1, \ldots, t_d \in~\R$ and $\eps_1, \ldots, \eps_d \in \{\pm 1\}$ we have
\begin{align}
&\P\left[ \eps_1\eta_1 \geq t_1, \ldots, \eps_d\eta_d \geq t_d\right] \label{eq:orthant_probab_for_g_d_0}\\
&=
\frac{1}{\sqrt{2 \pi}} \int_0^{\infty} \left( \prod_{j=1}^d \Phi \left((\eps_j \sqrt{r}x -t_j)\sqrt{\lambda_j} \right) + \prod_{j=1}^d \Phi \left((-\eps_j \sqrt{r}x -t_j)\sqrt{\lambda_j} \right) \right) \eee^{-x^2/2} \dd x \label{eq:orthant_probab_for_g_d_1}\\
&=
\frac{1}{\sqrt{2 \pi}} \int_{-\infty}^{\infty}  \prod_{j=1}^d \Phi \left((\eps_j \sqrt{r}x -t_j)\sqrt{\lambda_j} \right) \eee^{-x^2/2} \dd x. \label{eq:orthant_probab_for_g_d_2}
\end{align}
Here, $\Phi:\CC\to \CC$ denotes the  distribution function of the standard normal distribution, as defined in~\eqref{eq:compl_std_normal_distr}. For real $r<0$ we use the convention $\sqrt{r}=i\sqrt{-r}$.

Moreover, if we allow $r$ to be complex, the expression in~\eqref{eq:orthant_probab_for_g_d_1}  does not depend on the choice of the sign of $\sqrt r$      and represents an analytic function of $r$ on the half-plane
$$
\mathcal D := \{r\in \CC: \Re r > -1 / (\lambda_1+\ldots + \lambda_d)\}.
$$
\end{lemma}

\begin{remark}
The proof of Lemma~\ref{lemma:proba_gaussian_greater_const} given below builds on the method used in the proof of~\cite[Proposition~1.4]{kabluchko_zaporozhets_absorption}, which corresponds to the special case \(\lambda_1 = \ldots = \lambda_d\). Earlier versions of this method appeared in~\cite[Section~4]{rogers}, \cite[Lemma~4]{vershik_sporyshev_asymptotic_faces_random_polyhedra1992}, and~\cite{steck_orthant_equicorrelated,steck_owen_equicorr_multivar_normal}. A slightly less general version of Lemma~\ref{lemma:proba_gaussian_greater_const}, assuming \(t_1 = \ldots = t_d = 0\), appeared recently in the PhD thesis of~\citet[Theorem~2.B.1, p.~112]{kuchelmeister_phd}, where the same method from~\cite{kabluchko_zaporozhets_absorption} was employed.  \end{remark}

\begin{proof}[Proof of Lemma~\ref{lemma:proba_gaussian_greater_const}]
Note that \( \Sigma \) is (strictly) positive definite and therefore represents a valid covariance matrix.
Indeed, if $r\neq 0$, this follows from Lemma~\ref{lemma:condition_for_ex_of_orthocentric_cone}, while for $r=0$ the matrix $\Sigma$ is diagonal with positive diagonal entries.
The equality of~\eqref{eq:orthant_probab_for_g_d_1} and~\eqref{eq:orthant_probab_for_g_d_2} follows by  splitting the integral in~\eqref{eq:orthant_probab_for_g_d_2} into two integrals over the intervals $(-\infty, 0]$ and $(0, \infty)$ and performing a change of variables $x\mapsto -x$ in the first integral.

In the following, we prove the equality between~\eqref{eq:orthant_probab_for_g_d_0} and~\eqref{eq:orthant_probab_for_g_d_1}.   Fix $\lambda_1, \ldots, \lambda_d > 0$, $\eps_1, \ldots ,\eps_d \in \{\pm 1\}$ and $t_1, \ldots, t_d\in \R$.

\underline{Step 1:} Suppose $r>0$.  Let $\xi, \xi_1, \ldots, \xi_d\sim \mathrm{N}(0,1)$ be independent and identically distributed. Consider
$$
\begin{pmatrix}
\tilde{\eta}_1 \\ \vdots \\ \tilde{\eta}_d
\end{pmatrix}
:=
\begin{pmatrix}
\frac{1}{\sqrt{\lambda_1}} & 0 & \cdots & 0 & -\sqrt{r}\\
0 & \frac{1}{\sqrt{\lambda_2}} & \cdots & 0 & -\sqrt{r}\\
\vdots & \ddots & \ddots & \vdots & \vdots\\
0 & 0 & \cdots & \frac{1}{\sqrt{\lambda_d}} & -\sqrt{r}
\end{pmatrix}
\cdot
\begin{pmatrix}
\xi_1 \\ \vdots \\ \xi_d \\ \xi
\end{pmatrix}
=
\begin{pmatrix}
\frac{1}{\sqrt{\lambda_1}}\xi_1- \sqrt{r} \xi \\
\vdots \\
\frac{1}{\sqrt{\lambda_d}}\xi_d- \sqrt{r}\xi
\end{pmatrix}.
$$
Then, $(\tilde{\eta}_1, \ldots, \tilde{\eta}_d)$ is a zero-mean Gaussian random vector with covariance matrix $\Sigma$. Hence,
$$
(\eta_1, \ldots, \eta_d)
\stackrel{d}{=}
\left(\frac{1}{\sqrt{\lambda_1}}\xi_1- \sqrt{r} \xi,\ldots, \frac{1}{\sqrt{\lambda_d}}\xi_d- \sqrt{r}\xi \right),
$$
since both sides are multivariate Gaussian distributed and their expectations and covariance matrices agree.
It follows  that
\begin{align}
&\P[\eps_1\eta_1\geq t_1, \ldots, \eps_d\eta_d \geq t_d]
=
\P[\eps_1(-\eta_1)\geq t_1, \ldots, \eps_d(-\eta_d) \geq t_d]\notag\\
=&
\P\left[\forall j \in \{1, \ldots, d\}: \eta_j \leq -t_j \text{ if } \eps_j=1 \text{ and } \eta_j \geq t_j \text{ if } \eps_j=-1 \right]\notag\\
=&
\P\left[\forall j \in \{1, \ldots, d\}: \xi_j \leq (\sqrt{r} \xi - t_j) \sqrt{\lambda_j} \text{ if } \eps_j=1 \text{ and }  \xi_j \geq (\sqrt{r} \xi + t_j) \sqrt{\lambda_j} \text{ if } \eps_j=-1 \right]\notag\\
=&
\frac{1}{\sqrt{2\pi}} \int_{-\infty}^{\infty} \prod_{\substack{j=1\\ \eps_j=1}}^d \P\left[\xi_j \leq (\sqrt{r}x - t_j) \sqrt{\lambda_j}\right] \cdot \prod_{\substack{j=1\\ \eps_j=-1}}^d \P\left[\xi_j \geq (\sqrt{r}x + t_j) \sqrt{\lambda_j}\right] \eee^{-x^2/2}\dd x\notag\\
=&
\frac{1}{\sqrt{2\pi}} \int_{-\infty}^{\infty} \prod_{j=1}^d \Phi\left(( \eps_j \sqrt{r}x - t_j) \sqrt{\lambda_j}\right) \eee^{-x^2/2} \dd x\notag\\
=&
\frac{1}{\sqrt{2 \pi}} \int_0^{\infty} \left( \prod_{j=1}^d \Phi \left((\eps_j \sqrt{r}x -t_j)\sqrt{\lambda_j} \right) + \prod_{j=1}^d \Phi \left((-\eps_j \sqrt{r}x -t_j)\sqrt{\lambda_j} \right) \right) \eee^{-x^2/2} \dd x
\label{eq:proof_g_d_formula_def_F_1} \\
=&: F_1(r). \notag
\end{align}
We first used the fact that \( (\eta_1, \ldots, \eta_d) \overset{d}{=} (-\eta_1, \ldots, -\eta_d) \), and then applied the total probability formula by conditioning on \( \xi = x \). Also,  we applied the identity \( 1 - \Phi(z) = \Phi(-z) \).

\underline{Step 2:} Now we consider the  general case $r>-1/(\lambda_1 +\ldots +\lambda_d)$.
By~\eqref{eq:det_G} we obtain
$$
\det \Sigma
=
\frac{1/r + \lambda_1 + \ldots + \lambda_d}{1/r \cdot \lambda_1 \cdots \lambda_d}
=
\frac{1+r(\lambda_1 + \ldots + \lambda_d)}{\lambda_1 \cdots \lambda_d},
$$
first for $r\neq 0$, then by continuity also for $r=0$. Note that $\det \Sigma \neq 0$.
The inverse of $\Sigma$ is
$$
\Sigma^{-1}
=
\left( -\frac{r \lambda_i \lambda_j}{1+ r(\lambda_1 + \ldots + \lambda_d)} + \delta_{ij}\lambda_i \right)_{i,j=1}^d.
$$
This can be verified by calculating the product of this matrix with  $\Sigma$. In particular, the density of $(\eta_1, \ldots, \eta_d)$ is given by
$$
f(x)
=
f(x_1, \ldots, x_d)
=
\frac{1}{(2 \pi)^{d/2} \sqrt{\det \Sigma}} \exp \left\{ -\frac{1}{2} \lan x, \Sigma^{-1} x \ran \right\}, \qquad x\in \R^d.
$$
The form of $\Sigma^{-1}$ implies
$$
\lan x, \Sigma^{-1} x \ran
=
-\frac{r(\lambda_1 x_1 + \ldots + \lambda_d x_d)^2}{1+ r(\lambda_1 + \ldots + \lambda_d)} + \lambda_1 x_1^2 + \ldots + \lambda_d x_d^2.
$$
This yields
\begin{align}
\P[&\eps_1\eta_1 \geq t_1, \ldots, \eps_d\eta_d \geq t_1]
=
\int_{(t_1, \infty) \times \cdots \times (t_d, \infty)} f(\eps_1x_1,\ldots, \eps_dx_d) \dd x_1 \cdots \dd x_d \notag\\
=&
\frac{1}{(2\pi)^{d/2}} \sqrt{\frac{\lambda_1 \cdots \lambda_d}{1+r(\lambda_1+\ldots+\lambda_d)}} \times \notag\\
\times&
\int_{(t_1, \infty) \times \cdots \times (t_d, \infty)} \exp \left\{ \frac{1}{2} \frac{r}{1+r(\lambda_1 + \ldots + \lambda_d)}\left(\sum_{j=1}^d \eps_j \lambda_j x_j\right)^2 - \frac{1}{2} \sum_{j=1}^d \lambda_j x_j^2\right\} \dd x_1 \cdots \dd x_d
\label{eq:proof_g_d_formula_def_F_2}\\
=&: F_2(r) \notag.
\end{align}

\underline{Step 3:}   Next  we prove that $F_1(r)$, as defined in~\eqref{eq:proof_g_d_formula_def_F_1}, is analytic on the domain
$$
\mathcal D = \left\{r\in \CC: \Re r > - 1/(\lambda_1 +\ldots + \lambda_d) \right\}.
$$
Recall that $\lambda_1, \ldots, \lambda_d > 0$, $\eps_1, \ldots ,\eps_d \in \{\pm 1\}$ and $t_1, \ldots, t_d\in \R$ are fixed.

First, we verify that the integral defining $F_1(r)$ converges for all $r\in \mathcal D$. To this end, we prove: For every compact set $K\subseteq \mathcal D$ there exist constants $C_1(K)>0$ and $c(K)>0$ such that
\begin{equation}\label{eq:estimate_integrand_on_K}
\left| \left( \prod_{j=1}^d \Phi \left((\eps_j \sqrt{r}x -t_j)\sqrt{\lambda_j} \right) + \prod_{j=1}^d \Phi \left((-\eps_j \sqrt{r}x -t_j)\sqrt{\lambda_j} \right) \right) \eee^{-x^2/2} \right|
\leq  C_1(K) \eee^{-c(K)x^2}
\end{equation}
for all $x\geq 0$ and $r\in K$.
The following well-known bound for $\Phi$  can be found (without proof), e.g.,\ in~\cite[Proof of Proposition 1.4(a)]{kabluchko_zaporozhets_absorption}.
\begin{proposition}\label{prop:bound_Phi}
There exists a constant $C$ such that for all $z\in \CC$, we have the estimate
$$
|\Phi(z)| \leq C \max\{1, |\eee^{-z^2/2}|\}
=
C \max\{1, \eee^{-\Re(z^2)/2}\}.
$$
\end{proposition}
\begin{proof}
The equality is satisfied since  $|\eee^{a+ib}| = \eee^a$ for $a,b\in \R$. For the inequality, first note that there exists a constant $C_1$ such that $|\Phi(z)| \leq C_1$ for all $z\in \CC, |z|<1$. For the remaining domain, fix $\eps\in (0,\pi/4)$. From~\cite[Lemma 3.10]{kabluchkoklimovsky}, recall the asymptotic behavior
$$
\Phi(z)
=
\begin{cases}
	-\frac{1+o(1)}{\sqrt{2\pi}z} \eee^{-z^2/2}, & \text{ uniformly, as } |z|\rightarrow \infty, \text{ on } |\arg z|>\frac{\pi}{4}+\eps,\\
	1-\frac{1+o(1)}{\sqrt{2\pi}z} \eee^{-z^2/2}, & \text{ uniformly, as } |z|\rightarrow \infty, \text{ on } |\arg z|<\frac{3\pi}{4}+\eps,
\end{cases}
$$
where $\arg z \in (-\pi, \pi]$.
Hence, there exists $C_2$ such that $|\Phi(z)| \leq C_2 |\eee^{-z^2/2}|$ for all $z\in \CC, |z|\geq 1$ with $|\arg z| > \pi/4 +\eps$. Moreover, there exists $C_3>1$ such that for all $z\in \CC, |z|\geq 1$ with $|\arg z|\leq 3\pi/4-\eps$, we have
$$
|\Phi(z)|
\leq
1+C_3 |\eee^{-z^2/2}|
\leq 2C_3 \max\{1/C_3, |\eee^{-z^2/2}|\}.
$$
Hence, with $C=\max\{C_1,C_2, 2C_3\}$, the inequality is satisfied for all $z\in \CC$.
\end{proof}

Let  $\eps>0$. It follows from Proposition~\ref{prop:bound_Phi} that there exists a constant $C(K,\eps)>0$ such that
\begin{align*}
\left| \Phi \left( (\pm \eps_j \sqrt{r}x-t_j)\sqrt{\lambda_j} \right) \right|
\leq&
C \max \left\{ 1, \eee^{-\frac{1}{2}\left( \Re(r)x^2\lambda_j \mp 2 \lambda_j \eps_j \Re(\sqrt{r}) x t_j + t_j^2 \lambda_j \right)} \right\}\\
\leq&
C(K, \eps) \eee^{\eps x^2} \max \{1, \eee^{-\Re(r)x^2 \lambda_j/2} \}
\end{align*}
for all $r\in K$, $j=1, \ldots, d$ and $x \geq  0$. The second line follows from the estimate  $|\lambda_j \eps_j \Re(\sqrt{r}) x t_j|\leq \max\{\eps x^2, B(K,\eps)\}$ for some constant $B(K,\eps)>0$.  Note that \( \max \{ 1, \eee^{-\Re(r)x^2 \lambda_j/2}\} \) equals \( 1 \) when \( \Re r \geq 0 \), and \( \eee^{-\Re (r) x^2 \lambda_j/2} \) when \( \Re r \leq 0 \). It follows that
\begin{align*}
&\left| \left( \prod_{j=1}^d \Phi \left((\eps_j \sqrt{r}x -t_j)\sqrt{\lambda_j} \right) + \prod_{j=1}^d \Phi \left((-\eps_j \sqrt{r}x -t_j)\sqrt{\lambda_j} \right) \right) \eee^{-x^2/2} \right|\\
\leq
&2C(K, \eps)^d  \eee^{\eps d x^2} \max \left\{ \eee^{-x^2/2}, \exp\left(-\frac{1}{2}((\lambda_1 +\ldots + \lambda_d)\Re(r)+1)x^2 \right) \right\} = 2C(K, \eps)^d \eee^{-cx^2}
\end{align*}
with $c= \inf_{r\in K} \frac{1}{2} \min\{1,(\lambda_1+\ldots +\lambda_d)\Re(r)+1 \} -\eps d$. Note that for every compact  $K\subseteq \mathcal D$ we can choose $\eps>0$ such that $c>0$.   This proves~\eqref{eq:estimate_integrand_on_K}.

An immediate consequence of~\eqref{eq:estimate_integrand_on_K} is that
the integral in~\eqref{eq:proof_g_d_formula_def_F_1} defining $F_1(r)$ converges.
In what follows, we work toward establishing the analyticity of \( F_1(r) \) as a function of $r\in \mathcal D$.
Let us first prove that the integrand in~\eqref{eq:proof_g_d_formula_def_F_1} is analytic in $r\in \CC$. The function $\Phi$ can be written as an everywhere absolutely convergent power series~\eqref{eq:taylor_series_Phi}. Hence, for fixed $x>0$, the expression
$$
\prod_{j=1}^d \Phi \left((\eps_j z x -t_j)\sqrt{\lambda_j} \right)
=
\sum_{n=0}^{\infty} a_n z^n
$$
can also be written as an absolutely convergent power series in $z$. Consequently,
$$
\prod_{j=1}^d \Phi \left((\eps_j \sqrt{r}x -t_j)\sqrt{\lambda_j} \right) + \prod_{j=1}^d \Phi \left((-\eps_j \sqrt{r}x -t_j)\sqrt{\lambda_j} \right)
=
\sum_{n=0}^{\infty} 2a_{2n} r^n
$$
is an analytic function in $r\in \CC$.

Next, we prove that $F_1(r)$ is continuous on $\mathcal D$. Let $r, r_1, r_2,\ldots\in \mathcal D$ be such that $\lim_{n\to\infty}r_n =r$. Then, by~\eqref{eq:estimate_integrand_on_K}, there exists $C_1>0$ and $c>0$  such that
$$
\left| \left( \prod_{j=1}^d \Phi \left((\eps_j \sqrt{r_n}x -t_j)\sqrt{\lambda_j} \right) + \prod_{j=1}^d \Phi \left((-\eps_j \sqrt{r_n}x -t_j)\sqrt{\lambda_j} \right) \right) \eee^{-x^2/2} \right|
\leq
C_1\eee^{-cx^2}
$$
for all $n \in \N$ and all $x\geq 0$.
Since the integral of \( C\eee^{-c x^2} \) over \( (0, \infty) \) is finite, the dominated convergence theorem implies that \( \lim_{n \to \infty} F_1(r_n) = F_1(r) \), thereby establishing the continuity of \( F_1(r) \) on $\mathcal D$.

To prove that \( F_1(r) \) is analytic on $\mathcal D$, we apply Morera's theorem. It is sufficient to verify that the contour  integral of \( F_1(r) \) over every triangular path vanishes. Take $z_1,z_2,z_3\in \mathcal D$, i.e.\  $\Re z_i >-1/(\lambda_1 +\ldots + \lambda_d), i=1,2,3$. Let $\lan z_1, z_2, z_3\ran$ be the triangular contour $\alpha_1 \oplus \alpha_2 \oplus \alpha_3$ with $z_1, z_2, z_3$ as vertices. That is,
\begin{align*}
	\alpha = \alpha_1 \oplus \alpha_2 \oplus \alpha_3 = \begin{cases}
		\alpha_1(t) & \text{for } 0\leq t \leq 1,\\
		\alpha_2(t) & \text{for } 1\leq t \leq 2,\\
		\alpha_3(t) & \text{for } 2\leq t \leq 3,
	\end{cases}
\end{align*}
with
\begin{align*}
	\alpha_1(t) = z_1 + t(z_2-z_1) \quad \text{for }  0 \leqslant t \leqslant 1,\\
	\alpha_2(t) = z_2 + (t-1)(z_3-z_2) \quad \text{for } 1\leqslant t \leqslant 2,\\
	\alpha_3(t) = z_3 + (t-2)(z_1-z_3) \quad \text{for } 2\leqslant t \leqslant 3.
\end{align*}

By~\eqref{eq:estimate_integrand_on_K}, there exist $C_1>0$ and $c>0$ such that for all $r\in \alpha_1$,
\begin{align*}
&\int_0^1 \int_{0}^{\infty}\left| \left( \prod_{j=1}^d \Phi \left((\eps_j \sqrt{r}x -t_j)\sqrt{\lambda_j} \right) + \prod_{j=1}^d \Phi \left((-\eps_j \sqrt{r}x -t_j)\sqrt{\lambda_j} \right)\right) \eee^{-x^2/2} \alpha_1'(t) \right| \dd x \dd t\\
\leq&
|z_2-z_1| C_1 \int_0^{\infty} \eee^{-cx^2} \dd x < \infty.
\end{align*}
Analogous bounds hold for $\alpha_2$ and $\alpha_3$. Hence, Fubini's theorem can be applied and it follows that
\begin{multline*}
\oint_{\langle z_1,z_2,z_3 \rangle} F_1(r) \dd r
\\
=
\frac 1 {\sqrt {2\pi}} \int_{0}^{\infty} \oint_{\langle z_1,z_2,z_3 \rangle}  \left(\prod_{j=1}^d \Phi \left((\eps_j \sqrt{r}x -t_j)\sqrt{\lambda_j} \right) + \prod_{j=1}^d \Phi \left((-\eps_j \sqrt{r}x -t_j)\sqrt{\lambda_j} \right)\right) \eee^{-x^2/2}  \dd r \dd x
=
0,
\end{multline*}
where the inner integral vanishes by Cauchy's theorem -- recall that the integrand is an analytic function in $r$.
Morera's theorem is now applicable and entails that $F_1(r)$ is an analytic function on the half-plane $\Re r > -\frac{1}{\lambda_1 +\ldots + \lambda_d}$.

\underline{Step 4:} With the same strategy as in Step 3, we prove that $F_2(r)$, as defined in~\eqref{eq:proof_g_d_formula_def_F_2}, is analytic on the domain
$$
\mathcal D = \left\{r\in \CC: \Re r > -1/(\lambda_1 +\ldots + \lambda_d) \right\}.
$$
On this domain, the term in the square root, namely $\frac{\lambda_1 \cdots \lambda_d}{1+r(\lambda_1+\ldots+\lambda_d)}$, has a positive real part, hence the square root of this term is analytic.
It remains to show that
$$
\widetilde F_2(r):= \int_{(t_1, \infty) \times \cdots \times (t_d, \infty)} \exp \left\{ \frac{1}{2} \frac{r}{1+r(\lambda_1 + \ldots + \lambda_d)}\left(\sum_{j=1}^d \eps_j \lambda_j x_j\right)^2 - \frac{1}{2} \sum_{j=1}^d \lambda_j x_j^2\right\} \dd x_1 \cdots \dd x_d
$$
is analytic on $\mathcal D$.
Let $x\in \R^d$. Jensen's inequality entails
\begin{align}
\left( \sum_{j=1}^d \eps_j \lambda_j x_j \right)^2
=&
(\lambda_1 + \ldots +\lambda_d)^2 \left( \frac{\sum_{j=1}^d \eps_j \lambda_j x_j }{\lambda_1+\ldots+\lambda_d} \right)^2 \notag\\
\leq&
(\lambda_1+\ldots+\lambda_d)^2\frac{\sum_{j=1}^d \lambda_j (\eps_j x_j)^2 }{\lambda_1+\ldots+\lambda_d}
=
(\lambda_1+\ldots+\lambda_d) \sum_{j=1}^d \lambda_j x_j^2.\label{eq:proof_g_d_jensen_for_F2}
\end{align}
Using $|\eee^z|= \eee^{\Re z}$, we can conclude for all $r\in \CC$ with $\Re r > -1/(\lambda_1 + \ldots + \lambda_d)$ that
$$
\left| \exp \left\{ \frac{1}{2} \frac{r}{1+r(\lambda_1 + \ldots + \lambda_d)}\left(\sum_{j=1}^d \eps_j \lambda_j x_j\right)^2 - \frac{1}{2} \sum_{j=1}^d \lambda_j x_j^2\right\} \right|
\leq
\exp \left\{ -c \sum_{j=1}^d \lambda_j x_j^2 \right\}
$$
with
$$
c
=
\frac{1}{2} \min\left\{1, 1-\Re \frac{r(\lambda_1+\ldots+\lambda_d)}{1+r(\lambda_1 + \ldots + \lambda_d)} \right\}.
$$
Note that we used~\eqref{eq:proof_g_d_jensen_for_F2} if $\Re\frac{r}{1+r(\lambda_1 + \ldots + \lambda_d)}>0$, otherwise we used the trivial estimate $(\sum_{j=1}^d \eps_j \lambda_j x_j)^2\geq 0$.  It holds $c>0$, since from $\Re r>-\frac{1}{\lambda_1+\ldots+\lambda_d}$ it follows that
\begin{align*}
1-\Re \frac{r(\lambda_1+\ldots+\lambda_d)}{1+r(\lambda_1 + \ldots + \lambda_d)}
&=
\Re \frac{1}{1+r(\lambda_1 + \ldots + \lambda_d)}\\
&=
\frac{1+(\lambda_1 + \ldots +\lambda_d)\Re(r)}{(1+(\lambda_1 + \ldots +\lambda_d)\Re(r))^2 + ((\lambda_1 + \ldots + \lambda_d)\Im(r))^2} > 0.
\end{align*}
Thus, the integral in the definition of $\widetilde F_2(r)$ is finite. Next we show that $\widetilde F_2(r)$ is continuous on $\mathcal D$.
To this end, let $r,r_1,r_2,\ldots \in \mathcal D$  and such that $\lim_{n\to\infty}r_n = r$. Then, for all $n\in \N$ and $x\in \R^d$ it holds that
$$
\left| \exp \left\{ \frac{1}{2} \frac{r_n}{1+r_n(\lambda_1 + \ldots + \lambda_d)}\left(\sum_{j=1}^d \eps_j \lambda_j x_j\right)^2 - \frac{1}{2} \sum_{j=1}^d \lambda_j x_j^2\right\} \right|\leq \exp \left\{-c \sum_{j=1}^d \lambda_j x_j^2\right\},
$$
with
$$
c
=
\inf_n \frac{1}{2}\min\left\{1, 1-\Re \frac{r_n (\lambda_1+\ldots+\lambda_d)}{1+r_n(\lambda_1 + \ldots + \lambda_d)}\right\}>0.
$$
Therefore, the dominated convergence theorem applied to~\eqref{eq:proof_g_d_formula_def_F_2} yields $\lim_{n\to\infty}\widetilde F_2(r_n) = \widetilde F_2(r)$, thereby showing that $\widetilde F_2(r)$ is a continuous function in $r\in \mathcal D$.

Next we prove that the function $\widetilde F_2(r)$ is analytic on $\mathcal D$. By Morera's theorem, we have to verify that the contour integral of $\widetilde{F}_2$ over every triangular path in $\mathcal D$ vanishes.
Let $z_1,z_2,z_3\in \mathcal D$, that is $\Re z_i >-1/(\lambda_1 +\ldots + \lambda_d), i=1,2,3$.  As already defined above, let $\lan z_1, z_2, z_3\ran$ be the triangular contour  $\alpha_1 \oplus \alpha_2 \oplus \alpha_3$ with $z_1, z_2, z_3$ as vertices.
Then,
\begin{align*}
\int_0^1 \int_{\bigtimes_{j=1}^d (t_j, \infty)} &\left| \exp \left\{ \frac{1}{2} \frac{\alpha_1(t)}{1+\alpha_1(t)(\lambda_1 + \ldots + \lambda_d)}\left(\sum_{j=1}^d \eps_j \lambda_j x_j\right)^2 - \frac{1}{2} \sum_{j=1}^d \lambda_j x_j^2\right\} \alpha_1'(t) \right|\dd x_1\cdots \dd x_d \dd t\\
\leq&
|z_2-z_1| \int_{(t_1,\infty)\times \ldots \times (t_d, \infty)} \exp \left\{-c \sum_{j=1}^d \lambda_j x_j^2\right\}\dd x_1\cdots \dd x_d
<
\infty,
\end{align*}
where
$$
c=\inf_{r\in [z_1,z_2]} \frac{1}{2}\min \left\{ 1, 1-\Re \frac{r(\lambda_1+\ldots+\lambda_d)}{1+r(\lambda_1 + \ldots + \lambda_d)}\right\}>0.
$$
Analogous estimates hold for $\alpha_2$ and $\alpha_3$. Hence,  Fubini's theorem can be applied and it follows that
\begin{multline*}
 \oint_{\langle z_1, z_2, z_3\rangle} \widetilde F_2(r) \dd r
\\
\int_{\bigtimes_{j=1}^d (t_j, \infty)} \oint_{\langle z_1, z_2, z_3\rangle}  \exp \left\{ \frac{1}{2} \frac{r}{1+r(\lambda_1 + \ldots + \lambda_d)}\left(\sum_{j=1}^d \lambda_j x_j\right)^2 - \frac{1}{2} \sum_{j=1}^d \lambda_j x_j^2\right\}  \dd r \dd x_1\cdots \dd x_d
=
0,
\end{multline*}
where the inner contour integral vanishes by Cauchy's theorem which is applicable since the integrand is an analytic function in $r$.  Hence, Morera's theorem yields that $\widetilde F_2(r)$ is an analytic function on $\mathcal D$. Thus, $F_2$ is analytic on $\mathcal D$.

\underline{Step 5:} By Steps 1 and 2, $F_1$ and $F_2$ agree on $(0, \infty)$. The identity theorem for analytic functions, which is applicable by Steps 3 and 4, now entails that $F_1$ and $F_2$ agree on the domain $\Re r > -1/(\lambda_1 +\ldots + \lambda_d)$.
\end{proof}

The next corollary extends slightly~\cite[Corollary~2.B.1]{kuchelmeister_phd}, where the case $t_1=\ldots= t_d=0$ has been considered.
\begin{corollary}
Let $Z_1\sim \mathrm{N}(0, 1/\lambda_1),\ldots, Z_d\sim \mathrm{N}(0, 1/\lambda_d)$ be independent centered normal random variables with inverse variances $\lambda_1>0,\ldots, \lambda_d >0$. Consider their convex combination $Z:= (\lambda_1 Z_1 + \ldots + \lambda_d Z_d)/(\lambda_1+\ldots + \lambda_d)$.  Then, for every $t_1,\ldots, t_d\in \R$ and  $\eps_1,\ldots,\eps_d\in \{\pm 1\}$, we have
\begin{equation*}
\P\left[ \eps_1 (Z_1 - Z) \geq t_1, \ldots, \eps_d(Z_d - Z) \geq t_d\right]
=
\frac{1}{\sqrt{2 \pi}} \int_{-\infty}^{\infty}  \prod_{j=1}^d \Phi \left(\left( \frac{\eps_j \sqrt{-1} x}{\sqrt{\lambda_1 + \ldots + \lambda_d}} -t_j\right)\sqrt{\lambda_j} \right) \eee^{-x^2/2} \dd x.
\end{equation*}
\end{corollary}
\begin{proof}
The jointly Gaussian random variables $\eta_i := Z_i -Z$ satisfy $\E [\eta_i \eta_j] = \frac {\delta_{ij}} {\lambda_i}  - \frac 1 {\lambda_1 + \ldots + \lambda_d}$.  Thus, we apply Lemma~\ref{lemma:proba_gaussian_greater_const} with $r\downarrow  -\frac 1 {\lambda_1 + \ldots + \lambda_d}$. We omit the justification of continuity.
\end{proof}

\subsection{The case when \texorpdfstring{$\lambda_1>0, \ldots, \lambda_d>0$}{lambda\_1>0,...,lambda\_d >0}}
Now we are in position to derive an explicit formula for $\bg_d(\lambda_0; \lambda_1, \ldots, \lambda_d; \eps_1, \ldots, \eps_d)$ in the case when $\lambda_1>0,\ldots, \lambda_d >0$.
\begin{theorem}\label{theo:formula_for_bg_d(lambda_i,eps_i)}
Let $d\geq 0$, $\lambda_1, \ldots, \lambda_d>0$, $\eps_1, \ldots, \eps_d \in \{\pm 1\}$ and either $\lambda_0>0$ or $\lambda_0<-\lambda_1-\ldots -\lambda_d$. Then, we have
\begin{align}
\label{eq:bg_d(lambda_i,eps_i)_formula}
\bg_d(\lambda_0; \lambda_1, \ldots, \lambda_d; \eps_1, \ldots, \eps_d)
&=
\frac{1}{\sqrt{2 \pi}} \int_0^{\infty} \left( \prod_{j=1}^d \Phi \left(\eps_j \sqrt{\frac{\lambda_j}{\lambda_0}}x \right) + \prod_{j=1}^d \Phi \left(-\eps_j \sqrt{\frac{\lambda_j}{\lambda_0}}x \right) \right) \eee^{-x^2/2} \dd x\\
&=
\frac{1}{\sqrt{2 \pi}} \int_{-\infty}^{\infty} \prod_{j=1}^d \Phi \left(\eps_j \sqrt{\frac{\lambda_j}{\lambda_0}}x \right) \eee^{-x^2/2} \dd x.
\end{align}
Here, $\Phi:\CC\to\CC$ denotes the distribution function of the standard normal distribution, as defined in~\eqref{eq:compl_std_normal_distr}. C
\end{theorem}

\begin{proof}
For \( d = 0 \), we have \( \bg_0 = 1 \) by definition, and the right-hand side also equals 1 by the identity \( \frac{1}{\sqrt{2\pi}} \int_{-\infty}^{\infty} \eee^{-x^2/2} \dd x = 1 \) (with the convention that the empty product is 1). For \( d \geq 1 \), the claim follows from Definition~\ref{def:g_d_function}, in combination with Lemma~\ref{lemma:proba_gaussian_greater_const}, by setting \( t_1 = \ldots = t_d = 0 \) and \( r = 1/\lambda_0 \).
\end{proof}
As a corollary, we obtain an explicit formula for the solid angles of orthocentric cones with parameters \( \lambda_1, \ldots, \lambda_d > 0 \).
\begin{corollary}
Let $v_1, \ldots, v_d \in \R^d$, $\lambda_1, \ldots, \lambda_d > 0$, $\eps_1, \ldots, \eps_d \in \{\pm 1\}$ and either $\lambda_0>0$ or $\lambda_0<-\lambda_1-\ldots -\lambda_d$, be such that $\lan v_i, v_j \ran = 1/\lambda_0 + \delta_{ij}/\lambda_i$ for all $1\leq i,j \leq d$. Then, we have
$$
\alpha(\pos(\eps_1v_1, \ldots, \eps_dv_d))
=
\frac{1}{\sqrt{2\pi}} \int_{-\infty}^{\infty} \prod_{j=1}^d \Phi \left( \eps_j \sqrt{\frac{\lambda_j}{-\lambda_0-\lambda_1-\ldots -\lambda_d}} x \right) \eee^{-x^2/2} \dd x.
$$
\end{corollary}
\begin{proof}
Substitute the expression for \( \bg_d \) from Theorem~\ref{theo:formula_for_bg_d(lambda_i,eps_i)} into Theorem~\ref{theo:angle_pos(eps_i_v_i)_via_g_d(lambda_i,eps_i)}.
\end{proof}

\subsection{The case when \texorpdfstring{$\lambda_k<0$}{lambda\_k<0} for some \texorpdfstring{$k\in\{1,\ldots, d\}$}{k in \{1,...,d\}}}
The next theorem provides a formula for $\bg_d$ in the case when $\lambda_k<0$ for some $k\in\{1,\ldots, d\}$.
\begin{theorem}\label{theo:formula_for_bg_d(lambda_i,eps_i)_negative}
Let $d \in \N$ and $k \in \{1, \ldots, d\}$. Suppose that $\lambda_0, \ldots, \lambda_d$  satisfy $\lambda_i > 0$ for all $i \in \{0, \ldots, d\}\backslash\{k\}$, and $\lambda_k < 0$, with the additional condition that $\lambda_0 + \ldots + \lambda_d < 0$. Then, for all $\eps_1, \ldots, \eps_d \in \{\pm 1\}$, we have
\begin{multline*}
\bg_d(\lambda_0; \lambda_1,\ldots, \lambda_d;\eps_1, \ldots, \eps_d)
=
\sqrt{\frac{\lambda_0}{2\pi}} \Biggl(
\int_0^{\infty} \eee^{-\frac{\lambda_0 y^2}{2}} \prod_{\substack{j=1\\j\neq k}}^d \Phi \left(\eps_k \eps_j \sqrt{\lambda_j} y \right) \dd y
\\
+2 \int_0^{\infty} \eee^{\frac{\lambda_0 y^2}{2}} \Phi\left(-\sqrt{-\lambda_k} y \right) \Im \Bigg( \prod_{\substack{j=1\\j\neq k}}^d \Phi \left(i \eps_k \eps_j \sqrt{\lambda_j} y \Bigg) \right) \dd y
\Biggr).
\end{multline*}
\end{theorem}
In order to prove the aforementioned theorem, we first compute a quantity slightly more general than the internal angle of a simplex $S=[\nu H, e_1/\tau_1, \ldots, e_d/\tau_d]$, as in Example~\ref{exam:orthocenter_outside_alternative}, at the vertex $e_d/\tau_d$.
\begin{lemma}\label{lemma:angle_difficult_orthoc_simplex}
Let $\nu \in (0,1)$ and $a:=\nu/(1-\nu) \in (0,\infty)$. For $d \geq 2$, let $\tau_1,\ldots, \tau_d>0$, $p=\tau_1^2+\ldots +\tau_d^2$, $H=(\tau_1 e_1 + \ldots + \tau_d e_d)/p$ and $\eps_1, \ldots, \eps_{d-1},\eps \in \{\pm 1\}$. The solid angle of the cone
$$
T
=
\pos\left(\eps \left( \nu H-\frac{e_d}{\tau_d} \right), \eps_1 \left( \frac{e_1}{\tau_1}-\frac{e_d}{\tau_d} \right), \ldots,  \eps_{d-1} \left( \frac{e_{d-1}}{\tau_{d-1}}-\frac{e_d}{\tau_d} \right) \right),
$$
is given by
\begin{align*}
\alpha(T)=\frac{1}{\sqrt{2 \pi a (a +2)}} \Biggl(&
\int_0^{\infty} \eee^{-\frac{y^2}{2a(a+2)}} \prod_{j=1}^{d-1} \Phi \left(-\eps \eps_j \frac{\tau_j}{\sqrt{p}} y \right) \dd y\\
&+2 \int_0^{\infty} \eee^{\frac{y^2}{2a(a+2)}} \Phi \left( -\frac{a+1}{\sqrt{a(a+2)}}y \right) \Im \left( \prod_{j=1}^{d-1} \Phi\left(-i \eps \eps_j \frac{\tau_j}{\sqrt{p}}y \right)  \right)\dd y
\Biggr).
\end{align*}
\end{lemma}
\begin{proof}
Let us prove the formula in case $\eps=1$. Once this is established, the case $\eps=-1$ follows with the observation $\alpha(\pos(v_1, \ldots, v_d)) = \alpha(\pos(-v_1, \ldots, -v_d))$.
For $\eps=1$ we have
$$
T^{\circ}
=
\pos\left(H, \eps_1 \left(-a H - \frac{e_1}{\tau_1}\right), \ldots, \eps_{d-1} \left(-a H - \frac{e_{d-1}}{\tau_{d-1}}\right) \right).
$$
Indeed, the scalar product of the $i$-th generator of $T$ with the $j$-th generator of $T^\circ$ is negative if $i=j$ and zero otherwise. Alternatively, $-aH$ is the orthocenter of $S$ and therefore the lines spanned by the vectors $H,-a H - \frac{e_1}{\tau_1},\ldots, -a H - \frac{e_{d-1}}{\tau_{d-1}}$ are orthogonal to the facets of $S$ containing $\frac{e_d} {\tau_d}$ -- this implies the above representation of $T^\circ$.
With a standard Gaussian distributed random vector $\xi = (\xi_1, \ldots, \xi_d)$, the solid angle of $T$ is given by
\begin{align*}
\alpha(T)
&=
\P\left[\xi \in T \right]
=
\P\left[\lan \xi, x \ran \leq 0 \, \forall x\in T^{\circ}\right]\\
&=
\P\left[\lan \xi,H \ran \leq 0, \eps_j \left\lan -a H-\frac{e_j}{\tau_j}, \xi \right\ran\leq 0\, \forall j=1,\ldots,d-1 \right]\\
&=
\P\left[\lan \xi,H \ran \leq 0, \eps_j \left\lan a H+\frac{e_j}{\tau_j}, \xi \right\ran \geq 0\; \forall j=1,\ldots,d-1 \right]\\
&=
\int_{-\infty}^0 \P\left[ \eps_j\left(\frac{\xi_j}{\tau_j}-h \right) \geq -\eps_j (a+1)h \; \forall j=1,\ldots, d-1 \Big| \lan \xi,H \ran = h \right] \P_{\lan \xi,H \ran} (\dd h).
\end{align*}
For further simplification of this expression, first note that $\lan \xi, H \ran \sim \mathrm{N}(0,1/p)$. Moreover by~\cite[Theorem 3.3.4]{tong_book}, for fixed $h \in (-\infty, 0)$, the conditional distribution
$$
\left(\frac{\xi_1}{\tau_1}, \ldots, \frac{\xi_d}{\tau_d} \, \Big|\, \lan \xi, H \ran = h\right)
$$
is a multivariate Gaussian distribution with mean
$$
\mu
=
\begin{pmatrix}
\E[\xi_1/\tau_1]\\
\vdots\\
\E[\xi_d/\tau_d]
\end{pmatrix}
\cdot
\begin{pmatrix}
\Cov\left(\xi_1/\tau_1, \lan\xi,H\ran \right)\\
\vdots\\
\Cov\left(\xi_d/\tau_d, \lan\xi,H\ran \right)
\end{pmatrix}
\cdot
\left(\Var\left( \lan\xi,H\ran \right) \right)^{-1}
\cdot
\left(h-\E\left[ \lan\xi,H\ran \right] \right)
=
\begin{pmatrix}
h\\
\vdots\\
h
\end{pmatrix},
$$
and covariance matrix
\begin{align*}
\Sigma
&:=
\left(\Cov\left(\frac{\xi_i}{\tau_i}, \frac{\xi_j}{\tau_j} \right) \right)_{i,j=1}^d
-
\begin{pmatrix}
	\Cov\left(\frac{\xi_1}{\tau_1}, \lan\xi,H\ran \right)\\
	\vdots\\
	\Cov\left(\frac{\xi_d}{\tau_d}, \lan\xi,H\ran \right)
\end{pmatrix}
\left( \Var \lan\xi,H\ran \right)^{-1}
\begin{pmatrix}
	\Cov\left(\frac{\xi_1}{\tau_1}, \lan\xi,H\ran \right)\\
	\vdots\\
	\Cov\left(\frac{\xi_d}{\tau_d}, \lan\xi,H\ran \right)
\end{pmatrix}^T\\
&=
\left( -\frac{1}{p}+\frac{\delta_{ij}}{\tau_i^2} \right)_{i,j=1}^d.
\end{align*}

Let $\eta = (\eta_1, \ldots, \eta_d)$ be a Gaussian random vector with mean zero and covariance matrix $\Sigma$. It follows that
$$
\alpha(T)
=
\int_{-\infty}^0 \P\left[\eps_1 \eta_1 \geq -\eps_1(a+1)h, \ldots, \eps_{d-1} \eta_{d-1} \geq -\eps_{d-1} (a+1) h \right] \cdot \frac{\sqrt{p}}{\sqrt{2\pi}} \eee^{-h^2p/2} \dd h.
$$
Next, applying Lemma~\ref{lemma:proba_gaussian_greater_const}, making the substitution $z = -h$, and using Fubini's theorem, we obtain
\begin{align*}
\alpha(T)
&=
\int_{-\infty}^0 \frac{1}{\sqrt{2\pi}} \int_{-\infty}^{\infty} \prod_{j=1}^{d-1} \Phi\left( \left(\eps_j\sqrt{-\frac{1}{p}} x+\eps_j (a+1)h\right)\tau_j \right) \eee^{-x^2/2} \dd x \cdot \frac{\sqrt{p}}{\sqrt{2\pi}} \eee^{-h^2p/2} \dd h\\
&=
\frac{\sqrt{p}}{2\pi} \int_{-\infty}^{\infty} \eee^{-x^2/2} \int_0^{\infty} \eee^{-z^2p/2} \prod_{j=1}^{d-1} \Phi\left( \eps_j \left(\sqrt{-\frac{1}{p}} x -(a+1)z\right)\tau_j \right) \dd z \dd x.
\end{align*}
In the inner integral, substitute $y=x-(a+1)z\sqrt{-p}$, which yields $z=(x-y)/((a+1)\sqrt{-p})$, and use the convention $\sqrt{-p}=i\sqrt{p}$. This entails
$$
\alpha(T)
=
\frac{i}{2\pi (a+1)} \int_{-\infty}^{\infty} \int_x^{x-i\infty} \underbrace{\eee^{\frac{(x-y)^2}{2(a+1)^2}-\frac{x^2}{2}} \prod_{j=1}^{d-1} \Phi \left( -\eps_j i \frac{\tau_j}{\sqrt{p}} y \right)}_{=:f(x,y)} \dd y \dd x.
$$
Fix some $x\in \R$.  By Cauchy's theorem, the integral of $f(x,y)$ over the segment $[x, x-i B]$ can be written as the sum of integrals over $[x,0], [0, -iB], [-iB, x-iB]$, for all $B>0$. We now show that the last integral converges  to $0$ as $B\to+\infty$. Take some $C>|x|$.  For all $y\in \CC$ satisfying $|\Re y|\leq C$ and $\Im y\leq 0$ we have
$$
\left| \eee^{\frac{(x-y)^2}{2 (a+1)^2}-\frac{x^2}{2}} \right|
\leq
\mathrm{const}\cdot \left| \eee^{\frac{y^2}{2(a+1)^2}} \right|
\leq
\mathrm{const} \cdot \eee^{-\frac{\Im(y)^2}{2(a+1)^2}},
$$
and moreover
$$
\Re \left( \left( -\eps_j i \frac{\tau_j}{\sqrt{p}} y \right)^2 \right)
=
\frac{\tau_j^2}{p} \left( \Im(y)^2-\Re(y)^2 \right) \geq -\frac{\tau_j^2}{p} C.
$$
By Proposition~\ref{prop:bound_Phi}, we obtain the bound $|f(x,y)| \leq \mathrm{const} \cdot \eee^{-\Im(y)^2/(2(a+1))}$. It follows that  $\int_{-iB}^{x-iB} f(x,y) \dd y \to 0$ as $B\to +\infty$ and hence
\begin{align*}
\alpha(T)
&=
\frac{i}{2\pi (a+1)} \int_{-\infty}^{\infty} \int_x^{x-i\infty} f(x,y) \dd y \dd x\\
&=
\frac{i}{2\pi (a+1)} \left(\underbrace{\int_{-\infty}^{\infty} \int_0^{-i\infty} f(x,y) \dd y \dd x}_{=:I_1} - \underbrace{\int_{-\infty}^{\infty} \int_0^x f(x,y) \dd y \dd x}_{=:I_2} \right).
\end{align*}
For the computation of both $I_1$ and $I_2$, for all $y\in \CC$, we need the indefinite integral
\begin{equation}\label{eq:indefin_integral_pf_intern_angle}
\int \eee^{\frac{(x-y)^2}{2 (a+1)^2}-\frac{x^2}{2}} \dd x
=
\frac{\sqrt{2\pi}(a+1)}{\sqrt{a(a+2)}} \eee^{\frac{y^2}{2a(a+2)}} \Phi \left( \frac{a(a+2)x+y}{\sqrt{a(a+2)}(a+1)} \right) + \text{const},
\end{equation}
which  can be verified by differentiating the right-hand side. First, consider the integral $I_1$. Substituting $z=iy$ in the inner integral and applying Fubini's theorem, we obtain
\begin{align*}
I_1
&=
-i \int_0^{\infty} \int_{-\infty}^{\infty} \eee^{\frac{(x+iz)^2}{2 (a+1)^2}-\frac{x^2}{2}} \dd x \cdot \prod_{j=1}^{d-1} \Phi\left(- \eps_j \frac{\tau_j}{\sqrt{p}} z \right) \dd z\\
&=
-i \int_0^{\infty} \left[ \frac{\sqrt{2\pi}(a+1)}{\sqrt{a(a+2)}} \eee^{\frac{(-iz)^2}{2a(a+2)}} \Phi \left( \frac{a(a+2)x-iz}{\sqrt{a(a+2)}(a+1)} \right) \right]_{x=-\infty}^{x=\infty} \prod_{j=1}^{d-1} \Phi\left(- \eps_j \frac{\tau_j}{\sqrt{p}} z \right) \dd z\\
&=
-i \int_0^{\infty} \frac{\sqrt{2\pi}(a+1)}{\sqrt{a(a+2)}} \eee^{-\frac{z^2}{2a(a+2)}} (1-0) \cdot \prod_{j=1}^{d-1} \Phi\left(- \eps_j \frac{\tau_j}{\sqrt{p}} z \right) \dd z.
\end{align*}
The second equality follows with~\eqref{eq:indefin_integral_pf_intern_angle} and the limits in the final equality are a consequence of the known asymptotics of the function $\Phi$; see, e.g.~\cite[Eq.\ (3.7)]{kabluchkoklimovsky}.

The integral $I_2$ can be written as
\begin{align*}
I_2
&=
\int_{-\infty}^{\infty}  \int_{-\infty}^{\infty} \sgn(x) \ind\{0<|y|<|x|, \sgn(y)=\sgn(x)\} f(x,y) \dd y \dd x\\
&=
-\int_{-\infty}^0 \int_{-\infty}^y f(x,y) \dd x \dd y + \int_0^{\infty} \int_y^{\infty} f(x,y) \dd x \dd y,
\end{align*}
where Fubini's theorem is used in the second integral, and the outer integral is split into two parts. To evaluate the inner integrals, substitute $f(x, y)$ and apply~\eqref{eq:indefin_integral_pf_intern_angle} in each summand. This yields
\begin{align*}
I_2 =&
-\int_{-\infty}^0 \frac{\sqrt{2\pi}(a+1)}{\sqrt{a(a+2)}} \eee^{\frac{y^2}{2a(a+2)}} \left( \Phi\left(\frac{a(a+2)y+y}{\sqrt{a(a+2)}(a+1)} \right) -0 \right) \prod_{j=1}^{d-1} \Phi \left( -\eps_j i \frac{\tau_j}{\sqrt{p}} y \right) \dd y\\
&+ \int_0^{\infty} \frac{\sqrt{2\pi}(a+1)}{\sqrt{a(a+2)}} \eee^{\frac{y^2}{2a(a+2)}} \left(1- \Phi\left( \frac{a(a+2)y+y}{\sqrt{a(a+2)}(a+1)} \right) \right) \prod_{j=1}^{d-1} \Phi \left( -\eps_j i \frac{\tau_j}{\sqrt{p}} y \right) \dd y.
\end{align*}
After substituting $y\mapsto -y$ in the first integral and using $1-\Phi(z)=\Phi(-z)$ in the second integral, we arrive at
$$
I_2
=
\frac{\sqrt{2\pi}(a+1)}{\sqrt{a(a+2)}} \int_0^{\infty} \eee^{\frac{y^2}{2a(a+2)}} \Phi\left( -\frac{a(a+2)y+y}{\sqrt{a(a+2)}(a+1)} \right) \left[ \prod_{j=1}^{d-1} \Phi \left( -\eps_j i \frac{\tau_j}{\sqrt{p}} y \right) - \prod_{j=1}^{d-1} \Phi \left( \eps_j i \frac{\tau_j}{\sqrt{p}} y \right) \right] \dd y.
$$
The difference of products in the above term can be written as
\begin{multline*}
\prod_{j=1}^{d-1} \Phi \left( -\eps_j i \frac{\tau_j}{\sqrt{p}} y \right) - \prod_{j=1}^{d-1} \Phi \left( \eps_j i \frac{\tau_j}{\sqrt{p}} y \right)\\
=
\prod_{j=1}^{d-1} \Phi \left(- \eps_j i \frac{\tau_j}{\sqrt{p}} y \right) - \overline{\prod_{j=1}^{d-1} \Phi \left(- \eps_j i \frac{\tau_j}{\sqrt{p}} y \right)}
=
2i \Im \left(\prod_{j=1}^{d-1} \Phi \left(- \eps_j i \frac{\tau_j}{\sqrt{p}} y \right) \right).
\end{multline*}
Plugging the formulas for $I_1$ and $I_2$ in
$
\alpha(T)
=
\frac{i}{2\pi (a+1)} \left(I_1 - I_2 \right)
$
completes the proof.
\end{proof}

The idea of the proof of Theorem~\ref{theo:formula_for_bg_d(lambda_i,eps_i)_negative} is to choose the parameters in Lemma~\ref{lemma:angle_difficult_orthoc_simplex} in a way that $T$ is isometric to a cone whose solid angle is given by $\bg_d(\lambda_0;\lambda_1, \ldots, \lambda_d;\eps_1, \ldots, \eps_d)$.
\begin{proof}[Proof of Theorem~\ref{theo:formula_for_bg_d(lambda_i,eps_i)_negative}]
In the case $d=1$, we have $\bg_d(\lambda_0;\lambda_1;\eps_1) = \P[\eps_1 \eta_1\leq 0] = 1/2$, where $\eta_1$ is some zero-mean Gaussian distributed random variable. The right-hand side in Theorem~\ref{theo:formula_for_bg_d(lambda_i,eps_i)_negative} is also $1/2$, by the identity $\frac{1}{\sqrt{2\pi}} \int_{-\infty}^{\infty} \eee^{-x^2/2} \dd x = 1$ (the empty product is $1$).

Now, let $d\geq 2$.
We may assume $k=d$. Otherwise, interchange $\lambda_k$ with $\lambda_d$ and $\eps_k$ with $\eps_d$.
Consider $\lambda_0, \lambda_1, \ldots, \lambda_{d-1}>0$, $\lambda_d<-(\lambda_0+\lambda_1+\ldots +\lambda_{d-1})$ and $\eps_1, \ldots, \eps_d\in \{\pm 1\}$. By Theorem~\ref{theo:angle_pos(eps_i_v_i)_via_g_d(lambda_i,eps_i)}, we have
$$
\bg_d(\lambda_0;\lambda_1, \ldots, \lambda_d;\eps_1, \ldots, \eps_d)
=
\alpha \left( C_d(-(\lambda_0+ \lambda_1 + \ldots + \lambda_d);\lambda_1, \ldots, \lambda_d;\eps_1, \ldots, \eps_{d-1}, -\eps_d) \right).
$$
Next, we select parameters in Lemma~\ref{lemma:angle_difficult_orthoc_simplex} such that the cone
$T$ appearing there is isometric to the cone on the right-hand side.
Set $\tau_i = \sqrt{\lambda_i}$ for $i=1,\ldots, d-1$, $\tau_d = \sqrt{-(\lambda_0 + \lambda_1 + \ldots + \lambda_d)}$, $\nu=1-\sqrt{\frac{\lambda_0}{-\lambda_d}}$ and $\eps = -\eps_d$ and consider the cone
$$
T
=
\pos\left(\eps_1 \left( \frac{e_1}{\tau_1}-\frac{e_d}{\tau_d} \right), \ldots,  \eps_{d-1} \left( \frac{e_{d-1}}{\tau_{d-1}}-\frac{e_d}{\tau_d} \right), \eps \left( \nu H-\frac{e_d}{\tau_d} \right) \right),
$$
with $H=\frac{\tau_1 e_1 + \ldots + \tau_d e_d}{p}$, where $p=\tau_1^2 + \ldots + \tau_d^2$. We observe that,  up to isometry,
\begin{align*}
T
&=
C_d \left(\tau_d^2; \tau_1^2, \ldots, \tau_{d-1}^2, -\frac{p}{\nu(2-\nu)}; \eps_1, \ldots, \eps_{d-1}, \eps \right)\\
&=
C_d(-(\lambda_0+ \lambda_1 + \ldots + \lambda_d);\lambda_1, \ldots, \lambda_d;\eps_1, \ldots, \eps_{d-1}, -\eps_d).
\end{align*}
For the first step, compute the scalar products of the generators of $T$. For the second equality, recall the definitions of $\tau_i$, $\nu$, and $\eps$ to verify that the parameters of the orthocentric cones coincide. The solid angle of $T$ has been determined in  Lemma~\ref{lemma:angle_difficult_orthoc_simplex}. In view of  $p=-\lambda_0-\lambda_d$ and $a:=\nu/(1-\nu) = \sqrt{-\lambda_d/\lambda_0}-1$, this yields
\begin{align*}
\bg_d(\lambda_0;\lambda_1, \ldots, \lambda_d; \eps_1, \ldots, \eps_d)
=
\alpha(T)\\
=
\sqrt{\frac{\lambda_0}{2 \pi (-\lambda_0-\lambda_d)}} \Biggl(& \int_0^{\infty} \eee^{-\frac{\lambda_0 y^2}{2(-\lambda_0-\lambda_d)}} \prod_{j=1}^{d-1} \Phi \left(\eps_d \eps_j \sqrt{\frac{\lambda_j}{-\lambda_0-\lambda_d}} y \right) \dd y\\
+2\int_0^{\infty} \eee^{\frac{\lambda_0 y^2}{2(-\lambda_0-\lambda_d)}}& \Phi\left( -\sqrt{\frac{-\lambda_d}{-\lambda_0-\lambda_d}}y \right) \Im \left( \prod_{j=1}^{d-1} \Phi \left(i \eps_d \eps_j \sqrt{\frac{\lambda_j}{-\lambda_0-\lambda_d}} y \right) \right) \dd y
\Biggr).
\end{align*}
Finally, the substitution $y \mapsto y/\sqrt{-\lambda_0-\lambda_d}$ completes the proof.
\end{proof}

\section{Angles of rectangular simplices}\label{sec:right_angled_simplices}

\subsection{Tangent and normal cones of rectangular simplices}
\begin{proposition}[Tangent and normal cones of rectangular simplices]\label{prop:degenerate_simpl_cones}
For $d\geq 2$ and $\tau_1, \ldots, \tau_d>0$ consider the simplex $S=[0,e_1/\tau_1, \ldots, e_d/\tau_d] \subseteq \R^d$. The tangent and normal cones of $S$ at its faces are given as follows.
\begin{enumerate}
\item $T(S,S)=\R^d$ and $N(S,S)=\{0\}$.
\item For $k\in \{0, \ldots, d-1\}$ consider $F=[0,e_1/\tau_1, \ldots, e_k/\tau_k]$. Then, $N(F,S)$ is a $(d-k)$-dimensional orthant and $T(F,S)$ is the direct orthogonal sum of the lineality space of $F$ and a $(d-k)$-dimensional orthant.
\item For $k\in \{0, \ldots, d-1\}$ consider $F=[e_1/\tau_1, \ldots, e_{k+1}/\tau_{k+1}]$. Then, $N(F,S)=\pos(w_{k+1}, \ldots, w_d)$ for some vectors $w_{k+1}, \ldots, w_d$ with Gram matrix
$$
\left( \lan w_i, w_j \ran \right)_{i,j=k+1}^d
=
\begin{pmatrix}
\tau_1^2 + \ldots + \tau_d^2 & -\tau_{k+2} & -\tau_{k+3} & \cdots & -\tau_d\\
-\tau_{k+2} & 1 & 0 & \cdots & 0\\
\vdots & 0 & \ddots & \ddots & \vdots\\
\vdots& \vdots & \ddots & \ddots & 0\\
-\tau_d & 0 & \cdots & 0 & 1
\end{pmatrix}.
$$
The tangent cone $T(F,S)$ is the direct orthogonal sum of the lineality space of $F$ and a pointed cone $\pos(v_{k+1}, \ldots, v_d)$ for some vectors $v_{k+1}, \ldots, v_d$ with Gram matrix
\begin{equation}\label{eq:gram_matrix_tangent_cone_degenerate}
\left( \lan v_i, v_j \ran \right)_{i,j=k+1}^d
=
\begin{pmatrix}
1 & \tau_{k+2} & \tau_{k+3} &\cdots & \tau_d\\
\tau_{k+2} & \tau_1^2 + \ldots + \tau_{k+1}^2 + \tau_{k+2}^2 & \tau_{k+2} \tau_{k+3} &\cdots &\tau_{k+2}\tau_d\\
\vdots & \tau_{k+3}\tau_{k+2} & \ddots & \ddots & \vdots\\
\vdots & \vdots & \ddots & \ddots & \tau_{d-1}\tau_d\\
\tau_d & \tau_d\tau_{k+2} & \cdots & \tau_d\tau_{d-1} & \tau_1^2 + \ldots + \tau_{k+1}^2 + \tau_d^2
\end{pmatrix}
=:M.
\end{equation}
\end{enumerate}
\end{proposition}
\begin{proof}
\emph{Proof of (a)}. This follows from the definition of the tangent cone and $N(S,S)=T(S,S)^{\circ}$.

For (b) and (c) note that an outer normal vector of $S$ at the facet opposite to $0$ is given by $\tau_1 e_1 + \ldots + \tau_d e_d$.
Also, for all $i\in \{1, \ldots, d\}$, an outer normal vector at the facet opposite to $e_i/\tau_i$ is given by $-e_i$.

\vspace*{2mm}
\noindent
\emph{Proof of (b)}. The normal cone at any face $F$ is the positive hull of the outer normal vectors at all facets containing $F$. In particular, the normal cone of $S$ at the face $F=[0, e_1/\tau_1, \ldots, e_k/\tau_k]$ is $N(F,S) = \pos(-e_{k+1}, \ldots, -e_d)$, which is an orthant. The identity $T(F, S) = N(F, S)^{\circ}$ now yields the desired description of the tangent cone.

\vspace*{2mm}
\noindent
\emph{Proof of (c)}. Write $w_i=-e_i$ for $i=k+2, \ldots, d$ and let $w_{k+1}=\tau_1 e_1 + \ldots +\tau_d e_d$. The normal cone of $S$ at the face $F=[e_1/\tau_1, \ldots, e_{k+1}/\tau_{k+1}]$ is given by $N(F,S) = \pos(w_{k+1}, w_{k+2}, \ldots, w_d)$. Computing the scalar products yields the desired Gram matrix.  Then, the tangent cone $T(F,S) = N(F,S)^{\circ}$ is the direct orthogonal sum of the lineality space of $F$ and the polar cone of $\pos(w_{k+1}, \ldots, w_d)$, considered as a cone in $\lin(w_{k+1}, \ldots, w_d)$. The Gram matrix of $w_{k+1}, \ldots, w_d$ is non-singular. Indeed, a straightforward computation shows that
\begin{equation}~\label{eq:inverse_rectangular_cones}
\left( \lan w_i, w_j \ran \right)_{i,j=k+1}^d \cdot \frac{1}{\tau_1^2 + \ldots + \tau_{k+1}^2} M = \textrm{Id},
\end{equation}

where $M$ as in~\eqref{eq:gram_matrix_tangent_cone_degenerate}, and where $\textrm{Id}$ denotes the identity matrix.

As shown in the proof of Theorem~\ref{theo:orthocentric_cone_dual} (with different notation),  the polar cone of $\pos(w_{k+1}, \ldots, w_d)$ is spanned by some vectors $\widetilde{v}_{k+1}, \ldots, \widetilde{v}_d$ whose Gram matrix is the inverse of that of $w_{k+1}, \ldots, w_d$. The rescaled vectors $v_i = \widetilde{v_i}/(\tau_1^2 + \ldots \tau_{k+1}^2)^{1/2}$  generate the same polar cone. Since the Gram matrix of $\widetilde{v}_{k+1}, \ldots, \widetilde{v}_d$ is given by $\frac{1}{\tau_1^2 + \ldots + \tau_{k+1}^2} M$, it follows that $M$ is the Gram matrix of $v_{k+1}, \ldots, v_d$.
\end{proof}


\subsection{Limiting cases of the function \texorpdfstring{$\bg_d$}{g\_d}}
To compute the  angles of a rectangular orthocentric simplex $[0,e_1/\tau_1, \ldots, e_d/\tau_d]$, the following two limits of the function  $\bg_d$ will be needed.
\begin{proposition}\label{prop:g_d_limit_lam1}
For $d\geq 1$, $\lambda_0, \lambda_2, \ldots, \lambda_d>0$ and $\eps_1, \ldots, \eps_d \in \{\pm 1\}$ we have
$$
\lim_{\lambda_1 \rightarrow - \infty} \bg_d(\lambda_0; \lambda_1, \ldots, \lambda_d; \eps_1, \ldots, \eps_d)
=
\frac{1}{\sqrt{2\pi}} \int_0^{\infty} \prod_{j=2}^d \Phi\left(\eps_1 \eps_j \sqrt{\frac{\lambda_j}{\lambda_0}}x \right) \eee^{-x^2/2} \dd x.
$$
\end{proposition}

\begin{proposition}\label{prop:g_d(-lam0-lam1)_limit_lam1}
For $d\geq 1$ let $\lambda_2, \ldots, \lambda_d > 0$, $\lambda_0>\lambda_2+\ldots + \lambda_d$ and $\eps_1, \ldots, \eps_d \in \{\pm 1\}$. Then,
$$
\lim_{\lambda_1 \rightarrow -\infty} \bg_d\left( -\lambda_0-\lambda_1; \lambda_1, \ldots, \lambda_d; \eps_1, \ldots, \eps_d \right)
=
\frac{1}{2^d} + \int_0^{\infty} \frac{1}{\pi y} \eee^{-\lambda_0 y^2/2} \Im \left( \prod_{j=2}^d \Phi \left(i \eps_1 \eps_j \sqrt{\lambda_j} y \right) \right) \dd y.
$$
\end{proposition}

The following assertion is a preparation for the proof of both propositions.
\begin{lemma}\label{lemma:integrable_function_for_limit}
Let $d\geq 2$, $\lambda_0, \lambda_2, \ldots, \lambda_d>0$ and $\eps_1, \ldots, \eps_d \in \{\pm 1\}$.
\begin{enumerate}
\item Let $\lambda_1<0$ with $|\lambda_1|$  sufficiently large. Then, the functions
$$
y\mapsto \Bigg| \eee^{\lambda_0 y^2/2} \Phi\left( -\sqrt{-\lambda_1}y \right) \Im \left( \prod_{j=2}^d \Phi\left( i\eps_1\eps_j \sqrt{\lambda_j}y \right) \right) \Bigg|,
\qquad y>0,
$$
have an integrable majorant which does not depend on $\lambda_1$.
\item Additionally, assume $\lambda_0>\lambda_2 + \ldots + \lambda_d$. Provided that $\lambda_1<0$ and $|\lambda_1|$ is sufficiently large, the functions
$$
y\mapsto \Bigg| \sqrt{-\lambda_0-\lambda_1} \eee^{(-\lambda_0-\lambda_1) y^2/2} \Phi\left( -\sqrt{-\lambda_1}y \right) \Im \left( \prod_{j=2}^d \Phi\left( i\eps_1\eps_j \sqrt{\lambda_j}y \right) \right) \Bigg|,
\qquad y>0,
$$
have an integrable majorant which does not depend on $\lambda_1$.
\end{enumerate}
\end{lemma}
\begin{proof}
The Mills ratio inequality~\cite[Eq. 7]{Gordon_MillsRatio} says $\Phi(-x) \leq 1/(\sqrt{2\pi}x) \eee^{-x^2/2}$ for all $x>0$. Hence,
\begin{equation}\label{eq:mills_ratio_ineq_appl}
\Phi(-\sqrt{-\lambda_1}y)
\leq
\frac{1}{\sqrt{2\pi} \sqrt{-\lambda_1}y} \eee^{\lambda_1 y^2/2},
\quad
\text{for all $y>0$ and $\lambda_1<0$}.
\end{equation}

\vspace*{2mm}
\noindent
\emph{Proof of (a).} We shall find an integrable majorant on the intervals $(0,1]$ and $[1,\infty)$ which we shall  treat separately. On the interval $(0,1]$, we use~\eqref{eq:mills_ratio_ineq_appl} to bound the function in (a) from above by
\begin{equation}\label{eq:tech_upper_bound_134563}
y\mapsto \frac{\eee^{\lambda_0 y^2/2}}{\sqrt{2\pi}  y} \Bigg| \Im \left( \prod_{j=2}^d \Phi\left( i\eps_1\eps_j \sqrt{\lambda_j}y \right) \right)\Bigg|.
\end{equation}
This holds for all  $\lambda_1\leq -1$. It remains to show that~\eqref{eq:tech_upper_bound_134563} is integrable over $(0,1]$. Taylor's theorem yields
$$
\Phi(x)
=
\frac{1}{2} + \frac{1}{\sqrt{2\pi}}x + o(x), \; \text{ as } x\rightarrow 0.
$$
Hence, there exists some $c\in \R$ such that
\begin{equation}\label{eq:pf_lemma_integrability_bd_for_im_1}
\prod_{j=2}^d \Phi\left( i\eps_1\eps_j \sqrt{\lambda_j}y \right)
=
\frac{1}{2^{d-1}} + icy + o(y)
,\;\;
\Im \left( \prod_{j=2}^d \Phi\left( i\eps_1\eps_j \sqrt{\lambda_j}y \right) \right) = cy+o(y), \text{ as } y\rightarrow 0.
\end{equation}
It follows that the function~\eqref{eq:tech_upper_bound_134563} has a finite limit as $y\downarrow 0$, which implies its integrability.

Now, consider the interval $[1,\infty)$. By Proposition~\ref{prop:bound_Phi} there exists $C>0$, such that for all $y>0$ and $\lambda_1<0$ we have the inequality
\begin{equation}\label{eq:pf_lemma_integrability_bd_for_im_2}
\Bigg| \Im \left( \prod_{j=2}^d \Phi\left( i\eps_1\eps_j \sqrt{\lambda_j}y \right) \right) \Bigg|
\leq
C^{d-1} \prod_{j=2}^d \max \left\{ 1, \eee^{\lambda_j y^2/2} \right\}
=
C^{d-1} \eee^{(\lambda_2+\ldots + \lambda_d)y^2/2}.
\end{equation}
With~\eqref{eq:mills_ratio_ineq_appl}, we observe that for all $y\geq 1$ and $\lambda_1\leq -2\lambda_0-\lambda_2 - \ldots - \lambda_d$, the function in (a) is bounded from above by
$$
y\mapsto \const \cdot \frac{1}{\sqrt{-\lambda_1} y} \eee^{(\lambda_0+\lambda_1+\ldots + \lambda_d)y^2/2}
\leq
\const \cdot \eee^{-\lambda_0 y^2/2}.
$$
This upper bound is integrable over $[1, \infty)$. Combining the two intervals concludes the proof of (a).

\vspace*{2mm}
\noindent
\emph{Proof of (b).} This proof is similar to that of (a). Consider the interval $(0,1]$. By~\eqref{eq:mills_ratio_ineq_appl}, the function in (b) is bounded from above by
$$
y\mapsto  \eee^{-\lambda_0y^2/2} \frac{2}{\sqrt{2\pi}y} \Bigg| \Im \left( \prod_{j=2}^d \Phi\left( i\eps_1\eps_j \sqrt{\lambda_j}y \right) \right)\Bigg|,
$$
using that $\sqrt{-\lambda_0-\lambda_1}/\sqrt{-\lambda_1}<2$ provided $\lambda_1<0$ is sufficiently negative.
It follows from~\eqref{eq:pf_lemma_integrability_bd_for_im_1} that the upper bound has a finite limit as $y\downarrow 0$ and, hence, is integrable on $(0,1]$.

It remains to find an integrable majorant on the interval $[1, \infty)$. With~\eqref{eq:mills_ratio_ineq_appl} and~\eqref{eq:pf_lemma_integrability_bd_for_im_2} we obtain: For all $\lambda_1<-\lambda_0$ and $y\geq 1$, the function in (b) is bounded from above by
$$
\const \cdot \sqrt{-\lambda_0-\lambda_1} \eee^{(-\lambda_0-\lambda_1)y^2/2} \frac{\eee^{\lambda_1 y^2/2}}{\sqrt{2\pi} \sqrt{-\lambda_1}y} \eee^{(\lambda_2+\ldots + \lambda_d)y^2/2}
\leq
\const \cdot \eee^{(-\lambda_0+\lambda_2+\ldots + \lambda_d)y^2/2}.
$$
This is integrable over $[1,\infty)$, which concludes the proof.
\end{proof}

\begin{proof}[Proof of Proposition~\ref{prop:g_d_limit_lam1}]
For $d=1$, we have $\bg_1(\lambda_0;\lambda_1;\eps_1) = \P[\eps_1 \eta_1\leq 0] = 1/2$, where $\eta_1$ is some zero-mean Gaussian distributed random variable. Hence, the equation is satisfied.

Now, let $d\geq 2$. For $\lambda_1$ sufficiently negative, Theorem~\ref{theo:formula_for_bg_d(lambda_i,eps_i)_negative} yields
\begin{align*}
\bg_d(\lambda_0;\lambda_1, \ldots, \lambda_d; \eps_1, \ldots, \eps_d)
=
\sqrt{\frac{\lambda_0}{2\pi}} \Biggl(&
\int_0^{\infty} \eee^{-\frac{\lambda_0 y^2}{2}} \prod_{j=2}^d \Phi \left(\eps_1 \eps_j \sqrt{\lambda_j} y \right) \dd y\\
&+2 \int_0^{\infty} \eee^{\frac{\lambda_0 y^2}{2}} \Phi\left(-\sqrt{-\lambda_1} y \right) \Im \left( \prod_{j=2}^d \Phi \left(i \eps_1 \eps_j \sqrt{\lambda_j} y \right) \right) \dd y
\Biggr).
\end{align*}
The first integral does not depend in $\lambda_1$. In the second integral we have $\Phi\left(-\sqrt{-\lambda_1} y \right) \rightarrow 0$ as $\lambda_1 \rightarrow -\infty$, for all $y>0$, and Lemma~\ref{lemma:integrable_function_for_limit} (a) entails that dominated convergence theorem is applicable. Hence, the second integral converges to $0$ as $\lambda_1 \rightarrow -\infty$. Substituting $x=\sqrt{\lambda_0} y$ in the first integral completes the proof.
\end{proof}

\begin{proof}[Proof of Proposition~\ref{prop:g_d(-lam0-lam1)_limit_lam1}]
For $d=1$, we have $\bg_1(\lambda_0;\lambda_1;\eps_1) = \P[\eps_1 \eta_1\leq 0] = 1/2$, and the equation holds true.
Now, let $d\geq 2$. For $\lambda_1$ sufficiently negative, Theorem~\ref{theo:formula_for_bg_d(lambda_i,eps_i)_negative} gives $\bg_d(\lambda_0; \lambda_1, \ldots, \lambda_d; \eps_1, \ldots, \eps_d) = S_1(\lambda_1) + S_2(\lambda_1)$, where
\begin{align*}
S_1(\lambda_1)
=&
\sqrt{\frac{-\lambda_0-\lambda_1}{2\pi}} \int_0^{\infty} \eee^{-\frac{(-\lambda_0-\lambda_1) y^2}{2}} \prod_{j=2}^d \Phi \left(\eps_1 \eps_j \sqrt{\lambda_j} y \right) \dd y
\quad \text{ and } \quad\\
S_2(\lambda_1)
=&
2\sqrt{\frac{-\lambda_0-\lambda_1}{2\pi}} \int_0^{\infty} \eee^{\frac{(-\lambda_0-\lambda_1) y^2}{2}} \Phi\left(-\sqrt{-\lambda_1} y \right) \Im \left( \prod_{j=2}^d \Phi \left(i \eps_1 \eps_j \sqrt{\lambda_j} y \right) \right) \dd y.
\end{align*}
For $S_1$, substitute $x=\sqrt{-\lambda_0-\lambda_1}y$ and use the dominated convergence theorem:
\begin{align*}
\lim_{\lambda_1 \rightarrow -\infty} S_1(\lambda_1)
=&
\lim_{\lambda_1 \rightarrow -\infty} \frac{1}{\sqrt{2\pi}} \int_0^{\infty} \eee^{-x^2/2} \prod_{j=2}^d \Phi \left( \eps_1 \eps_j \sqrt{\frac{\lambda_j}{-\lambda_0-\lambda_1}}x \right) \dd x\\
=&
\frac{1}{\sqrt{2\pi}} \int_0^{\infty} \eee^{-x^2/2} \prod_{j=2}^d \Phi \left( 0 \right) \dd x
=
\frac{1}{\sqrt{2\pi}} \frac{1}{2^{d-1}} \frac{\sqrt{2\pi}}{2}
=
\frac{1}{2^d}.
\end{align*}
For $S_2$, the Mills ratio asymptotics of $\Phi$, see, e.g.\ \cite[Lemma 3.10]{kabluchkoklimovsky}, gives that for all $y>0$,
$$
\Phi(-\sqrt{-\lambda_1}y)
=
\frac{1+o(1)}{\sqrt{2\pi} \sqrt{-\lambda_1}y} \eee^{\lambda_1 y^2/2}, \quad \lambda_1 \rightarrow -\infty.
$$
For all $y>0$ this entails
\begin{align*}
&2\sqrt{\frac{-\lambda_0-\lambda_1}{2\pi}} \eee^{\frac{(-\lambda_0-\lambda_1) y^2}{2}} \Phi\left(-\sqrt{-\lambda_1} y \right) \Im \left( \prod_{j=2}^d \Phi \left(i \eps_1 \eps_j \sqrt{\lambda_j} y \right) \right)\\
=
&\frac{1+o(1)}{\pi y} \eee^{-\lambda_0 y^2/2} \Im \left( \prod_{j=2}^d \Phi \left(i \eps_1 \eps_j \sqrt{\lambda_j} y \right) \right)
\xrightarrow[\lambda_1 \rightarrow -\infty]{}
\frac{1}{\pi y} \eee^{-\lambda_0 y^2/2} \Im \left( \prod_{j=2}^d \Phi \left(i \eps_1 \eps_j \sqrt{\lambda_j} y \right) \right).
\end{align*}
By  Lemma~\ref{lemma:integrable_function_for_limit} (b), the dominated convergence theorem is applicable and it follows that
$$
\lim_{\lambda_1 \rightarrow -\infty} S_2(\lambda_1)
=
\int_0^{\infty} \frac{1}{\pi y} \eee^{-\lambda_0 y^2/2} \Im \left( \prod_{j=2}^d \Phi \left(i \eps_1 \eps_j \sqrt{\lambda_j} y \right) \right) \dd y.
$$
The desired limit is given by the sum of the limits of $S_1(\lambda_1)$ and $S_2(\lambda_1)$.
\end{proof}

\subsection{Angles of rectangular simplices}
With the limits for $\bg_d$ established, we can now derive formulas for the internal and external angles of rectangular orthocentric simplices.
The main idea is to consider the simplex $[0, \frac{e_1}{\tau_1}, \ldots, \frac{e_d}{\tau_d}]$ as a limit of the simplex $[\frac{\tau_0 e_0 +\ldots + \tau_d e_d}{\tau_0^2+\ldots + \tau_d^2}, \frac{e_1}{\tau_1}, \ldots, \frac{e_d}{\tau_d}]$ as $\tau_0 \rightarrow +\infty$.

\begin{theorem}[Angles of rectangular simplices]
For $d\geq 2$ and $\tau_1, \ldots, \tau_d>0$ consider the simplex $S=[0, \frac{e_1}{\tau_1}, \ldots, \frac{e_d}{\tau_d}]$.
\begin{enumerate}
\item For $k\in \{0,\ldots, d\}$ let $F=[0, \frac{e_1}{\tau_1}, \ldots, \frac{e_k}{\tau_k}]$. The internal and external angles of $S$ at $F$ are given by $\beta(F,S) = 1/2^{d-k} = \gamma(F,S)$.
\item For $k\in \{0, \ldots, d-1\}$ let $F=[\frac{e_1}{\tau_1}, \ldots, \frac{e_{k+1}}{\tau_{k+1}}]$. The internal angle of $S$ at $F$ is given by
\begin{align}
\beta(F,S)
&=
\lim_{\lambda_1 \rightarrow \pm \infty}\bg_{d-k} \left( -(\tau_1^2 + \ldots + \tau_d^2)-\lambda_1; \lambda_1, \tau_{k+2}^2, \ldots, \tau_d^2; \sgn(\lambda_1),1,\ldots, 1 \right) \label{eq:rectangular_simplex_internal_angle_limit}\\
&=
\frac{1}{2^{d-k}} + \frac{1}{\pi} \int_0^{\infty} \frac{\eee^{-(\tau_1^2+\ldots + \tau_d^2)y^2/2}}{y} \Im\left( \prod_{j=k+2}^d \Phi\left(-i \tau_j y\right) \right) \dd y. \label{eq:rectangular_simplex_internal_angle_explicit_formula}
\end{align}
The external angle of $S$ at $F$ is given by
\begin{align}
\gamma(F,S)
&=
\lim_{\lambda_1 \rightarrow \pm \infty} \bg_{d-k}\left( \tau_1^2 +\ldots +\tau_{k+1}^2; \lambda_1, \tau_{k+2}^2, \ldots, \tau_d^2; 1, \ldots, 1 \right) \label{eq:rectangular_simplex_external_angle_limit}\\
&=
\frac{1}{\sqrt{2\pi}} \int_0^{\infty} \prod_{j=k+2}^d \Phi \left( \frac{\tau_j}{\sqrt{\tau_1^2 + \ldots + \tau_{k+1}^2}} \right) \eee^{-x^2/2}  \dd x. \label{eq:rectangular_simplex_external_angle_explicit_formula}
\end{align}
\end{enumerate}
\end{theorem}

\begin{proof}
\emph{Proof of (a).} The case $k=d$ follows from Proposition~\ref{prop:degenerate_simpl_cones} (a), as both sides are equal to $1$. The remaining cases, $k\in \{0,\ldots, d-1\}$, can be derived from Proposition~\ref{prop:degenerate_simpl_cones} (b).

\vspace*{2mm}
\noindent
\emph{Proof of (b).} In the case $k=d-1$, both sides are equal to $1/2$. Now, consider $k\in \{0, \ldots, d-2\}$. For all $\tau_0>0$, denote $S_{\tau_0}=[\frac{\tau_0 e_0 +\ldots + \tau_d e_d}{\tau_0^2+\ldots + \tau_d^2}, \frac{e_1}{\tau_1}, \ldots, \frac{e_d}{\tau_d}]$. We claim
$$
\beta(F,S) = \lim_{\tau_0 \rightarrow +\infty} \beta(F,S_{\tau_0})
\quad \text{ and } \quad
\gamma(F,S) = \lim_{\tau_0 \rightarrow +\infty} \gamma(F,S_{\tau_0}).
$$
Indeed, this can be seen as follows. By Theorem~\ref{theo:cones_angles_orthoc_simplices_outside} (b), for all $\tau_0>0$, the tangent cone $T(F,S_{\tau_0})$ is the direct orthogonal sum of the lineality space of $F$ (which does not influence the angle) and a pointed cone
$$
C_{d-k}\left( \tau_1^2 + \ldots + \tau_{k+1}^2; -\tau_0^2 - \ldots - \tau_d^2, \tau_{k+2}^2, \ldots, \tau_d^2; 1, \ldots, 1 \right)
=
\pos(v_{\tau_0,{k+1}}, \ldots, v_{\tau_0,d}),
$$
for suitable  vectors $v_{\tau_0,{k+1}}, \ldots, v_{\tau_0,d}$.

Moreover, by Proposition~\ref{prop:degenerate_simpl_cones} (c),
the tangent cone of $S$ at $F$ is the direct orthogonal sum of the lineality space and a cone $\pos(v_{k+1}, \ldots, v_d)$ which has Gram matrix~\eqref{eq:gram_matrix_tangent_cone_degenerate}. A computation shows that
$$
\lim_{\tau_0 \rightarrow +\infty} \frac{\lan v_{\tau_0,i}, v_{\tau_0,j} \ran}{\|v_{\tau_0,i} \| \|v_{\tau_0,j}\|}
=
\frac{\lan v_i, v_j \ran}{\|v_i\| \|v_j\|}
$$
for all $k+1\leq i\neq j \leq d$. Moreover, the Gram matrix of $v_{k+1}, \ldots, v_d$ is non-singular by~\eqref{eq:inverse_rectangular_cones}, so the vectors are linearly independent. Consequently,~\cite[Corollary~5.2]{kabluchko_zaporozhets_gauss_simplex} entails
\begin{align*}
\lim_{\tau_0 \rightarrow +\infty} \beta(F,S_{\tau_0})
&=
\lim_{\tau_0 \rightarrow +\infty} \alpha\left( T(F,S_{\tau_0}) \right)
=
\lim_{\tau_0 \rightarrow +\infty} \alpha \left( \pos\left(v_{\tau_0,k+1}, \ldots, v_{\tau_0,d} \right) \right)
=
\alpha \left( \pos\left(v_{k+1}, \ldots, v_d \right) \right)\\
&=
\alpha\left( T(F,S) \right)
=
\beta(F,S).
\end{align*}
Inverting the Gram matrices also   implies $\lim_{\tau_0 \rightarrow \infty} \gamma(F,S_{\tau_0}) = \gamma(F,S)$.

Taking the angles of $S_{\tau_0}$ from Theorem~\ref{theo:cones_angles_orthoc_simplices_outside} (d), we obtain
\begin{align*}
\beta(F,S)
=&
\lim_{\tau_0 \rightarrow +\infty} \bg_{d-k} \left( \tau_0^2; -\tau_0^2-\tau_1^2 - \ldots - \tau_d^2,\tau_{k+1}^2, \ldots, \tau_d^2; -1,1,\ldots, 1 \right)\\
=&
\lim_{\lambda_1 \rightarrow -\infty}\bg_{d-k} \left( -(\tau_1^2 + \ldots + \tau_d^2)-\lambda_1; \lambda_1, \tau_{k+2}^2, \ldots, \tau_d^2; -1,1,\ldots, 1 \right)\\
=&
\frac{1}{2^{d-k}} + \frac{1}{\pi} \int_0^{\infty} \frac{\eee^{-(\tau_1^2+\ldots + \tau_d^2)y^2/2}}{y} \Im\left( \prod_{j=k+2}^d \Phi\left(-i \tau_j y\right) \right) \dd y.
\end{align*}
Here, for the second equation we set $\lambda_1 :=-\tau_0^2-\tau_1^2-\ldots - \tau_d^2 \to -\infty$, and the third equation follows with Proposition~\ref{prop:g_d(-lam0-lam1)_limit_lam1}. Similarly, Theorem~\ref{theo:cones_angles_orthoc_simplices_outside} (d) and Proposition~\ref{prop:g_d_limit_lam1} entail
\begin{align*}
\gamma(F,S)
=&
\lim_{\tau_0 \to +\infty} \bg_{d-k}\left( \tau_1^2 +\ldots +\tau_{k+1}^2; -\tau_0^2-\tau_1^2 -\ldots - \tau_d^2, \tau_{k+2}^2, \ldots, \tau_d^2; 1, \ldots, 1 \right)\\
=&
\lim_{\lambda_1 \rightarrow -\infty} \bg_{d-k}\left( \tau_1^2 +\ldots +\tau_{k+1}^2; \lambda_1, \tau_{k+2}^2, \ldots, \tau_d^2; 1, \ldots, 1 \right)\\
=&
\frac{1}{\sqrt{2\pi}} \int_0^{\infty} \prod_{j=k+2}^d \Phi\left( \frac{\tau_j}{\sqrt{\tau_1^2 + \ldots + \tau_{k+1}^2}}x \right) \eee^{-x^2/2} \dd x.
\end{align*}

To complete the proof of Part~(b), it remains to justify that the limits in formulas~\eqref{eq:rectangular_simplex_internal_angle_limit} and~\eqref{eq:rectangular_simplex_external_angle_limit} for $\beta(F,S)$ and $\gamma(F,S)$ also hold as $\lambda_1 \to +\infty$ (in contrast to the previously considered case $\lambda_1 \to -\infty$). To this end, we regard the rectangular simplex $S = [0, \frac{e_1}{\tau_1}, \ldots, \frac{e_d}{\tau_d}]$ as the limit of the simplices $S'_{\tau_0} := [\frac{e_0}{\tau_0}, \frac{e_1}{\tau_1}, \ldots, \frac{e_d}{\tau_d}]$ as $\tau_0 \to +\infty$.

By Theorem~\ref{theo:orthocentr_simpl_relint_tangent_normal_angles}(a), for each $\tau_0 > 0$, the tangent cone $T(F, S'_{\tau_0})$ is the orthogonal direct sum of the lineality space of $F$ and a pointed cone
\[
C_{d-k}\left( \tau_1^2 + \ldots + \tau_{k+1}^2; \tau_0^2, \tau_{k+2}^2, \ldots, \tau_d^2; 1, \ldots, 1 \right).
\]
On the one hand, by Corollary~\ref{cor:intrinsic_volumes_orthocentr_simplices}, the corresponding internal and external angles are given by
\begin{align*}
\beta(F, S'_{\tau_0})
&=
\bg_{d-k}\left(-\tau_0^2 - \tau_1^2 - \ldots - \tau_d^2; \tau_0^2, \tau_{k+2}^2, \ldots, \tau_d^2; 1, \ldots, 1 \right), \\
\gamma(F, S'_{\tau_0})
&=
\bg_{d-k}\left(\tau_1^2 + \ldots + \tau_{k+1}^2; \tau_0^2, \tau_{k+2}^2, \ldots, \tau_d^2; 1, \ldots, 1 \right).
\end{align*}
On the other hand, using the same Gram matrix arguments as above, one verifies that $\beta(F, S'_{\tau_0}) \to \beta(F, S)$ and $\gamma(F, S'_{\tau_0}) \to \gamma(F, S)$ as $\tau_0 \to +\infty$. Substituting $\lambda_1 := \tau_0^2$ then establishes~\eqref{eq:rectangular_simplex_internal_angle_limit} and~\eqref{eq:rectangular_simplex_external_angle_limit} in the limit $\lambda_1 \to +\infty$. Observe that by Equation~\eqref{eq:bg_d(lambda_i,eps_i)_formula} of Theorem~\ref{theo:formula_for_bg_d(lambda_i,eps_i)},
$$
\beta(F, S'_{\tau_0})
=
\lim_{\tau_0 \to + \infty}\frac{1}{\sqrt{2 \pi}} \int_{-\infty}^{\infty} \prod_{j\in \{0,k+2,\ldots, d\}} \Phi \Bigg(\frac{ i\tau_j x}{\sqrt{\tau_0^2 + \tau_1^2 + \ldots + \tau_d^2}} \Bigg) \eee^{-x^2/2} \dd x.
$$
The above provides an indirect proof that the limit on right-hand side coincides with the expression in~\eqref{eq:rectangular_simplex_internal_angle_explicit_formula}. A direct verification of this identity is non-trivial.
\end{proof}


\section*{Acknowledgement}
Supported by the German Research Foundation under Germany’s Excellence Strategy EXC 2044 – 390685587, Mathematics
Münster: Dynamics - Geometry - Structure and by the DFG priority program SPP 2265 Random Geometric Systems.

\addcontentsline{toc}{section}{References}
\bibliography{angles_of_general_orthocentric_simplices_bib}
\bibliographystyle{plainnat}

\vspace{1cm}

\footnotesize

\textsc{Zakhar Kabluchko: Institut f\"ur Mathematische Stochastik,
Universit\"at M\"unster,
Orl\'eans-Ring 10,
48149 M\"unster, Germany}\\
\textit{E-mail}: \texttt{zakhar.kabluchko@uni-muenster.de}\\

\textsc{Philipp Schange: Institut f\"ur Mathematische Stochastik, Universit\"at M\"unster,
Orl\'eans-Ring 10,
48149 M\"unster, Germany}\\
\textit{E-mail}: \texttt{philipp.schange@uni-muenster.de}\\

\end{document}